\documentclass[a4paper,36pt,notitlepage]{article}
\usepackage{amsmath,amsthm,amssymb,mathrsfs}
\usepackage{enumerate}
\usepackage{graphicx}
\usepackage[top=3cm, bottom=3.5cm, left=2.5cm, right=2.5cm]{geometry}

\newtheorem{dfn}{Definition}[subsection]
\newtheorem{thm}[dfn]{Theorem}

\newtheorem{lem}[dfn]{Lemma}
\newtheorem{cor}[dfn]{Corollary}
\newtheorem{con}[dfn]{Conjecture}
\newtheorem{rem}[dfn]{Remark}
\newtheorem{prop}[dfn]{Proposition}\makeatletter

         \@addtoreset{equation}{section}
\makeatother

\begin{document}
\title{\bf {\large Growth exponent for loop-erased random walk in three dimensions}}
\author{Daisuke Shiraishi}
\date{Department of Mathematics\\ Kyoto University}

\maketitle
\begin{abstract}
Let $M_{n}$ be the number of steps of the loop-erasure of a simple random walk on $\mathbb{Z}^{3}$ run until its first exit from a ball of radius $n$. 
In the paper, we will show the existence of the growth exponent, i.e., we show that there exists $\beta > 0$ such that 
\begin{equation*}
\lim_{n \to \infty} \frac{ \log E (M_{n} ) }{ \log n } = \beta.
\end{equation*}
 \end{abstract}

\section{Introduction}
\subsection{Introduction}
Let $S$ be the simple random walk on $\mathbb{Z}^{d}$ started at the origin and let $\tau_{n}$ be its first exit from the ball of radius $n$ centered at the origin. How does the  random walk path $S[0, \tau_{n} ]$ look like? This question has fascinated probabilists and mathematical physicists for a long time, and it continues to be an unending source of challenging problems. 

Cut points are one of the most important objects to study the random walk path (\cite{BGL}, \cite{BGS}, \cite{BSZ}, \cite{JP}, \cite{Law b}, \cite{L}, \cite{Greg}, \cite{S poi}, \cite{S pt}). Here a time $k \in [0, \tau_{n}]$ is called a (local) cut time if $S[0, k] \cap S[k+1, \tau_{n} ] = \emptyset$ and $S(k)$ is a (local) cut point if $k$ is a cut time.  We call  random walk path between each consecutive cut points a \textit{piece} so that the random walk path consists of the disjoint union of several pieces. The number of cut points are studied in many papers (\cite{JP}, \cite{Law b}, \cite{Law2} \cite{L}). In \cite{Law b}, it is proved that the expected number of cut points is of order $n^{2}$ for $d \ge 5$. Since $\tau_{n}$ is also of order $n^{2}$, the set of cut times has a positive density in $[0, \tau_{n}]$ in higher dimensions. For $d=4$, it is shown in \cite{Law2} that the expected number of cut points is of order $n^{2} (\log n)^{-1/2}$. Finally, for $d=2, 3$, it is proved in \cite{L} that there exist $ \xi_{d} \  (d=2, 3)$ such that the expected number of cut points is comparable to $n^{2- \xi_{d}}$. The exponent $\xi_{d}$ is called the intersection exponent. For the value of $\xi_{2}$, Lawler, Schramm and Werner \cite{LSW} prove that $\xi_{2} = \frac{5}{4}$ by using the SLE techniques. Consequently, the expected number of cut times up to time $\tau_{n}$ grows like $n^{\frac{3}{4}}$ for $d=2$. The exact value of $\xi_{3}$ is not known. The best rigorous estimates for $\xi_{3}$ \cite{Greg, Law3} are $\frac{1}{2} < \xi_{3} < 1.$

In higher dimensions, $d \ge 5$, roughly we may think of $S[0, \tau_{n} ]$ as a union of $O( n^{2} )$-stationary and ergodic pieces (\cite{BSZ}, \cite{Cro}). In that case, length of each piece has a finite moment and a correlation of two pieces is negligible, which enables us to analyze the path in detail.  Borrowing a term from physics we might say that the upper critical dimension for cut points is $4$. In $4$ dimensions, a logarithmic correction is required in the analysis of pieces. Study of geometrical structure of pieces in 4 dimensions is done in \cite{S pt}.  Roughly speaking, it is proved that a piece has a ``long sparse loop" if the length of the piece is large (see \cite{S pt} for the details). 

In 2 and 3 dimensions, the situation is more complicated since a correlation of two pieces is not negligible and each piece has no common distribution. To deal with this problem, we reconsider the non-intersecting random walk in this paper. In \cite{S ejp}, in order to investigate the structure of the path around cut points, the following problem was considered: if we condition that $S[0,n] \cap S[n+1, 2n] = \emptyset$, then how does the path look like around $S(n)$? Let $S^{1},S^{2}$ be independent simple random walks started at the origin. Then, thanks to the translation invariance and the reversibility of the simple random walk, our problem may be reduced to clarify the structure of $S^{1},S^{2}$ around the origin when we condition that $S^{1}[0,n] \cap S^{2}[1,n] = \emptyset$. To tackle this problem, the non-intersecting two-sided random walk paths were constructed for $d=2,3$ in \cite{S ejp}, namely the following limit exists:
\begin{equation}\label{limit existance}
\lim _{n \rightarrow \infty } P ( \cdot \  | \ S^{1}[0, \tau^{1}_{n} ] \cap S^{2}[1, \tau^{2}_{n} ] = \emptyset )=: \overline{P} (\cdot ),
\end{equation}
where $\tau^{i}_{n}$ is the first time that $S^{i}$ exits from a ball of radius $n$ centered at the origin, see \eqref{limit} for the precise definition of $\overline{P}$. Let $\overline{S}^{1}, \overline{S}^{2}$ be the associated two-sided random walks whose probability law is $\overline{P}$ and we define $\overline{S} (n)$ by $\overline{S} (n) = \overline{S}^{2}(n)$ if $n \ge 0$ and $\overline{S} (n) = \overline{S}^{1}(-n)$ if $n < 0$. We call $\overline{S}$ a non-intersecting random walk (see Figure 1 for $\overline{S}$).

In \cite{S}, it is proved that $\overline{S}$ has infinitely many \textit{global} cut points almost surely. Here, $n \in \mathbb{Z}$ is called global cut time for $\overline{S}$ if the entire of the past path $\overline{S} ( -\infty , n]$ and the future path $\overline{S} [n+1, \infty)$ do not intersect. We call $\overline{S}(n)$ a global cut point if $n$ is a global cut time. In \cite{S}, it is shown that the number of global cut points of $\overline{S}$ lying in the ball of radius $n$ is of order $n^{2-\xi_{d}}$. Therefore, we see that the number of local cut points for $S$ and the number of global cut points for $\overline{S}$ lying in a ball of radius $n$ are of the same order of magnitude.

For $k \in \mathbb{Z}$, we write $\overline{T}_{k}$ for the $k$-th global cut times with $\overline{T}_{0}=0$ and $\overline{T}_{k} < \overline{T}_{k+1}$ for each $k$. (Note that by definition, 0 is always a global cut time.) We call each $\overline{S}[ \overline{T}_{k} , \overline{T}_{k+1}]$ a piece again. In the present paper, we first show that each piece has common distribution and that $\overline{S} [0, \overline{T}_{1}]$ is asymptotically independent of $\overline{S} [\overline{T}_{k} , \overline{T}_{k+1}]$ as $|k| \to \infty$. More precisely, let $\overline{\theta}$ be a translation shift with respect to the first global cut point so that $\overline{S} \circ \overline{\theta} (m) = \overline{S} (m + \overline{T}_{1}) - \overline{S} ( \overline{T}_{1})$ for all $m$, see \eqref{shift} for $\overline{\theta}$.  We recall that a measure preserving system $(X, {\cal B}, \mu, T)$ is mixing if 
$\lim_{n \to \infty} \mu (A \cap T^{-n} B ) = \mu (A) \mu (B)$ for all $A, B \in {\cal B}$. The first our main result is the following.

\begin{thm}\label{ososugi}
Let $d=2, 3$. The law of $\overline{S}$ is invariant under the shift $\overline{\theta}$ and $\overline{\theta}$ is mixing.
\end{thm}

As an application of Theorem \ref{ososugi}, we investigate some quantities generated by the random walk path $\overline{S} [0,   \overline{T}_{n}]$. The quantities that we are interested in are 
\begin{itemize}
\item $q^{1}_{n} = \text{len} \big( LE (\overline{S} [0,   \overline{T}_{n}] ) \big)$, length (number of steps) of the loop-erasure of $\overline{S} [0,   \overline{T}_{n}]$,

\item $q^{2}_{n} = d_{\overline{S} [0, \overline{T}_{n} ]} (0, \overline{S} ( \overline{T}_{n} ) )$, graph distance between the origin and $\overline{S} ( \overline{T}_{n} ) $ on $\overline{S} [0,   \overline{T}_{n}]$,

\item $q^{3}_{n} = R_{\overline{S} [0, \overline{T}_{n} ]} (0, \overline{S} ( \overline{T}_{n} ) )$, effective resistance between the origin and $\overline{S} ( \overline{T}_{n} ) $ on $\overline{S} [0,   \overline{T}_{n}]$.
\end{itemize}
(See Section 4 for definitions and backgrounds of loop-erased random walk (LERW), graph distance and effective resistance.) These three quantities have the following similarities: for each $1 \le i \le 3$, $q^{i}_{n}$ can be written in terms of sum of compositions of $q^{i}_{1}$ and $\overline{\theta}^{k}$, i.e., we have 
\begin{equation}\label{ososugi-2}
q^{i}_{n} = \sum_{k=0}^{n-1} q^{i}_{1} \circ \overline{\theta}^{k}.
\end{equation}
Using this expression, we want to apply some results of ergodic theory to analyze $q^{i}_{n}$. If $q^{i}_{1}$ had a finite moment, we could apply Birkhoff's theorem to show that $q^{i}_{n}$ grows like $c n$. However this is not the case since $q^{i}_{1}$ has an infinite moment for all $i$. To deal with this issue, we use Aaronson's results derived in \cite{A}. In \cite{A}, it is shown that for all $a > 1$, either the ratio $\frac{q^{i}_{n}}{n^{a}}$ converges to 0 as $n \to \infty$ a.s. or $\limsup_{n \to \infty} \frac{q^{i}_{n}}{n^{a}} = \infty $ a.s. We are interested in the infimum of $a$ satisfying that $\frac{q^{i}_{n}}{n^{a}}$ converges to 0 a.s. and denote the infimum by $a^{i}_{d}$. Then we have

\begin{thm}\label{kanben}
Let $d=2, 3$. Suppose that $q^{i}_{n}$ and $a^{i}_{d}$ ($1 \le i \le 3$) are as above. Then for every $a > a^{i}_{d}$ the ratio $\frac{q^{i}_{n}}{n^{a}}$ converges to 0 as $n \to \infty$ almost surely. On the other hand, for all $a < a^{i}_{d}$, $\limsup_{n \to \infty} \frac{q^{i}_{n}}{n^{a}} = \infty $ almost surely.
\end{thm}

With Theorem \ref{kanben} in mind, it is natural to ask whether $q^{i}_{n}$ is logarithmically asymptotic to $n^{a^{i}_{d}}$. We are also interested in a comparison between these quantities for $\overline{S}$ and the corresponding quantities for $S$. Unfortunately, we could not give answers to such questions for graph distance and effective resistance. Actually, in an early stage of this project, we tried to find a way to prove $q^{i}_{n}$ is of order $n^{a^{i}_{d}}$ just by using general results of ergodic theory. However, since we could not find such a way, we decided to focus on the length of the loop-erasure of $\overline{S}$. (We should also mention that Theorem \ref{kanben} is the only place where we used a general result from ergodic theory.) For the length of the loop-erasure $q^{1}_{n}$ in 2 dimensions, we have the following theorem:

\begin{thm}\label{kanben-1}
 Let $d=2$. We let $\overline{\tau}^{+}_{n} = \inf \{j \ge 0 \ \big| \ |\overline{S} (j) | \ge n \}$ be the first time that $\overline{S} [0, \infty)$ exits from a ball of radius $n$. Then we have $a^{1}_{2} = \frac{5}{3}$, and 
 \begin{align}
&\lim_{n \to \infty} \frac{\log q^{1}_{n} }{\log n } = a^{1}_{2}, \text{ a.s.,} \label{kanben-1-1} \\
&\lim_{n \to \infty} \frac{\log \text{len} \big( LE (\overline{S} [0,   \overline{\tau}^{+}_{n}] ) \big) }{\log n } = \frac{5}{4}, \text{ a.s.} \label{kanben-1-3}
\end{align}
\end{thm}

Since the expected length of $LE ( S[0, \tau_{n} ] )$ is of order $n^{\frac{5}{4}}$ in 2 dimensions (see \cite{Ken}, \cite{Mas} and \cite{Lawler} for this), Theorem \ref{kanben-1} gives that the length of $LE (\overline{S} [0,   \overline{\tau}^{+}_{n}] )$ and $LE ( S[0, \tau_{n} ] )$ are of the same order of magnitude.

We want to establish same type of results as Theorem \ref{kanben-1} in 3 dimensions. To prove Theorem \ref{kanben-1}, it turns out that we need various results of loop-erased random walks in 2 dimensions (e.g. the expected length of $LE (S[0, \tau_{n} ] )$ is of order $n^{\frac{5}{4}}$, exponential tail bounds on the length of length of $LE (S[0, \tau_{n} ] )$, $\cdots$). Unfortunately, those necessary results have not been established up to now in 3 dimensions. In the present article, we will show the following theorem for loop-erased random walks in 3 dimensions, which will be used to prove that $q^{1}_{n}$ is of order $n^{a^{1}_{3}}$.

 \begin{thm}\label{okonomi} 
Let $d=3$. We write $M_{n} = \text{len} \big( LE ( S[0, \tau_{n} ] ) \big)$ for the length of the loop-erasure of $S[0, \tau_{n} ]$. Then there exists $\alpha \in [\frac{1}{3}, 1)$ such that 
\begin{equation}\label{okonomi-1}
\lim_{n \to \infty} \frac{ \log  E (M_{n} ) }{\log n } =2- \alpha.
\end{equation}
Furthermore, it follows that there exists $c > 0$ such that for all $n \ge 1$ and $\kappa \ge 1$
\begin{equation}\label{okonomi-2}
P \big( M_{n} \ge \kappa E (M_{n} ) \big) \le 2 e^{- c \kappa },
\end{equation}
and that for any $\epsilon \in (0,1)$, there exist $0 < c_{ \epsilon }, C_{ \epsilon } < \infty$ such that for all $\kappa \ge 1$ and $n  \ge 1$, 
\begin{align}\label{okonomi-3}
P \Big( M_{n} \le \frac{ E(M_{n} ) }{\kappa} \Big) \le C_{ \epsilon } \exp \big( - c_{ \epsilon } \kappa^{\frac{1}{2- \alpha} - \epsilon } \big).
\end{align}
\end{thm}

We should mention that once we show the existence of $\alpha$ as in \eqref{okonomi-1}, bounds $\alpha \in [\frac{1}{3}, 1)$ immediately follow from Lawler's estimates in \cite{Law5} where it is proved that $c n^{2- \xi_{3}} \le E (M_{n} ) \le C n^{\frac{5}{3}}$ (recall that $\frac{1}{2} < \xi_{3} < 1$). To our knowledge, the existence of the exponent $\alpha$ as in \eqref{okonomi-1} and exponential tail bounds on $M_{n}$ as in \eqref{okonomi-2} and \eqref{okonomi-3} are new results. In 2 dimensions, $E (M_{n} )$ is known to be of order $n^{\frac{5}{4}}$ (see \cite{Ken}, \cite{Mas} and \cite{Lawler}), and exponential tail bounds on $M_{n}$ are established in \cite{BM}. Theorem \ref{okonomi} is crucial and enough to derive an analog of Theorem \ref{kanben-1} in 3 dimensions as follows:

\begin{thm}\label{kanben-33}
 Let $d=3$. Recall that $\overline{\tau}^{+}_{n} = \inf \{j \ge 0 \ \big| \ |\overline{S} (j) | \ge n \}$ stands for the first time that $\overline{S} [0, \infty)$ exits from a ball of radius $n$, and that $\xi_{3}$ is the intersection exponent in 3 dimensions. Then we have $a^{1}_{3} = \frac{2- \alpha}{2- \xi_{3}}$, and 
 \begin{align}
&\lim_{n \to \infty} \frac{\log q^{1}_{n} }{\log n } = a^{1}_{3}, \text{ a.s.,} \label{kanben-33-1} \\
&\lim_{n \to \infty} \frac{\log \text{len} \big( LE (\overline{S} [0,   \overline{\tau}^{+}_{n}] ) \big) }{\log n } = 2- \alpha, \text{ a.s.} \label{kanben-33-3}
\end{align}
\end{thm}

\begin{rem}
Let $G_{n}$ and $R_{n}$ be the graph distance and effective resistance between the origin and $S( \tau_{n} )$ on the path $S[0, \tau_{n} ]$. To our knowledge, up to now it has not been proved or disproved that the exponents $\beta_{1}$ and $\beta_{2}$ with $E ( G_{n} ) = n^{\beta_{1} + o(1)}$ and $E ( R_{n} ) = n^{\beta_{2} + o(1)}$ exist in 2 and 3 dimensions. Furthermore, exponential tail bounds on $G_{n}$ and $R_{n}$ also have not been established.
\end{rem}

\subsection{Some words about the proofs}
In this subsection, we will explain ideas of main theorems. For Theorem \ref{ososugi}, the invariance of the law of $\overline{S}$ under the shift $\overline{\theta}$ is straightforward, but to show that $\overline{S}$ is mixing takes more work. Using the $\pi$-$\lambda$ Theorem (see \cite{Dur} Theorem A.1.4), it suffices to prove that the first piece $\overline{S} [0,   \overline{T}_{1}]$ and the $n$-th piece $\overline{S} [\overline{T}_{n-1},   \overline{T}_{n}] - \overline{S} (\overline{T}_{n-1})$ are almost independent if $n$ is large. Since the $n$-th piece typically lies in the outside of a large ball when $n$ is large, we need to control the independence of $\overline{S} [0, \overline{\tau}^{+}_{l}]$ and $\overline{S} [\overline{\tau}^{+}_{m}, \infty)$ with $l \ll m$ (see Theorem \ref{kanben-1} for $\overline{\tau}^{+}_{l}$). With this in mind, we take $N$ large and consider two pairs of paths $\overline{\gamma} = (\gamma^{1}, \gamma^{2} )$ and $\overline{\gamma}' = (\gamma^{3}, \gamma^{4} )$ such that for each $i =1, 2$,
\begin{equation}\label{misutta}
P \Big(  S^{2 i-1}[0, \tau^{2 i-1}_{N} ] \cap S^{2i}[1, \tau^{2 i}_{N} ] = \emptyset, \ \big ( S^{2 i-1}[0, \tau^{2 i-1}_{l} ], S^{2 i}[0, \tau^{2 i}_{l} ] \big) = (\gamma^{2 i-1}, \gamma^{2i} ) \Big) > 0,
\end{equation}
where $S^{1}, \cdots , S^{4}$ are independent simple random walks started at the origin and $\tau^{j}_{r}$ stands for the first time that $S^{j}$ exits from a ball of radius $r$. Namely, $\overline{\gamma}$ and $\overline{\gamma}'$ are possible configurations of $\overline{S}$ up to its first exit of the ball of radius $l$ (we will call such $\overline{\gamma}$ an initial configuration). We write $A^{N}_{i, l}$ for the event in the probability of \eqref{misutta}. In order to deal with the independence of $\overline{S} [0, \overline{\tau}^{+}_{l}]$ and $\overline{S} [\overline{\tau}^{+}_{m}, \infty)$, we will show that the distribution of $\big ( S^{1}[\tau^{1}_{m}, \tau^{1}_{N} ], S^{2 }[\tau^{2 }_{m}, \tau^{2 }_{N} ] \big)$ conditioned on $A^{N}_{1, l}$ is almost same as the distribution of $\big ( S^{3}[\tau^{3}_{m}, \tau^{3}_{N} ], S^{4 }[\tau^{4 }_{m}, \tau^{4 }_{N} ] \big)$ conditioned on $A^{N}_{2, l}$ if $l \ll m$ (see Theorem \ref{coupling3} for the details). This implies that $\overline{S} [\overline{\tau}^{+}_{m}, \infty)$ is almost independent of its initial configuration and we can conclude that $\overline{\theta}$ is mixing.

Once we establish Theorem \ref{ososugi}, Theorem \ref{kanben} immediately follows from Theorem A' of \cite{A}. 

We next consider Theorem \ref{kanben-1}. Since the number of global cut points of $\overline{S}$ lying in a ball of radius $n$ is typically of order $n^{2-\xi_{2}}$ in 2 dimensions (see Theorem 1.1 of \cite{S}) and $\xi_{2} = \frac{5}{4}$ (see \cite{LSW}), we see that the distance between the origin and $\overline{S} ( \overline{T}_{n} )$ is roughly of order $n^{\frac{4}{3}}$. Indeed we will see that for any $\epsilon > 0$, with high probability $\overline{T}_{n}$ is bounded above by $ \overline{\tau}^{+}_{n^{\frac{4}{3} + \epsilon }}$. This implies that the length of $LE (\overline{S} [0,   \overline{T}_{n}] )$ is bounded above by the length of $LE (\overline{S} [0,   \overline{\tau}^{+}_{n^{\frac{4}{3} + \epsilon }}] )$. However tail bounds on $M_{k}$ derived in \cite{BM} shows that the probability that $M_{k} \ge k^{\frac{5}{4} + \epsilon}$ is less than $C e^{- c k^{\epsilon} }$, where $\frac{5}{4}$ comes from the fact that the growth exponent for loop-erased random walk in 2 dimensions is equal to $\frac{5}{4}$ (see \cite{Ken}, \cite{Mas} and \cite{Lawler}). Therefore, the probability that $M_{n^{\frac{4}{3} + \epsilon }} \ge n^{\frac{5}{3} + 2 \epsilon}$ is exponentially small in $n$, which is much smaller than the probability that $S^{1}$ and $S^{2}$ do not intersect up to the first time that they exit from a ball of radius $n^{\frac{4}{3} + \epsilon }$ (such a non-intersecting probability is a polynomial order, see \eqref{intersection-exp}). Consequently, we see that the length of $LE (\overline{S} [0,   \overline{T}_{n}] )$ is bounded above by $n^{\frac{5}{3} + 2 \epsilon}$ with high probability. Similar considerations along with lower tail bounds on $M_{k}$ derived in \cite{BM} give the opposite inequality and we get Theorem \ref{kanben-1}.

We want to prove Theorem \ref{kanben-33} by the same strategies as Theorem \ref{kanben-1}. However, the following is missing in 3 dimensions:
\begin{itemize}
\item[(i)] Existence of the exponent $\beta$ such that $E (M_{n} ) = n^{\beta + o(1)}$ as $n \to \infty$.

\item[(ii)] Exponential tail bounds on $M_{n}$.
\end{itemize}
Once we deal with these two issues, Theorem \ref{kanben-33} follows from the same arguments as in the proof of Theorem \ref{kanben-1} explained as above. 

For the first issue (i), the crucial object is so called an escape probability that we will explain from now. We are interested in the probability that a simple random walk started at the origin and the loop-erasure of an independent simple random walk started at the origin do not intersect up to the first time they exit from a ball of radius $n$. We denote the probability by $Es (n)$ (see Section 6.2 for the precise definition of $Es (n)$). We write $B (n)$ for the ball of radius $n$ centered at the origin. Suppose that a point $x$ with $\frac{n}{3} \le |x| \le  \frac{2n}{3} $ lies in $LE ( S[0, \tau_{n}] )$. Then the definition of the loop-erasure (see Section 4) gives that the following holds:
\begin{itemize}
\item $S$ hits $x$ up to $\tau_{n}$.

\item The loop-erasure of the random walk $S$ from the origin to $x$ and $S$ from $x$ to the boundary of $B (n)$ do not intersect.
\end{itemize}
Reversing a path, the probability of this event is equal to the probability that a simple random walk from $x$ up to the boundary of $B (n)$ and the loop-erasure of an independent random walk from $x$ to the origin do not intersect. It turns out that this probability is comparable to $\frac{Es (n)}{n}$, which enables to conclude that $E (M_{n} )$ is comparable to $n^{2} Es (n)$. Therefore the issue (i) is reduced to proving that there exists $\alpha$ such that $Es (n) = n^{- \alpha + o(1)}$ as $n \to \infty$.

In order to show the existence of the exponent $\alpha$, we will give various relations between escape probabilities on various scales (see Propositions \ref{up-to-const lew}, \ref{up to const indep1} and \ref{up to const indep2}). In particular, it will be shown in Proposition \ref{step1} that $Es (2^{m+n} )$ is comparable to the product of $Es (2^{n})$ and the probability that a random walk from $(-2^{n}, 0, 0)$ to the boundary of $B (2^{m+n})$ and the loop-erasure of an independent random walk from $(2^{n}, 0, 0)$ to the boundary of $B (2^{m+n})$ do not intersect (we denote this probability by $a_{m, n}$). Since we know the existence of the scaling limit of LERW in 3 dimensions (see \cite{Koz}), it is natural to predict that the limit of $a_{m, n}$ as $n \to \infty$ exists for each fixed $m$. In fact, it will be shown in Proposition \ref{main lem} that the limit of $a_{m, n}$ exists with the help of some results derived in \cite{Koz}. Using the existence of $\lim_{n \to \infty} a_{m, n}$ and Proposition \ref{step1}, a standard subadditive argument shows that there exists $\alpha$ such that $Es (n) = n^{- \alpha + o(1)}$ (see Theorem \ref{main result-kaetta}).

Estimates on escape probabilities established in Section 6 and the existence of $\alpha$ as in Theorem \ref{main result-kaetta} are enough to get exponential tail bounds on $M_{n}$ by imitating proofs in \cite{BM} (see Section 8 for tail bounds on $M_{n}$).

\begin{figure}
\begin{center}

\includegraphics[width=13cm]{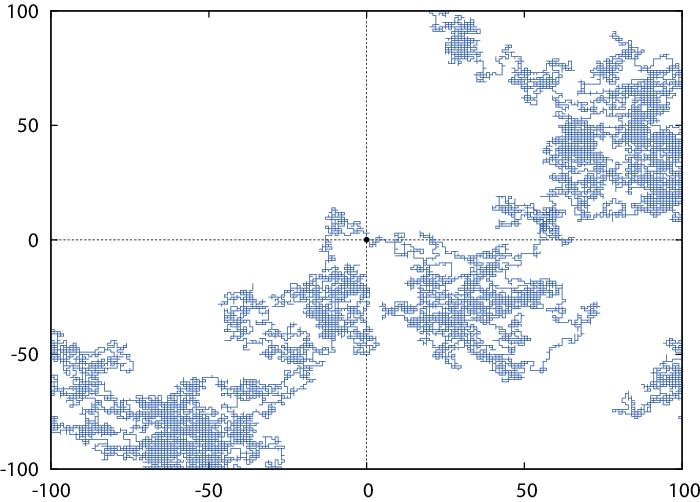}

\end{center}

\caption{A non-intersecting random walk trace $\overline{S}$ for $d=2$. }

\end{figure}

\subsection{Structure of the paper}
In the next subsection, we will collect notation and definitions which will be used throughout the paper.

In Section 2, we will prove the first claim of Theorem \ref{ososugi}, i.e., we will show that the law of $\overline{S}$ is invariant under the shift $\overline{\theta}$ in Theorem \ref{stationary-o}.

Section 3 will be devoted to prove the second claim of Theorem \ref{ososugi}. We will show that $\overline{\theta}$ is mixing in Theorem \ref{mixing}.

As an application of Theorem \ref{ososugi}, in Section 4 we will consider asymptotic behaviors of three quantities, the length of the loop-erasure, graph distance and effective resistance of $\overline{S} [0,   \overline{T}_{n}] $ along with Aaronson's results in \cite{A}. We will show Theorem \ref{kanben} in Theorem \ref{critical exp} after giving some backgrounds of these three quantities in Section 4.1.

We will prove Theorem \ref{kanben-1} in Section 5 by establishing Proposition \ref{upper-2dim-loop} and Proposition \ref{lower-2dim-loop}.

From Section 6 to Section 9, we will focus on LERW in 3 dimensions. In particular, from Section 6 to Section 8, we will focus on the loop-erasure of usual simple random walks (not the loop-erasure of $\overline{S}$). In Section 6, we will give various relations between escape probabilities on various scales. In order to give such relations of escape probabilities, the separation lemma (see Theorem \ref{sep lem 3dim}) is an important tool. Theorem \ref{sep lem 3dim} roughly states that a random walk and an independent LERW that are conditioned not to intersect are likely to be not very close at their endpoints, which allows to derive relations between escape probabilities on various scales, see Propositions \ref{up-to-const lew}, \ref{up to const indep1}, \ref{up to const indep2}) and \ref{step1}.

Using those results of the escape probabilities obtained in Section 6, we will prove the existence of the exponent $\alpha$ such that $Es (n)$ is of order $n^{- \alpha}$ in Theorem \ref{main result-kaetta}.

Section 8 will be devoted to establish exponential tail bounds on $M_{n}$ in three dimensions. We will prove \eqref{okonomi-2} and \eqref{okonomi-3} in Theorem \ref{dounaruno} and Theorem \ref{yattoda}. It is also proved that $E ( M_{n} )$ is comparable to $n^{2} Es (n) $ in Theorem \ref{leipzig} and Proposition \ref{zhan}. Combining this with Theorem \ref{main result-kaetta}, we get \eqref{okonomi-1}.

Using results obtained in Section 6 -- Section 8, we will prove Theorem \ref{kanben-33} in Section 9 by giving Proposition \ref{upper-3dim-loop} and Proposition \ref{lower-3dim-loop}.

In Section 10, we will summarize our results and discuss some future works.

\subsection{Notation}
In this subsection, we will collect some notation and definitions which will be used in the present article many times. 

Take a sequence of points $\lambda = [\lambda (0), \lambda(1), \cdots, \lambda(m) ]$ in $\mathbb{Z}^{d}$. We call $\lambda$ a path of length $m$ if $| \lambda(j) - \lambda(j+1)| = 1$ for all $j$. Here $| \cdot |$ stands for the Euclid distance in $\mathbb{R}^{d}$. We write $\text{len} \lambda$ for the length of $\lambda$. For two pats $\lambda = [\lambda (0), \lambda(1), \cdots, \lambda(m) ]$ and $\gamma = [\gamma (0), \gamma(1), \cdots, \gamma (n) ]$ with $\lambda(m) = \gamma (0)$, we write $\lambda + \gamma = [\lambda (0), \lambda(1), \cdots, \lambda(m), \gamma(1), \cdots, \gamma (n) ]$. We set $\lambda^{R} = [\lambda (m), \lambda(m-1), \cdots, \lambda(0) ]$ for the time reversal of $\lambda$. We call $\lambda$ a simple path if $\lambda (i) \neq \lambda (j)$ for all $i \neq j$.

We let $B (x, n) = \{ y \in \mathbb{Z}^{d} \ | \ |y - x | \le n \}$ be the discrete ball of radius $n$ centered at $x$. We write $B (n )$ for $B (0, n)$ the ball of radius $n$ centered at the origin. 

For a set $A \subset \mathbb{Z}^{d}$, we let $\partial A = \{ x \notin A \ | \ \text{there exits } y \in A \text{ such that } |x- y| =1    \}$ be the outer boundary of $A$. We write $\partial_{i} A = \{ x \in A \ | \ \text{there exits } y \notin A \text{ such that } |x- y| =1    \}$ for the inner boundary of $A$.

For a subset $A \subset \mathbb{R}^{d}$, $r > 0$, and a point $x \in \mathbb{R}^{d}$, we write $x + A = \{ x + y \ | \ y \in A \}$ and $r A = \{ r y \ | \ y \in A \}$.

$S$, $S^{1}$, $S^{2}$, $S^{3}$ and $S^{4}$ stand for independent simple random walks in $\mathbb{Z}^{d}$. We write $P^{x}$ and $E^{x}$ for probability of $S$ and its expectation  assuming that $S(0) =x$. If $x=0$, we use $P$ and $E$ instead of $P^{0}$ and $E^{0}$. We sometimes consider a product probability of $S^{i}$ and $S^{j}$ with $i < j$ assuming that $S^{i} (0) = x$ and $S^{j} (0) = y$. We write $P^{x, y}$ and $E^{x, y}$ for the product probability and its expectation. Of course it depends on $i$ and $j$ and we should write $P^{x, y}_{i, j}$ instead of $P^{x, y}$ to emphasize that it stands for the product probability measure of $S^{i}$ and $S^{j}$. However, in order to avoid complication of notation, we will use $P^{x, y}$. For example, $P^{x, y} \big( S^{1} [0, n] \cap A \neq \emptyset, \  S^{3} [0, m] \cap A \neq \emptyset \big)$ stands for the probability that both $S^{1} [0, n]$ and $S^{3} [0, m]$ hit $A$ assuming that $S^{1} (0) = x$ and $ S^{3} (0) =y $. We also use $P$ and $E$ instead of $P^{0, 0}$ and $E^{0, 0}$ if $x=y=0$.

For a Markov chain $X$ and a set $A$, we let $\tau^{X}_{A} = \{ j \ge 0 \ | \ X (j) \notin A \}$ be the first time that $X$ exists from $A$. If $X$ is $S$, we use $\tau_{A}$ instead of $\tau^{S}_{A}$. Furthermore if $X = S^{i}$, we use $\tau^{i}_{A}$ instead of $\tau^{S^{i}}_{A}$. When $A = B ( n)$, we use $\tau^{X}_{n}$ instead of $\tau^{X}_{B (n)}$. We also use $\tau_{n}$ (resp. $\tau^{i}_{n}$) for the case that $A = B (n)$ and $X = S$ (resp. $X =S^{i}$). If $A = \{ x \}$, we write $\tau^{X}_{x}$ instead of $\tau^{X}_{ \{ x \} }$. Then $\tau_{x}$ and $\tau^{i}_{x}$ can be defined for the case that $X= S$ and $X = S^{i}$. We let $\sigma^{X}_{A} = \{ j \ge 1 \ | \ X (j) \in A \}$ be the first time that $X$ hits $A$. For the first hitting time, $\sigma_{A}$, $\sigma^{i}_{A}$, $\sigma^{X}_{n}$, $\sigma_{n}$, $\sigma^{i}_{n}$, $\sigma^{X}_{x}$, $\sigma_{x}$ and $\sigma^{i}_{x}$ can be defined similarly.

For a Markov chain $X$ and $x, y \in A \subset \mathbb{Z}^{d}$, we write 
\begin{equation*}
G^{X} (x, y, A) = E^{x}_{X} \Big( \sum_{ j=0}^{ \tau^{X}_{A} -1 } {\bf 1} \{ X (j) = y \} \Big)
\end{equation*}
for Green's function of $X$ in $A$, where $P^{x}_{X}$ and $E^{x}_{X}$ stands for the probability of $X$ and its expectation assuming that $X(0) =x$. If $X = S$, we use $G (x, y, A)$ instead of $G^{S} (x, y, A)$.

Let $\Lambda (n)$ be the set of paths satisfying that
\begin{align*}
&\gamma (0)=0, \gamma (j) \in  B(n) \text{ for all } j=0,1, \cdots , \text{len}\gamma -1 \\
&\gamma( \text{len} \gamma ) \in \partial B (n). 
\end{align*}
We write $\Lambda (\infty) = \{ \gamma \ | \ \gamma (0)=0,  \text{len} \gamma = \infty, \text{ and } \lim_{j \to \infty} |\gamma (j) | = \infty \}$ for a set of infinite paths. We next define a set of pairs of paths $\overline{\gamma}=(\gamma ^{1}, \gamma ^{2})$ satisfying a non-intersecting condition as follows. Let 
\begin{equation*}
\Gamma (n) = \{ \overline{\gamma}=(\gamma ^{1}, \gamma ^{2}) \in \Lambda (n) \times \Lambda (n) \ | \ \gamma^{1} [0, \text{len}\gamma^{1}] \cap \gamma^{2} [1, \text{len}\gamma^{2}] = \emptyset \}. 
\end{equation*}
We also write $\Gamma (\infty ) = \{ \overline{\gamma}=(\gamma ^{1}, \gamma ^{2}) \in \Lambda (\infty) \times \Lambda (\infty) \ | \ \gamma^{1} [0, \infty ) \cap \gamma^{2} [1, \infty ) = \emptyset \}$ for a set of pairs of infinite paths satisfying the non-intersecting condition. 

Let $S^{1}, S^{2}$ be independent simple random walks in $\mathbb{Z}^{d}$ started at the origin. We write
\begin{equation}\label{nonint}
\overline{A}_{n} = \big\{ \big( S^{1}[0, \tau^{1}_{n} ] , S^{2}[0, \tau^{2}_{n} ] \big) \in \Gamma (n)  \big\}
\end{equation}
for the event that $S^{1}$ and $S^{2}$ do not intersect up to the first time that they exit from $B (n)$. The intersection exponent $\xi_{d}$ ($d=2, 3$) is characterized by 
\begin{equation}\label{intersection-exp}
P ( \overline{A}_{n} ) \asymp n^{- \xi_{d}},
\end{equation}
see \cite{L} for the intersection exponent. For the value of $\xi_{2}$, it is proved in \cite{LSW} that $\xi_{2} = \frac{5}{4}$. The exact value of $\xi_{3}$ is not known. The best rigorous estimates for $\xi_{3}$ are $\frac{1}{2} < \xi_{3} < 1$, see \cite{Greg, Law3}. 

In \cite{S ejp}, it was proved that for each $L \in \mathbb{N}$ and a pair of paths $\overline{\gamma} = (\gamma ^{1}, \gamma ^{2}) \in \Gamma (L)$, the limit of the conditional probability
\begin{equation}\label{limit}
\lim _{n \rightarrow \infty } P \Big( \big( S^{1}[0, \tau^{1}_{L} ] , S^{2}[0, \tau^{2}_{L} ] \big) = \overline{\gamma} \ \big| \ \overline{A}_{n} \Big)
\end{equation}
exists. If we denote the value of \eqref{limit} by $\overline{P} ( \overline{\gamma} )$, then $\overline{P}$ extends uniquely to a probability measure on $\Gamma (\infty )$. We denote this probability space by $(\Omega, {\cal F}, \overline{P})$. Let $\overline{S}^{1}, \overline{S}^{2}$ be the associated two-sided random walks whose probability law is $\overline{P}$. 
We set $\overline{\tau}^{i}_{n} = \inf \{j \ge 0 \ | \ \overline{S}^{i} (j) \notin B (n) \}$ for the first time that $\overline{S}^{i}$ exits from $B (n)$. For $n \in \mathbb{Z}$, we write
\begin{eqnarray*}
\overline{S}(n) =\left\{ \begin{array}{ll}
\overline{S}^{2}(n) & (n \ge 0) \\
\overline{S}^{1}(-n) & (n < 0). 
\end{array} \right.
\end{eqnarray*} 
for the doubly infinite random walk. For $m \in \mathbb{Z}$, we write $\theta_{m}$ for the translation shift so that $\overline{S} \circ \theta_{m} (n) = \overline{S} (n+m) - \overline{S} (m)$ for each $n \in \mathbb{Z}$. In \cite{S}, \textit{global} cut points for $\overline{S}$ are studied. Here, $n \in \mathbb{Z}$ is called global cut time for $\overline{S}$ if the entire of the past part $\overline{S} ( -\infty , n]$ and the future part $\overline{S} [n+1, \infty)$ do not intersect. We call $\overline{S}(n)$ a global cut point if $n$ is a global cut time. We set $\overline{\tau}^{+}_{n} = \inf \{j \ge 0 \ | \ \overline{S} (j) \notin B (n) \}$ for the first time that $\overline{S} [0 , \infty )$ exits from $B (n)$. We also define $\overline{\tau}^{-}_{n} = \sup \{j \le 0 \ | \ \overline{S} (j) \notin B (n) \}$ for $\overline{S} (-\infty , 0 ]$. In \cite{S}, it is proved that the number of global cut times lying in $[0, \overline{\tau}^{+}_{n} ]$ is equal to $n^{2-\xi_{d} + o(1)}$ as $n \to \infty$ with probability one. Here $\xi_{d}$ stands for the intersection exponent for simple random walks in $d$ dimensions (see \cite{L} for the intersection exponent). This is true for the number of global cut times lying in $[-\overline{\tau}^{-}_{n}, 0 ]$. In particular, $\overline{S}$ has infinitely many global cut times both in $(-\infty , 0]$ and $[0, \infty)$ almost surely. Thus we may define the set of global cut times $\overline{T} = \{ \cdots , \overline{T}_{-2}, \overline{T}_{-1}, \overline{T}_{0}, \overline{T}_{1}, \overline{T}_{2}, \cdots \}$ with $\overline{T}_{0} = 0$ and $\overline{T}_{j} < \overline{T}_{j+1}$ for each $j$. We define the translation shift with respect to the first global cut point by  
\begin{equation}\label{shift}
\overline{\theta} := \theta_{\overline{T}_{1}}.
\end{equation}

Throughout the paper, we use $c, c^{\prime} , c_{1}, C, C^{\prime} , C_{1}, \dotsb$ to denote arbitrary positive constants which may change from line to line. If a constant is to depend on some other quantity, this will be made explicit. For example, if $c$ depends on $\epsilon$, we write $c_{\epsilon}$ (or $c( \epsilon )$). We write $a_{n} \asymp b_{n}$ if there exist constants $c_{1},c_{2}$ such that 
\begin{equation}\label{asymp}
c_{1}b_{n} \le a_{n} \le c_{2}b_{n}.
\end{equation}
We write $a_{n} \sim b_{n}$ if 
\begin{equation}\label{sim}
\lim_{n \to \infty} \frac{a_{n}}{b_{n}} = 1.
\end{equation}
Finally, we denote $a_{n} \approx b_{n}$ if 
\begin{equation}\label{approx}
\lim_{n \to \infty} \frac{\log a_{n}}{\log b_{n}} = 1.
\end{equation}
To avoid complication of notation, we don't use $\lfloor r \rfloor$ (the largest integer $\le r$) even though it is necessary to carry it.

\section{Invariance under the translation shift}
Recall that the conditioned random walk $\overline{S}$ has infinitely many global cut times both in positive and negative times almost surely. The shift $\overline{\theta}$ is the translation shift that translates the first global cut point to the origin, see \eqref{shift} for $\overline{\theta}$. 
In this section, we will prove that the law of $\overline{S}$ is invariant under the shift $\overline{\theta}$ in Theorem \ref{stationary-o} below. From this, we see that $\overline{S} [\overline{T}_{i}, \overline{T}_{i+1}] - \overline{S} ( \overline{T}_{i} )$ has same distribution as $\overline{S} [0, \overline{T}_{1}]$ for each $i \in \mathbb{Z}$. It turns out that the proof of Theorem \ref{stationary-o} follows from a standard application of the translation invariance of the usual simple random walk $S$.

\begin{thm}\label{stationary-o}
The law of $\overline{S}$ is invariant under the shift $\overline{\theta}$.
\end{thm}

\begin{proof}
In order to prove the theorem, by the $\pi$-$\lambda$ Theorem (see \cite{Dur} Theorem A.1.4), it suffices to show that
\begin{equation}\label{invariant}
\overline{P} \big( \overline{\theta}^{-1} A \big) = \overline{P} \big(  A \big),
\end{equation}
where $A$ is an event that
\begin{equation*}
A= \{ \overline{S} [0 , \overline{T}_{1}] = \lambda \},
\end{equation*}
with $\overline{P} (A) > 0$.

So fix a path $\lambda = [\lambda(0) , \lambda (1), \cdots , \lambda (l)] $ with length $l$ and assume that $\overline{P} (A) > 0$.
The definition of $\overline{\theta}$ immediately gives that

\begin{equation}\label{henkei}
\overline{P} \big( \overline{\theta}^{-1} A \big) =  \overline{P} \big(  \overline{S} [\overline{T}_{1} , \overline{T}_{2}] - \overline{S} (\overline{T}_{1}) = \lambda \big).
\end{equation}
We want to say that the right hand side of \eqref{henkei} is equal to $\overline{P} (A)$. To show it, we consider every path $\gamma$ such that the probability of the first piece $\overline{S} [0, \overline{T}_{1}]$ being $\gamma$ is positive, i.e., we define

\begin{equation}\label{possible}
\text{Bead} = \Big\{  \gamma \ \big| \  \overline{P} \big(  \overline{S} [0 , \overline{T}_{1}] = \gamma \big) >0  \Big\}.
\end{equation}
Then we have 
\begin{equation*}
\overline{P} \big(  \overline{S} [\overline{T}_{1} , \overline{T}_{2}] - \overline{S} (\overline{T}_{1}) = \lambda \big) = \sum_{\gamma \in \text{Bead} } \overline{P} \big( \overline{S} [0 , \overline{T}_{1}] = \gamma, \ \ \overline{S} [\overline{T}_{1} , \overline{T}_{2}] = \lambda + \gamma ( \text{len}\gamma )  \big).
\end{equation*}
Fix $\gamma \in \text{Bead}$ such that $\overline{P} \big( \overline{S} [0 , \overline{T}_{1}] = \gamma, \ \ \overline{S} [\overline{T}_{1} , \overline{T}_{2}] = \lambda + \gamma ( \text{len}\gamma )  \big) >0$. Let $\text{len}\gamma = k$. The definition of $\overline{P}$ (see \eqref{limit}) gives that
\begin{align*}
&\overline{P} \big( \overline{S} [0 , \overline{T}_{1}] = \gamma, \ \ \overline{S} [\overline{T}_{1} , \overline{T}_{2}] = \lambda + \gamma ( \text{len}\gamma )  \big) \\
&= \lim_{N \to \infty} P \Big( S^{2}[0,k] = \gamma , \ S^{2}[k, k+l] = \lambda + \gamma ( \text{len}\gamma ), \ F_{N} \ \big| \ \overline{A}_{N}  \Big),
\end{align*}
where $\overline{A}_{N}$ was defined as in \eqref{nonint} and $F_{N}$ is defined by
\begin{equation*}
F_{N}= \Big\{ S^{1} [0, \tau^{1}_{N}] \cap \Big( \gamma (0, k] \cup \big( \lambda + \gamma ( \text{len}\gamma ) \big) \cup S^{2} [k+l, \tau^{2}_{N}] \Big) = \emptyset \Big\}. 
\end{equation*}
Namely, $F_{N}$ is the event that time $k= \text{len}\gamma$ and $k+l$ are cut times for $S^{2}$ up to time $\tau^{2}_{N}$. Now we want to translate $\gamma (k)$ to the origin and to use the translation invariance for the usual simple random walk. With this in mind, we write $\gamma^{R} = [\gamma (k) , \gamma(k-1) , \cdots , \gamma(0)]$ for the time reverse of $\gamma$. Let $x = \gamma (k)$. Since $ B ( N- |x| ) \subset B ( N) - x \subset B (N + |x|)$, we can use the translation invariance to show that

\begin{align*}
&P \Big( S^{1}[0,k] = \gamma^{R} - x , \ S^{2}[0, l] = \lambda , \ F^{+}_{N}, \ \overline{A}_{N + |x|}  \Big) \\
&\le P \Big( S^{2}[0,k] = \gamma , \ S^{2}[k, k+l] = \lambda + \gamma ( \text{len}\gamma ), \ F_{N}, \ \overline{A}_{N}  \Big) \\
&\le P \Big( S^{1}[0,k] = \gamma^{R} - x , \ S^{2}[0, l] = \lambda , \ F^{-}_{N}, \ \overline{A}_{N - |x| }  \Big),
\end{align*}
where
\begin{align*}
&F^{+}_{N} = \Big\{ S^{1} [k, \tau^{1}_{N+|x|}] \cap \big( \gamma (0, k]-x \big) = \emptyset, \ S^{2} (l , \tau^{2}_{N+ |x|} ] \cap \lambda = \emptyset  \Big\} \\
&F^{-}_{N} = \Big\{ S^{1} [k, \tau^{1}_{N-|x|}] \cap \big( \gamma (0, k]-x \big) = \emptyset, \ S^{2} (l , \tau^{2}_{N- |x|} ] \cap \lambda = \emptyset  \Big\}.  
\end{align*}
Namely, $F^{+}_{N}$ is the event that time $k$ is a cut time for $S^{1}$ up to $\tau^{1}_{N+|x|}$ and $l$ is a cut time for $S^{2}$ up to $\tau^{2}_{N+ |x|}$. $F^{-}_{N}$ is the event obtained by replacing $\tau^{i}_{N+|x|}$ by $\tau^{i}_{N-|x|}$ in the definition of $F^{+}_{N}$.

By definition of $\overline{P}$, we have
\begin{align*}
&\lim_{N \to \infty} \frac{P \Big( S^{1}[0,k] = \gamma^{R} - x , \ S^{2}[0, l] = \lambda , \ F^{+}_{N}, \ \overline{A}_{N + |x|}  \Big)}{P (\overline{A}_{N + |x|})}  \\
&= \overline{P} \big( \overline{S} [\overline{T}_{-1} , 0] = \gamma - x, \ \ \overline{S} [0 , \overline{T}_{1}] = \lambda   \big) \\
&=\lim_{N \to \infty} \frac{P \Big( S^{1}[0,k] = \gamma^{R} - x , \ S^{2}[0, l] = \lambda , \ F^{-}_{N}, \ \overline{A}_{N - |x|}  \Big)}{P (\overline{A}_{N - |x|})}.
\end{align*}
However, by Corollary 4.2 in \cite{S ejp}, we have
\begin{equation*}
\lim_{N \to \infty} \frac{P (\overline{A}_{N \pm |x| })}{P ( \overline{A}_{N})} = 1,
\end{equation*}
which implies that 
\begin{align*}
&\overline{P} \big( \overline{S} [\overline{T}_{-1} , 0] = \gamma - x, \ \ \overline{S} [0 , \overline{T}_{1}] = \lambda   \big) \\
&=\overline{P} \big( \overline{S} [0 , \overline{T}_{1}] = \gamma, \ \ \overline{S} [\overline{T}_{1} , \overline{T}_{2}] = \lambda + \gamma ( \text{len}\gamma )  \big).
\end{align*}
By taking the sum for $\gamma \in \text{Bead}$ such that $\overline{P} \big( \overline{S} [0 , \overline{T}_{1}] = \gamma, \ \ \overline{S} [\overline{T}_{1} , \overline{T}_{2}] = \lambda + \gamma ( \text{len}\gamma )  \big) >0$, we have 
\begin{equation*}
\overline{P} \big( \overline{\theta}^{-1} A \big) = \overline{P} \big(  A \big),
\end{equation*}
and finish the proof.

\end{proof}

\section{Ergodicity w.r.t. the translation shift}
In this section, we prove the shift $\overline{\theta}$ is mixing in Theorem \ref{mixing} below. We will explain the sketch of the proof here. In order to prove Theorem \ref{mixing}, again by the $\pi$-$\lambda$ Theorem (see \cite{Dur} Theorem A.1.4), it suffices to show that
 
\begin{equation}\label{mixing1}
\lim_{n \to \infty} \overline{P} \big( A \cap  \overline{\theta}^{-n} B \big) = \overline{P} \big(  A \big) \overline{P} \big(  B \big),
\end{equation}
where we write
\begin{equation}\label{target}
A= \{ \overline{S} [0 , \overline{T}_{1}] = \lambda \} \ \ \ B= \{ \overline{S} [0 , \overline{T}_{1}] = \gamma \},
\end{equation} 
with $\lambda, \gamma \in \text{Bead}$ (see \eqref{possible} for $\text{Bead}$). In order to prove \eqref{mixing1} for those events, we want to show that two events
\begin{equation}\label{indep-i}
\{ \overline{S} [0 , \overline{T}_{1}] = \lambda \} \text{ and } \{ \overline{S} [\overline{T}_{n} , \overline{T}_{n+1}] - \overline{S} (\overline{T}_{n}) = \gamma \}
\end{equation}
are asymptotically independent as $n \to \infty$. Suppose that $\lambda \subset B(r)$ and $\gamma + \overline{S} (\overline{T}_{n}) \subset B (R)^{c}$. By taking $n$ large, we may assume $R$ is much bigger than $r$. Therefore, in order to show that two events in \eqref{indep-i} are almost independent, we need to control the independence between $\overline{S} [0, \overline{\tau}^{+}_{r}]$ and $\overline{S} [\overline{\tau}^{+}_{R}, \infty )$ (see Section 1.4 for $\overline{\tau}^{+}_{r}$). Since $\overline{S}$ is a conditioned random walk and does not satisfy the strong Markov property, in order to achieve it, we need a careful consideration as follows. Take two pairs of paths $\overline{\gamma}_{k} = (\gamma^{1}_{k}, \gamma^{2}_{k})$ ($k=1, 2$) satisfying that $\gamma^{i}_{k} \subset B(r)$ for all $i, k$ (we call $\overline{\gamma}_{k}$ an initial configuration). We are interested in the conditional law of $(S^{1} [0, \tau^{1}_{R}], S^{2} [0, \tau^{2}_{R}] )$ conditioned on $\overline{A}_{R}$ and $S^{i} [0, \tau^{i}_{r}] = \gamma^{i}_{k}$ for $i=1, 2$ (recall that $\overline{A}_{R}$ was defined as in \eqref{nonint}). The law of those conditional two-sided walks near $B(r)$ may have a big difference between $k=1$ and $k=2$. However, we will prove that the law the conditional two-sided walks after exiting a large ball (outside a large ball) for $k=1$ is close to that for $k=2$ in Theorem \ref{coupling3} below. This theorem allows to prove Theorem \ref{mixing}. 

In order to prove Theorem \ref{coupling3}, we need to compare the probability of the events $\overline{A}_{R} \cap \{ S^{i} [0, \tau^{i}_{r}] = \gamma^{i}_{k} \}$ for $k=1, 2$. We will give the difference between them in Lemma \ref{hikakusuru}.



\subsection{Forgetting an initial configuration}
The goal of this subsection is Theorem \ref{coupling3}. As we discussed above, Theorem \ref{coupling3} states roughly that the two-sided walk conditioned on $\overline{A}_{R}$ after exiting a large ball is almost independent from an initial configuration. In order to prove this theorem, we will show that the probability of $\overline{A}_{R}$ with a given initial configuration is almost independent from the initial configuration if the configuration satisfies some suitable conditions in Lemma \ref{hikakusuru}. To achieve the lemma, we will first collect some results derived in  \cite{S} in Section 3.1.1, which will be used later.


\subsubsection{Separation Lemma and Up-to-constants estimates}
In this subsection, we will collect some known results derived in \cite{S}. One of the important result here is so called ``separation lemma" (Proposition \ref{sep lemma}). This lemma  roughly states that two paths that are conditioned not to intersect are likely to be not very close at their endpoints. There are many ways to define the ``separation event". Here we choose a particular one considered in \cite{S}.

Assume $d=2$ or $3$. For each $l < n$ and $\overline{\gamma}= (\gamma ^{1}, \gamma ^{2}) \in \Gamma (l)$, see Section 1.4 for $\Gamma (l)$. We define
\begin{eqnarray}\label{non intersect} 
A_{n} ( \overline {\gamma} )=\left\{ \begin{array}{ll}
S^{1}[0, \tau^{1}_{n}] \cap \gamma^{2} = \emptyset , \\
S^{2}[0, \tau^{2}_{n}] \cap \gamma^{1} = \emptyset , \\
S^{1}[0, \tau^{1}_{n}] \cap S^{2}[0, \tau^{2}_{n}] = \emptyset  
\end{array}
\right\}.
\end{eqnarray}
Let $w^{i} = \gamma^{i} ( \text{len}\gamma ^{i} )$. We assume $S^{i}(0)= w^{i}$ when we consider $A_{n} ( \overline {\gamma} )$. Let
\begin{equation}\label{migihidari}
I(r) = \{ (x_{1}, \cdots , x_{d}) \in \mathbb{Z}^{d} : x_{1} \ge r \}, \ \ I^{\prime}(r) = \{ (x_{1}, \cdots , x_{d}) \in \mathbb{Z}^{d} : x_{1} \le -r \}.
\end{equation}
For each $l \in \mathbb{N}$, let \textsf{Sep}$(l)$ denote the separation event (see Figure \ref{Sep-fig} for \textsf{Sep})
\begin{equation}\label{sep}
\textsf{Sep}(l) = \Big\{ S^{1}[0, \tau^{1}_{2l}] \subset B \big( \frac{3l}{2} \big) \cup I \big( \frac{4l}{3} \big)  \Big\} \cap \Big\{ S^{2}[0, \tau^{2}_{2l}] \subset B \big( \frac{3l}{2} \big) \cup I^{\prime} \big( \frac{4l}{3} \big)  \Big\}.
\end{equation}

In Proposition 2.1 of \cite{S}, the following proposition was proved. This proposition states that conditioned on $A_{2l} ( \overline {\gamma} )$, the conditional probability of $S^{1}$ and $S^{2}$ being well-separated in the sense that they satisfy \textsf{Sep}$(l)$ is positive.

\begin{figure}\label{Sep-fig}
\begin{center}

 \includegraphics[width=7.5cm]{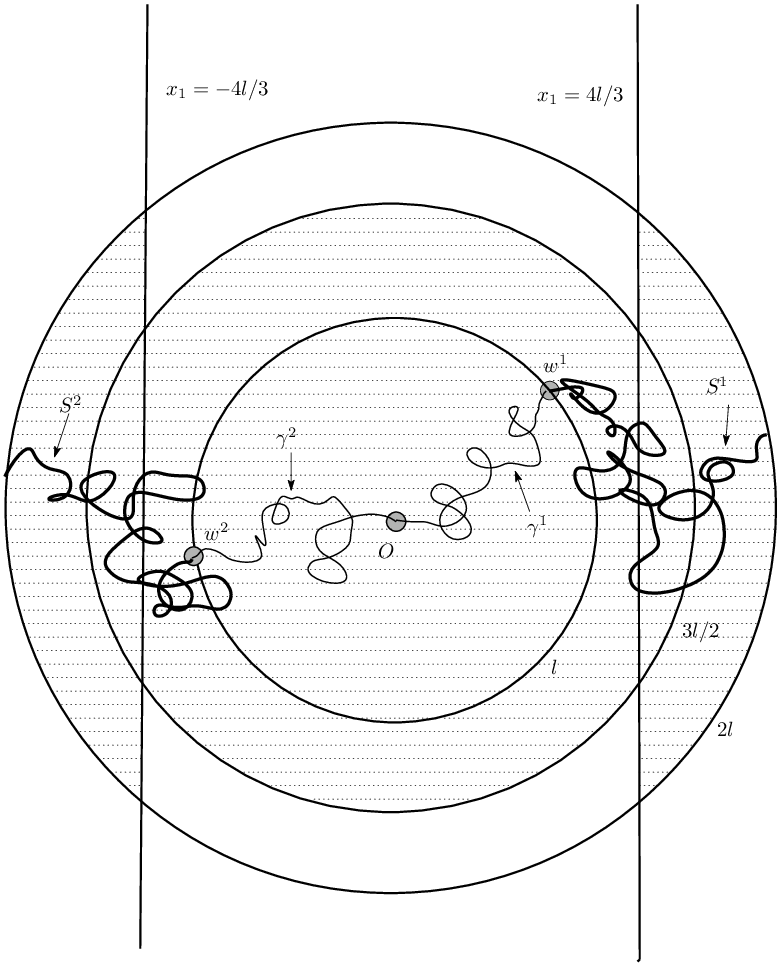}

\end{center}

 \caption{The event $A_{2l} ( \overline {\gamma} ) \cap \textsf{Sep}(l)$. }

 \end{figure}

\begin{prop}\label{sep lemma}
There exists $c > 0$ such that for all $l \in \mathbb{N}$ and $\overline{\gamma} = (\gamma ^{1} , \gamma ^{2} ) \in \Gamma (l)$, 
\begin{equation}\label{sep ineq}
P^{w^{1}, w^{2}} \big( \textsf{Sep}(l) \ \big| \ A_{2l} ( \overline {\gamma} ) \big) \ge c,
\end{equation}
where $w^{i} = \gamma^{i} ( \text{len}\gamma ^{i} )$.
\end{prop}

\medskip

In Corollary 2.2 of \cite{S} the following corollary was proved. Roughly speaking, we compare $A_{n} ( \overline {\gamma} )$ with the probability that $A_{2l} ( \overline {\gamma} )$ and $S^{1}, S^{2}$ do not intersect from first time that they hit $\partial B (2l)$ to the first time that they hit $\partial B (n)$. Namely, we want to separate the event into ``before" and ``after" exiting $B(2l)$. It turns out that the probability of latter event is comparable to $(\frac{n}{l})^{-\xi_{d}}$ where $ \xi_{d} $ denotes the intersection exponent as in \eqref{intersection-exp}. We also point out that the upper bound of \eqref{cor ineq} is not difficult and that Proposition \ref{sep lemma} was used to prove the lower bound of \eqref{cor ineq}, see Corollary 2.2 of \cite{S}.

\begin{cor}\label{cor1}
There exist $c_{1},c_{2}$ such that for all $l , n$ with $2l < n$ and all $\overline {\gamma} = (\gamma^{1} , \gamma^{2}) \in \Gamma (l)$ with $w^{i} = \gamma^{i}( \text{len}\gamma^{i} ) \in \partial B (l)$, 
\begin{equation}\label{cor ineq}
c_{1} (\frac{n}{l})^{-\xi_{d}} P^{w^{1},w^{2}} ( A_{2l} ( \overline {\gamma} ) ) \le P^{w^{1},w^{2}} ( A_{n} ( \overline {\gamma} ) ) \le c_{2} (\frac{n}{l})^{-\xi_{d}} P^{w^{1},w^{2}} ( A_{2l} ( \overline {\gamma} ) ).
\end{equation}
\end{cor}

\subsubsection{Good sets of paths}
The goal of this subsection is Lemma \ref{hikakusuru}. This lemma shows that the probability of $A_{n} ( \overline {\gamma} )$ is close to that of $A_{n} ( \overline {\gamma}' )$ assuming that initial configurations $\overline {\gamma}$ and $\overline {\gamma}'$ satisfy some condition. If $\overline {\gamma}$ and $\overline {\gamma}'$ are almost same configurations, then the difference between the probability of $A_{n} ( \overline {\gamma} )$ and that of $A_{n} ( \overline {\gamma}' )$ is small. With this in mind, we will first consider some condition that the initial configuration should satisfy (we call the configuration which satisfy the condition a ``good" one).

Take $l < n$ and two initial configurations $\overline{\gamma}= (\gamma ^{1}, \gamma ^{2}), \overline{\gamma}' = (\gamma ^{3} , \gamma ^{4} ) \in \Gamma (l)$. Let $w^{i}$ be the end point of $\gamma^{i}$. We write $q_{l, n} ( \overline{\gamma} ) = P^{w^{1}, w^{2}} \big( A_{n} ( \overline {\gamma} ) \big) $ and define $q_{l, n} ( \overline{\gamma}' )$ similarly. We want to compare the difference between $q_{l, n} ( \overline{\gamma} )$ and $q_{l, n} ( \overline{\gamma}' )$. We will only consider some set of initial configurations defined as follows. We define 
\begin{equation*}
\text{Good}_{l,k} = \{ \overline{\gamma} \in  \Gamma (l) \ : \ q_{l, 2l} ( \overline{\gamma} ) \ge 1/ \sqrt{k} \},
\end{equation*}
for $k \ge 1$ and $l$. When the endpoint of $\gamma^{i}$ is not very close to $\gamma^{3-i}$ for each $i=1, 2$, $q_{l, 2l} ( \overline{\gamma} )$ is not so small. Therefore such a configuration is in a good set $\text{Good}_{l,k}$.

Note that $\bigcup_{k} \text{Good}_{l,k} = \Gamma (l)$. Furthermore, by Corollary \ref{cor1}, we see that for $n > 2l$ 
\begin{align}\label{good-notgood}
& q_{l, n} ( \overline{\gamma} ) \ge  c_{1} \frac{1}{ \sqrt{k}} (\frac{n}{l})^{-\xi_{d}}, \ \ \overline{\gamma} \in \text{Good}_{l,k} \notag \\
& q_{l, n} ( \overline{\gamma} ) \le  c_{2} \frac{1}{ \sqrt{k}} (\frac{n}{l})^{-\xi_{d}}, \ \ \overline{\gamma} \notin \text{Good}_{l,k}.
\end{align}

The next lemma shows that conditioned on $A_{n} ( \overline {\gamma} )$, $S^{1}$ and $S^{2}$ do not return to a small ball with conditional high probability when $\overline {\gamma}$ is a good initial configuration. Furthermore, under the same assumption, we will also prove that conditioned on $A_{2n} ( \overline {\gamma} )$, the pair of $S^{1}$ and $S^{2}$ is in a good set of configurations with high probability. 

\begin{lem}\label{modorinikui}
There exists $c < \infty$ such that if $l \le m \le n$ and $2l < n$, then for all $\overline{\gamma} \in \text{Good}_{l,k}$,
\begin{align*}
&\Big| P^{w^{1}, w^{2}} \Big( A_{n} ( \overline {\gamma} ) \cap F \Big) - q_{l, n} ( \overline{\gamma} ) \Big| \le c \frac{1}{ \sqrt{k}}  q_{l, n} ( \overline{\gamma} ) \\
&\Big| P^{w^{1}, w^{2}} \Big( A_{2n} ( \overline {\gamma} ) \cap F \cap G \Big) - q_{l, 2n} ( \overline{\gamma} ) \Big| \le c \frac{1}{ \sqrt{k}}  q_{l, 2n} ( \overline{\gamma} ),
\end{align*}
where 
\begin{equation*}
F= \big\{ ( S^{1}[0, \tau^{1}_{m}] \cup S^{2}[0, \tau^{2}_{m}] ) \cap B \big( \frac{l}{k} \big) = \emptyset \big\}
\end{equation*}
and
\begin{equation*}
G= \Big\{ \big( \gamma ^{1} + S^{1}[0, \tau^{1}_{m}] , \gamma ^{2}+ S^{2}[0, \tau^{2}_{m}] \big) \in \text{Good}_{m,k} \Big\}.
\end{equation*}

\end{lem}
\begin{proof}
We first consider the first inequality in three dimensions.
Let
\begin{equation*}
F'= \big\{  S^{1}[0, \tau^{1}_{m}]  \cap B \big( \frac{l}{k} \big) \neq \emptyset \big\}.
\end{equation*}
For $d=3$, by Proposition 1.5.10 of \cite{Law b}, we have
\begin{equation*}
P^{w^{1}, w^{2}} ( F' ) \le c \frac{1}{k}.
\end{equation*}
However, Corollary 4.6 of \cite{L} shows that 
\begin{equation*}
\max_{z^{1}, z^{2} \in B(l) } P^{z^{1}, z^{2}} \big( S^{1} [0, \tau^{1}_{n}] \cap S^{2} [0, \tau^{2}_{n}] = \emptyset \big) \le C (\frac{n}{l})^{-\xi_{3}}.
\end{equation*}
(We mention that $\zeta_{d}$ was used to stand for the intersection exponent in \cite{L} and that $\xi_{d} = 2 \zeta_{d}$.) Hence, by the strong Markov property, 
\begin{equation*}
P^{w^{1}, w^{2}} \Big( A_{n} ( \overline {\gamma} ) \cap F' \Big) \le c \frac{1}{k} (\frac{n}{l})^{-\xi_{3}}.
\end{equation*}
On the other hand, since $\overline{\gamma} \in \text{Good}_{l,k}$, we have 
\begin{equation*}
q_{l, n} ( \overline{\gamma} ) \ge c_{1} \frac{1}{ \sqrt{k}} (\frac{n}{l})^{-\xi_{3}},
\end{equation*}
which implies that 
\begin{align*}
\Big| P^{w^{1}, w^{2}} \Big( A_{n} ( \overline {\gamma} ) \cap F \Big) - q_{l, n} ( \overline{\gamma} ) \Big| &\le 2 P^{w^{1}, w^{2}} \Big( A_{n} ( \overline {\gamma} ) \cap F' \Big) \\
&\le 2c \frac{1}{k} (\frac{n}{l})^{-\xi_{3}} \\
&\le c \frac{1}{ \sqrt{k}}  q_{l, n} ( \overline{\gamma} ).
\end{align*}

Next we consider $d=2$. 
For $d=2$, since $P^{w^{1}, w^{2}} ( F' )$ is not small enough, we need a different way as follows. Assume that the event $F'$ occurs. Let $u_{1} := \inf \{ t \ | \ S^{1} (t) \in B \big( \frac{l}{k} \big) \}$ and let $u_{2} := \inf \{ t \ge u_{1} \ | \ S^{1} (t) \in \partial B \big( l \big) \}$. By applying the Beurling estimate (see Theorem 6.8.1 of \cite{Law b2}) to both events $S^{1} [0, u_{1}] \cap \gamma^{2} = \emptyset$ and $S^{1} [u_{1}, u_{2}] \cap \gamma^{2} = \emptyset$, we have

\begin{equation*}
 P^{w^{1}, w^{2}} ( F' \cap \{ S^{1}[0 , u_{2} ] \cap \gamma^{2} = \emptyset \} ) \le  \frac{c}{k}.
\end{equation*}
Therefore, by using the strong Markov property as above, we also get the first inequality for $d=2$. 

The second inequality is easy. Using the strong Markov property as well as Corollary \ref{cor1} and \eqref{good-notgood}, we see that
\begin{align*}
P^{w^{1}, w^{2}} \Big( A_{2n} ( \overline {\gamma} ) \cap  G^{c} \Big) &= P^{w^{1}, w^{2}} \Big( A_{m} (\overline {\gamma} ) \cap   A_{2n} ( \overline {\gamma} ) \cap G^{c} \Big) \\
&\le P^{w^{1}, w^{2}} \Big( A_{m} (\overline {\gamma} ) \Big) P^{w^{1}, w^{2}} \Big( A_{2n} ( \overline {\gamma} ) \  \big| \ A_{m} (\overline {\gamma} ) \cap  G^{c} \Big) \\
&\le c (\frac{m}{l})^{-\xi_{d}} P^{w^{1},w^{2}} ( A_{2l} ( \overline {\gamma} ) )  \frac{1}{\sqrt{k}} (\frac{n}{m})^{-\xi_{d}} \\
&\le c P^{w^{1},w^{2}} ( A_{2l} ( \overline {\gamma} ) ) \frac{1}{\sqrt{k}} (\frac{n}{l})^{-\xi_{d}} \\
&\le c q_{l, 2n} ( \overline{\gamma} ) \frac{1}{\sqrt{k}}.
\end{align*}
So we finish the proof.
\end{proof}

For two pairs of paths $\overline{\gamma} = (\gamma ^{1} , \gamma ^{2} ),  \overline{\gamma}' = (\gamma ^{3} , \gamma ^{4} ) \in \Gamma (l)$, we write $\overline{\gamma} =_{k}  \overline{\gamma}'$ if $\overline{\gamma}$ after exiting $B (\frac{l}{k})$ is same as that of $\overline{\gamma}'$. Namely, if we let $\tau_{i} (\frac{l}{k}) = \inf \{ j \ge 0 \ | \ \gamma^{i} (j) \notin {\cal B} (\frac{l}{k}) \}$ for $ 1 \le i \le 4$, then we write $\overline{\gamma} =_{k}  \overline{\gamma}'$ when $\gamma^{i} [ \tau_{i} (\frac{l}{k}), \text{len} \gamma^{i} ] = \gamma^{i+2} [ \tau_{i+2} (\frac{l}{k}), \text{len} \gamma^{i+2} ]$ for each $i=1, 2$.

The next lemma shows that if the initial configuration $\overline{\gamma}$ is good and $\overline{\gamma} =_{k}  \overline{\gamma}'$, then the probability of $A_{n} ( \overline {\gamma} )$ is close to that of $A_{n} ( \overline {\gamma}' )$.


\begin{lem}\label{hikakusuru}
There exists $c_{0} < \infty$ such that if $l < 2n$, $k \ge 1$, $ \overline{\gamma} \in \text{Good}_{l,k}$, $\overline{\gamma}' \in \Gamma (l)$ and $\overline{\gamma} =_{k}  \overline{\gamma}'$, then we have
\begin{equation*}
\big| q_{l, n} ( \overline{\gamma} ) - q_{l, n} ( \overline{\gamma}' ) \big| \le c_{0} \frac{1}{\sqrt{k}} q_{l, n} ( \overline{\gamma} ).
\end{equation*}
\end{lem}

\begin{proof}
Let $w^{i}$ ($i=1,2$) be the endpoint of $\gamma^{i}$. Since $\overline{\gamma} =_{k}  \overline{\gamma}'$, the endpoint of $\gamma^{i+2}$ is $w^{i}$. Let $S^{i}$ be the simple random walk started at $w^{i}$. For each $i=1,2$, let 
\begin{equation*}
F_{i} = \{ S^{i} [0, \tau^{i}_{n} ] \cap B (\frac{l}{k}) = \emptyset \}.
\end{equation*}
Since $\overline{\gamma} =_{k}  \overline{\gamma}'$, when $F_{1}$ and $F_{2}$ occur, the probability of $A_{n} ( \overline {\gamma} )$ is same as that of $A_{n} ( \overline {\gamma}' )$.  So we have
\begin{equation*}
P^{w^{1}, w^{2}} \big( A_{n} ( \overline {\gamma} ) \cap F_{1} \cap F_{2} \big) = P^{w^{1}, w^{2}} \big( A_{n} ( \overline {\gamma}' ) \cap F_{1} \cap F_{2} \big).
\end{equation*}
Thus,
\begin{equation*}
\big| q_{l, n} ( \overline{\gamma} ) - q_{l, n} ( \overline{\gamma}' ) \big| \le \sum_{i=1}^{2} P^{w^{1}, w^{2}} \big( A_{n} ( \overline {\gamma} ) \cap F_{i}^{c}  \big) + \sum_{i=1}^{2} P^{w^{1}, w^{2}} \big( A_{n} ( \overline {\gamma}' ) \cap F_{i}^{c}  \big).
\end{equation*}
We will only show that 
\begin{equation*}
P^{w^{1}, w^{2}} \big( A_{n} ( \overline {\gamma} ) \cap F_{1}^{c}  \big) \le c \frac{1}{\sqrt{k}} q_{l, n} ( \overline{\gamma} ).
\end{equation*}
The other three terms can be estimated similarly. We first consider when $d=3$. Recall that we define $u_{1} := \inf \{ t \ | \ S^{1} (t) \in B \big( \frac{l}{k} \big) \}$ and let $u_{2} := \inf \{ t \ge u_{1} \ | \ S^{1} (t) \in \partial B \big( l \big) \}$ in the proof of the previous lemma. Assume that $F_{1}^{c}$ occurs. Then $u_{1} < u_{2} < \tau^{1}_{n}$. Thus, using the Markov property and Proposition 1.5.10 of \cite{Law b},
\begin{align*}
&P^{w^{1}, w^{2}} \big( A_{n} ( \overline {\gamma} ) \cap F_{1}^{c}  \big) \le P^{w^{1}, w^{2}} \big( u_{1} < u_{2} < \xi_{n}, \ S^{1} [u_{2} , \tau^{1}_{n}] \cap S^{2}[0, \tau^{2}_{n}] = \emptyset \big) \\
&\le P^{w^{1}} ( u_{1} < \tau^{1}_{n} ) \max_{x, y \in \partial {\cal B} (l) } P^{x,y} \big( S^{1} [0 , \tau^{1}_{n}] \cap S^{2}[0, \tau^{2}_{n}] = \emptyset \big) \\
& \le  \frac{c}{k} \max_{x, y \in \partial {\cal B} (l) } P^{x,y} \big( S^{1} [0 , \tau^{1}_{n}] \cap S^{2}[0, \tau^{2}_{n}] = \emptyset \big).
\end{align*}
By Corollary 4.6 in \cite{L}, 
\begin{equation*}
\max_{x, y \in \partial {\cal B} (l) } P^{x,y} \big( S^{1} [0 , \tau^{1}_{n}] \cap S^{2}[0, \tau^{2}_{n}] = \emptyset \big) \le c \big( \frac{n}{l} \big)^{- \xi_{3}}.
\end{equation*}
Therefore, if $\overline{\gamma} \in \text{Good}_{l,k}$, by \eqref{good-notgood} we have
\begin{equation*}
P^{w^{1}, w^{2}} \big( A_{n} ( \overline {\gamma} ) \cap F_{1}^{c}  \big) \le c \frac{1}{\sqrt{k}} q_{l, n} ( \overline{\gamma} ),
\end{equation*}
which finishes the proof when $d=3$.

Now we will show the lemma when $d=2$. Note that by using the strong Markov property and Corollary 4.6 in \cite{L} again,
\begin{align*}
&P^{w^{1}, w^{2}} \big( A_{n} ( \overline {\gamma} ) \cap F_{1}^{c}  \big) \\
& \le P^{w^{1}, w^{2}} \big( u_{1} < u_{2} < \xi_{n}, \ S^{1}[0, u_{2} ] \cap \gamma^{2} = \emptyset, \ S^{1} [u_{2} , \tau^{1}_{n}] \cap S^{2}[0, \tau^{2}_{n}] = \emptyset \big) \\
&\le P^{w^{1}} \big( u_{1} < u_{2} < \xi_{n}, \ S^{1}[0, u_{2} ] \cap \gamma^{2} = \emptyset \big) \max_{x, y \in \partial {\cal B} (l) } P^{x,y} \big( S^{1} [0 , \tau^{1}_{n}] \cap S^{2}[0, \tau^{2}_{n}] = \emptyset \big) \\
&\le c \big( \frac{n}{l} \big)^{- \xi_{2}} P^{w^{1}} \big( u_{1} < u_{2} < \xi_{n}, \ S^{1}[0, u_{2} ] \cap \gamma^{2} = \emptyset \big).
\end{align*}
However, if we apply the Beurling estimate (see Theorem 6.8.1 of \cite{Law b2}) to two events $\{ S^{1}[0, u_{1} ] \cap \gamma^{2} = \emptyset \}$ and $\{ S^{1}[u_{1}, u_{2} ] \cap \gamma^{2} = \emptyset \}$, we have
\begin{equation*}
P^{w^{1}} \big(  u_{1} < u_{2} < \xi_{n}, \ S^{1}[0, u_{2} ] \cap \gamma^{2} = \emptyset \big) \le  \frac{c}{k}.
\end{equation*}
Therefore, if $\overline{\gamma} \in \text{Good}_{l,k}$, by \eqref{good-notgood} we have
\begin{equation*}
P^{w^{1}, w^{2}} \big( A_{n} ( \overline {\gamma} ) \cap F_{1}^{c}  \big) \le c \frac{1}{\sqrt{k}} q_{l, n} ( \overline{\gamma} ),
\end{equation*}
which finishes the proof when $d=2$.
\end{proof}

\subsubsection{Coupling}
The goal of this subsection is Theorem \ref{coupling3}. Theorem \ref{coupling3} roughly states that the conditional law of $S^{1}, S^{2}$ after exiting a large ball conditioned on $A_{n} ( \overline{\gamma})$ is almost independent of the initial configuration $\overline{\gamma}$. To state more precisely, we write $\mu _{l, n} (\overline{\gamma})$ for the probability measure on the space of two-sided paths, which is induced by $\big( S^{1} [0, \tau^{1}_{n}] ,  S^{2} [0, \tau^{2}_{n}] \big)$ conditioned on the event $A_{n} ( \overline{\gamma})$. In Theorem \ref{coupling3}, we want to show that $\mu _{l, n} (\overline{\gamma})$ is close to $\mu _{l, n} (\overline{\gamma}')$ in ``outside" a large ball. To achieve this, we will consider a coupling. This approach is based on the same spirit as in Theorem 4.1 of \cite{Greg}. If $\overline{\gamma} =_{k} \overline{\gamma}'$ for large $k$, then we can couple $\mu _{l, n} (\overline{\gamma})$ and $\mu _{l, n} (\overline{\gamma}')$ with high probability such that they are close (see Proposition \ref{coupling1}). But if $k$ is not large, then we can couple them with positive probability (see Proposition \ref{coupling2}). Using these propositions, we will prove Theorem \ref{coupling3}. Proposition \ref{coupling1}, Proposition \ref{coupling2} and Theorem \ref{coupling3} are discrete analogs of Proposition 4.4, Proposition 4.5 and Theorem 4.1 of \cite{Greg} respectively.

For $\overline{\gamma} \in \Gamma (l)$ and $l < m < n$, let $\mu _{l,m,n} (\overline{\gamma})$ be the probability measure on the space of two-sided paths, which is induced by $\big( S^{1} [0, \tau^{1}_{m}] ,  S^{2} [0, \tau^{2}_{m}] \big)$ conditioned on the event $A_{n} ( \overline{\gamma})$. Note that a two-sided path $\overline{ \lambda } = (\lambda ^{1} , \lambda ^{2} )$ is in the support of $\mu _{l,m,n} (\overline{\gamma})$ if and only if $\lambda^{i} (0) = w^{i}$ for each $i=1,2$ and $(\gamma^{1} + \lambda^{1}, \gamma^{2}+\lambda^{2}) \in \Gamma (m)$.

We will first prove the following proposition which states that if $\overline{\gamma} =_{k} \overline{\gamma}'$ for $k$ large enough, then the paths stay coupled with high probability. 


\begin{prop}\label{coupling1}
There exists $C_{0}$ such that we have the following: Suppose that $k, l , m , n$ are positive integers with $2l < m$ and $2m < n$. Let $\overline{\gamma} , \overline{\gamma}' \in \Gamma (l)$. Assume that $\overline{\gamma} \in \text{Good}_{l,k}$ and $\overline{\gamma} =_{k} \overline{\gamma}'$. Then we can define $\overline{\lambda }_{l,m} , \overline{\lambda }'_{l,m}$ on the same probability space $( \Omega , {\cal F} , P )$ such that $\overline{\lambda }_{l,m}$ has the distribution $\mu _{l,m,n} (\overline{\gamma})$, $\overline{\lambda }'_{l,m}$ has the distribution $\mu _{l,m,n} (\overline{\gamma}')$, and that
\begin{align*}
&P \big( \overline{\lambda }_{l,m} =_{\frac{km}{l}} \overline{\lambda }'_{l,m} \big) \ge 1-C_{0} \frac{1}{\sqrt{k}}, \\
&P \big( \overline{\lambda }_{l,m} \in \text{Good}_{m , k} \big) \ge 1 - C_{0} \frac{1}{\sqrt{k}}.
\end{align*}

\end{prop}

\begin{proof}
Take $\overline{\gamma} = (\gamma ^{1} , \gamma ^{2} ),  \overline{\gamma}' = (\gamma ^{3} , \gamma ^{4} ) \in \Gamma (l)$. Assume that $\overline{\gamma} \in \text{Good}_{l,k}$ and $\overline{\gamma} =_{k} \overline{\gamma}'$. In order to prove the lemma, as in the proof of Proposition 4.4 of \cite{Greg}, it suffices to estimate the total variation distance between $\mu _{l,m,n} (\overline{\gamma})$ and $\mu _{l,m,n} (\overline{\gamma}')$. Let $w^{i}$ be the endpoint of $\gamma^{i}$ ($i=1, 2$). 

Take a pair of paths $\overline{\lambda} = (\lambda^{1}, \lambda^{2})$ with $\lambda^{i} (0) = w^{i}$ for each $i=1,2$ and $(\gamma^{1} + \lambda^{1}, \gamma^{2}+\lambda^{2}) \in \Gamma (m)$. Suppose that $\lambda^{i} \cap B ( \frac{l}{k} ) = \emptyset$ for each $i =1, 2$. Since $\overline{\gamma} =_{k} \overline{\gamma}'$, we see that $(\gamma^{3} + \lambda^{1}, \gamma^{4}+\lambda^{2}) \in \Gamma (m)$. We write $v^{i}$ for the endpoint of $\lambda^{i}$. Let $\overline{\gamma} + \overline{\lambda} := (\gamma^{1} + \lambda^{1}, \gamma^{2}+\lambda^{2})$ and we write $\overline{\gamma}' + \overline{\lambda}$ for $(\gamma^{3} + \lambda^{1}, \gamma^{4}+\lambda^{2})$. Suppose that $\overline{\gamma} + \overline{\lambda} \in \text{Good}_{m , k}$. Note that by the strong Markov property,
\begin{equation*}
\mu _{l,m,n} (\overline{\gamma}) [\overline{\lambda}] = \frac{P^{w^{1}, w^{2}} \big( S^{i} [0, \tau^{i}] = \lambda^{i} \text{ for } i= 1, 2 \big) P^{v^{1}, v^{2}} \big( A_{n} ( \overline{\gamma} + \overline{\lambda} ) \big) }{P^{w^{1}, w^{2}} \big( A_{n} ( \overline{\gamma} ) \big)}.
\end{equation*}
Since $\overline{\gamma} + \overline{\lambda} \in \text{Good}_{m , k}$, $\overline{\gamma}' + \overline{\lambda} \in \Gamma (m)$ and $\overline{\gamma} =_{k} \overline{\gamma}'$, by using Lemma \ref{hikakusuru}, we see that 
\begin{equation*}
| \mu _{l,m,n} (\overline{\gamma}) [\overline{\lambda}] - \mu _{l,m,n} (\overline{\gamma}') [\overline{\lambda}] | \le \frac{c}{\sqrt{k}} \mu _{l,m,n} (\overline{\gamma}) [\overline{\lambda}].
\end{equation*}

Let $H$ be the set of pairs of paths $\overline{\lambda} = (\lambda^{1}, \lambda^{2})$ with $\lambda^{i} (0) = w^{i}$ such that $(\gamma^{1} + \lambda^{1}, \gamma^{2}+\lambda^{2}) \in \Gamma (m)$,  $\lambda^{i} \cap B ( \frac{l}{k} ) = \emptyset$, and $\overline{\gamma} + \overline{\lambda} \in \text{Good}_{m , k}$. Then by Lemma \ref{modorinikui},
\begin{equation*}
\mu _{l,m,n} (\overline{\gamma}) [ H^{c} ] \le \frac{c}{\sqrt{k}}.
\end{equation*}

Therefore, we have
\begin{equation*}
P \big( \overline{\lambda }_{l,m} \neq_{\frac{km}{l}} \overline{\lambda }'_{l,m} \big) = \frac{1}{2} \parallel \mu _{l,m,n} (\overline{\gamma}) - \mu _{l,m,n} (\overline{\gamma}') \parallel  \le C_{0} \frac{1}{\sqrt{k}},
\end{equation*}
for some $C_{0} < \infty$. The second inequality follows from Lemma \ref{modorinikui} and we finish the proof.
\end{proof}

What about the case that $\overline{\gamma} =_{k} \overline{\gamma}'$ for small $k$, or $\overline{\gamma}$ and $\overline{\gamma}'$ do not have the same end points? In such cases, we will show that the coupling still can be started, with positive probability in the next proposition. 

We fix an integer $K$ such that $C_{0} \frac{2}{\sqrt{K}} < \frac{1}{2}$ where $C_{0}$ is the constant as in Proposition \ref{coupling1}. For the case that $k$ is not large, or $\overline{\gamma}$ and $\overline{\gamma}'$ do not have the same end points, we will use the following coupling.

\begin{prop}\label{coupling2}
There exists $b >0$ such that if $l < n$ are positive integers with $Kl < n$ and $\overline{\gamma}, \overline{\gamma}' \in  \Gamma (l)$, then we can couple $\mu _{l,Kl,n} (\overline{\gamma})$ and $\mu _{l,Kl,n} (\overline{\gamma}')$ such that with probability at least $b$,
\begin{equation*}
\overline{\lambda }_{l,Kl} =_{\frac{K}{4}} \overline{\lambda }'_{l,Kl},
\end{equation*}
and
\begin{equation*}
\overline{\lambda }_{l,Kl} \in \text{Good}_{Kl , \frac{K}{4}}.
\end{equation*}
\end{prop}

\begin{proof}
Take $\overline{\gamma} = (\gamma ^{1} , \gamma ^{2} ),  \overline{\gamma}' = (\gamma ^{3} , \gamma ^{4} ) \in  \Gamma (l)$. We attach $(S^{1}[0, \tau^{1}_{Kl}], S^{2} [0, \tau^{2}_{Kl}] )$ to $\overline{\gamma}$ and  $\overline{\gamma}'$ in the following way. By Proposition \ref{sep lemma}, with positive conditional probability conditioned on $A_{n} ( \overline{\gamma} )$ (resp. $A_{n} ( \overline{\gamma}' )$), we can attach $(S^{1} [0, \tau^{1}_{2l}], S^{2} [0, \tau^{2}_{2l}]) $ to $\overline{\gamma}$ (resp. $\overline{\gamma}'$) such that $\overline{\gamma} + (S^{1} [0, \tau^{1}_{2l}], S^{2} [0, \tau^{2}_{2l}])$ (resp. $\overline{\gamma}' + (S^{1} [0, \tau^{1}_{2l}], S^{2} [0, \tau^{2}_{2l}])$) satisfies $\textsf{Sep}(l)$. Next we can attach $(S^{1} [\tau^{1}_{2l}, \tau^{1}_{4l}], S^{2} [\tau^{2}_{2l}, \tau^{2}_{4l}]) $ with positive conditional probability such that $\overline{\gamma} + (S^{1} [0, \tau^{1}_{4l}], S^{2} [0, \tau^{2}_{4l}])$ and $\overline{\gamma}' + (S^{1} [0, \tau^{1}_{4l}], S^{2} [0, \tau^{2}_{4l}])$ have the same endpoints and both of them satisfy $\textsf{Sep}(2 l)$. Finally, since they are separated, with positive probability, we can attach the same random walks $(S^{1} [\tau^{1}_{4l}, \tau^{1}_{Kl}], S^{2} [\tau^{2}_{4l}, \tau^{2}_{Kl}]) $ such that $\overline{\lambda }_{l,Kl} =_{\frac{K}{4}} \overline{\lambda }'_{l,Kl}$ and $\overline{\lambda }_{l,Kl} \in \text{Good}_{Kl , K}$. So we finish the proof.
\end{proof}

For $\overline{\gamma} \in \Gamma (l)$ and $l <n$, we write $\mu _{l,n} (\overline{\gamma})$ for $\mu _{l,n,n} (\overline{\gamma})$. Recall that $\mu _{l,n,n} (\overline{\gamma})$ is the probability measure induced by $ \big( S^{1} [0, \tau^{1}_{n}] ,  S^{2} [0, \tau^{2}_{n}] \big)$ conditioned that $A_{n}(\overline{\gamma})$ holds. For any two paths $\overline{\gamma}, \overline{\gamma}' \in \Gamma (l)$, we want to say $\mu _{l,n} (\overline{\gamma})$ and $\mu _{l,n} (\overline{\gamma}')$ are close. Clearly, if the endpoints for $\overline{\gamma}$ are not same as those of $\overline{\gamma}'$, $\mu _{l,n} (\overline{\gamma})$ and $\mu _{l,n} (\overline{\gamma}')$ are not close near $B (l)$. So we will show that $\big( S^{1} [0, \tau^{1}_{n}] ,  S^{2} [0, \tau^{2}_{n}] \big)$ conditioned that $A_{n}(\overline{\gamma})$ and $\big( S^{1} [0, \tau^{1}_{n}] ,  S^{2} [0, \tau^{2}_{n}] \big)$ conditioned that $A_{n}(\overline{\gamma}')$ are close in the outside a large ball. To show this, we construct two random variables $\overline{\lambda }_{l, n},  \overline{\lambda }'_{l, n}$ on the same probability space $(\Omega, {\cal F}, P)$ such that $\overline{\lambda }_{l, n}$ has the distribution $\mu _{l,n} (\overline{\gamma})$, $\overline{\lambda }'_{l, n}$ has the distribution $\mu _{l,n} (\overline{\gamma}')$ and they are close in outside a large ball with high probability, that is, we will show the following coupling result.

\begin{thm}\label{coupling3}
There exist $0 < c, \beta < \infty$ such that for all integers $l, m , n$ with $0 < 2l< m \le n$ and all $\overline{\gamma}, \overline{\gamma}' \in \Gamma (l)$, we can define $\overline{\lambda }_{l, n},  \overline{\lambda }'_{l, n}$ on the same probability space $(\Omega, {\cal F}, P)$ such that $\overline{\lambda }_{l, n}$ has the distribution $\mu _{l,n} (\overline{\gamma})$, $\overline{\lambda }'_{l, n}$ has the distribution $\mu _{l,n} (\overline{\gamma}')$, and that
\begin{equation}\label{owarida}
P \big( \overline{\lambda }_{l, n} =_{\frac{n}{m}} \overline{\lambda }'_{l, n} \big) \ge 1 - c \big( \frac{m}{l} \big)^{- \beta}.
\end{equation}
\end{thm}

\begin{proof}
Recall that $K$ is the constant as in Proposition \ref{coupling2}. We write $J$ for the largest integer such that $K^{J} l \le \frac{m}{2}$. We will first construct a coupling of $\mu _{l, K^{J} l, n} (\overline{\gamma})$ and $\mu _{l, K^{J} l, n} (\overline{\gamma}')$. To achieve it, we first define a coupling of $\overline{\gamma} + (S^{1} [0, \tau^{1}_{Kl} ], S^{2} [0, \tau^{2}_{Kl} ] )$ and $\overline{\gamma}' + (S^{1} [0, \tau^{1}_{Kl} ], S^{2} [0, \tau^{2}_{Kl} ] )$, and then we define a coupling of $\overline{\gamma} + (S^{1} [0, \tau^{1}_{K^{2} l} ], S^{2} [0, \tau^{2}_{K^{2} l} ] )$ and $\overline{\gamma}' + (S^{1} [0, \tau^{1}_{K^{2} l} ], S^{2} [0, \tau^{2}_{K^{2} l} ] )$, etc. For each $j \le J$, we write $\sigma (j) $ for the largest integer $k$ such that in the coupling at $j$-th stage, 
\begin{equation*}
\overline{\gamma} + (S^{1} [0, \tau^{1}_{K^{j} l} ], S^{2} [0, \tau^{2}_{K^{j} l} ] ) =_{k} \overline{\gamma}' + (S^{1} [0, \tau^{1}_{K^{j} l} ], S^{2} [0, \tau^{2}_{K^{j} l} ] )
\end{equation*}
and $\Big( \overline{\gamma} + (S^{1} [0, \tau^{1}_{K^{j} l} ], S^{2} [0, \tau^{2}_{K^{j} l} ] ) \Big) \in \text{Good}_{K^{j} l , k}$. 

 Given $\overline{\gamma} + (S^{1} [0, \tau^{1}_{K^{j} l} ], S^{2} [0, \tau^{2}_{K^{j} l} ] )$ and $\overline{\gamma}' + (S^{1} [0, \tau^{1}_{K^{j} l} ], S^{2} [0, \tau^{2}_{K^{j} l} ] )$ after $j$-th stage, we proceed the next step as follows.
\begin{itemize}
\item We construct a coupling of $\overline{\gamma} + (S^{1} [0, \tau^{1}_{K^{j+1} l} ], S^{2} [0, \tau^{2}_{K^{j+1} l} ] )$ and \\ $\overline{\gamma}' + (S^{1} [0, \tau^{1}_{K^{j+1} l} ], S^{2} [0, \tau^{2}_{K^{j+1} l} ] )$ using Proposition \ref{coupling1} if $\sigma (j) \ge \frac{K}{4}$.

\item We construct a coupling of $\overline{\gamma} + (S^{1} [0, \tau^{1}_{K^{j+1} l} ], S^{2} [0, \tau^{2}_{K^{j+1} l} ] )$ and \\ $\overline{\gamma}' + (S^{1} [0, \tau^{1}_{K^{j+1} l} ], S^{2} [0, \tau^{2}_{K^{j+1} l} ] )$ using Proposition \ref{coupling2} if $\sigma (j) < \frac{K}{4}$.
\end{itemize} 
Given $\overline{\gamma} + (S^{1} [0, \tau^{1}_{K^{j} l} ], S^{2} [0, \tau^{2}_{K^{j} l} ] )$ and $\overline{\gamma}' + (S^{1} [0, \tau^{1}_{K^{j} l} ], S^{2} [0, \tau^{2}_{K^{j} l} ] )$ with $\sigma (j) = k \ge \frac{K}{4}$. Then Proposition \ref{coupling1} shows that the conditional probability of $\sigma (j+1)$ being equal to  $k K$ is bounded below by $1-C_{0} \frac{1}{\sqrt{k}}$. On the other hand, given a configuration after $j$-th stage, the conditional probability that $\sigma (j+1) \ge \frac{K}{4}$ is bounded below by $b$ by Proposition \ref{coupling2}. Therefore, by comparison with a one-dimensional Markov chain as in the proof of Theorem 4.1 \cite{Greg}, we see that there exist $0 < c, \beta < \infty$ such that
\begin{equation*}
P \big( \sigma (J) \le K^{\frac{J}{2}} \big) \le c (\frac{m}{l})^{-\beta}.
\end{equation*}
Therefore, we can construct a coupling of $\mu _{l, K^{J} l, n} (\overline{\gamma})$ and $\mu _{l, K^{J} l, n} (\overline{\gamma}')$ such that, with probability at least $1- c (\frac{m}{l})^{-\beta}$, 
\begin{equation*}
\Big( \overline{\gamma} + (S^{1} [0, \tau^{1}_{K^{J} l} ], S^{2} [0, \tau^{2}_{K^{J} l} ] ) \Big) =_{K^{\frac{J}{2}}} \Big( \overline{\gamma}' + (S^{1} [0, \tau^{1}_{K^{J} l} ], S^{2} [0, \tau^{2}_{K^{J} l} ] ) \Big),
\end{equation*}
and $\Big( \overline{\gamma} + (S^{1} [0, \tau^{1}_{K^{J} l} ], S^{2} [0, \tau^{2}_{K^{J} l} ] ) \Big) \in \text{Good}_{K^{J} l , K^{\frac{J}{2}}}$.

Once we have constructed the coupling as above, by using Proposition \ref{coupling1}, we can couple $\mu _{l,n} (\overline{\gamma})$ and $\mu _{l,n} (\overline{\gamma}')$ such that, with probability $\ge 1 - c \big( \frac{m}{l} \big)^{- \beta}$, \eqref{owarida} holds. Thus we finish the proof of the theorem.
\end{proof}

\subsection{Local dependence of global cut points and mixing}
In this subsection, we will show that the shift $\overline{\theta}$ is mixing in Theorem \ref{mixing}. As we discussed at the beginning of Section 3, we need to control the independence between two events $\{ \overline{S} [0 , \overline{T}_{1}] = \lambda \}$ and $\{ \overline{S} [\overline{T}_{n} , \overline{T}_{n+1}] - \overline{S} (\overline{T}_{n}) = \gamma \}$ with given two paths $\lambda$ and $\gamma$. To achieve it, we want to replace the global cut times $\overline{T}_{1}$ and $\overline{T}_{n}$ in the events into ``local cut times". By definition, the event ``$k$ is a global cut time for $\overline{S}^{1}$" depends on both $\overline{S}^{1}[0, \infty)$ and $\overline{S}^{2}[0, \infty)$. However, using the transience of $\overline{S}^{i}$, it turns out that if $\overline{S}^{1} [a_{k}, k] \cap \overline{S}^{1} [k+1, b_{k}]= \emptyset$  for suitable times $0 < a_{k} < b_{k} < \infty$ depending on $k$, then with high probability $k$ becomes in fact a global cut time. Such ``local dependence" of global cut points were also used in \cite{S} to give a lower bound of the number of global cut points using the second moment method and Markovian-type ``iteration arguments" (see Proposition 3.6 of \cite{S} for the details). 

\begin{thm}\label{mixing}
The translation shift $\overline{\theta}$ is mixing.
\end{thm}


\begin{proof}
Recall that $\text{Bead} $ was defined as in \eqref{possible}. In order to prove the theorem, by the $\pi$-$\lambda$ Theorem (see \cite{Dur} Theorem A.1.4), it suffices to show that
\begin{equation}\label{mixing1-1}
\lim_{n \to \infty} \overline{P} \big( A \cap  \overline{\theta}^{-n} B \big) = \overline{P} \big(  A \big) \overline{P} \big(  B \big),
\end{equation}
where we write
\begin{equation}\label{target1}
A= \{ \overline{S} [0 , \overline{T}_{1}] = \lambda \} \ \ \ B= \{ \overline{S} [0 , \overline{T}_{1}] = \gamma \},
\end{equation} 
with $\lambda, \gamma \in \text{Bead}$. In order to prove \eqref{mixing1} for those events, we want to show that two events
\begin{equation}\label{indep-i-1}
\{ \overline{S} [0 , \overline{T}_{1}] = \lambda \} \text{ and } \{ \overline{S} [\overline{T}_{n} , \overline{T}_{n+1}] - \overline{S} (\overline{T}_{n}) = \gamma \}
\end{equation}
are asymptotically independent as $n \to \infty$. With this in mind, take $\lambda, \gamma \in \text{Bead}$ and let 
\begin{equation*}
A= \{ \overline{S}[0 , \overline{T}_{1} ] = \lambda \}, \ B^{n} = \{ \overline{S}[\overline{T}_{n} , \overline{T}_{n+1} ] - \overline{S}( \overline{T}_{n} )= \gamma \}.
\end{equation*}
We will show that 
\begin{equation*}
| \overline{P} \big( A \cap B^{n} \big) - \overline{P} \big( A ) \overline{P} \big( B ) | \to 0, 
\end{equation*}
as $n \to \infty$.

For each $L$, by the transience of $\overline{S}$, 
\begin{equation}\label{transi1}
\overline{P} \big(    \overline{S} [ \overline{\tau}^{+}_{L} , \infty ) \cap \lambda \neq \emptyset \big) \le c L^{-\frac{1}{2}},
\end{equation}
for some constant $c$ depending on $\lambda$. (For this inequality, we used the following fact proved in Lemma 3.8 of \cite{S}: for each $m < n$, 
\begin{equation}\label{mukashino}
\overline{P} \big(    \overline{S} [ \overline{\tau}^{+}_{n} , \infty ) \cap B ( m) \neq \emptyset \big) \le C \big( \frac{n}{m} \big)^{-\frac{1}{2}},
\end{equation}
for $d=2, 3$.)

We call $k$ a cut time up to $\overline{\tau}^{+}_{L}$ if 
\begin{equation*}
\overline{S} [0, k] \cap \overline{S} [k+1, \overline{\tau}^{+}_{L} ] = \emptyset.
\end{equation*}
Let $\overline{T}^{L}_{1}$ be the first cut time up to $\overline{\tau}^{+}_{L}$ and 
\begin{equation*}
A^{L} = \{ \overline{S}[0 , \overline{T}^{L}_{1} ] = \lambda \}.
\end{equation*}
Using \eqref{transi1}, we see that 
\begin{equation*}
| \overline{P} \big( A \cap B^{n} \big) -  \overline{P} \big( A^{L} \cap B^{n} \big) | \le  c L^{-\frac{1}{2}}.
\end{equation*}

Note that
\begin{equation}\label{chotonoyatsu}
\overline{P} \big( | \overline{S} (\overline{T}_{n}) | > n^{1/4} \big) \ge 1- C n^{-\frac{1}{24}}.
\end{equation}
To see this, it follows that $\overline{T}_{n} \ge n$ and that
\begin{equation*}
\overline{P} \big( \max_{0 \le j \le n} | \overline{S} (j) | < n^{1/3} \big) \le C e^{-c n^{\frac{1}{6}}}.
\end{equation*}
(See Proposition 2.4.5 of \cite{Law b2} for this inequality.) So we can assume that $\overline{\tau}^{+}_{n^{1/3}} \le n \le \overline{T}_{n}$ with probability at least $1- C e^{-c n^{\frac{1}{6}}}$. Now suppose that $\overline{\tau}^{+}_{n^{1/3}} \le n \le \overline{T}_{n}$ and $| \overline{S} (\overline{T}_{n}) | \le n^{1/4}$. This implies that $\overline{S} [\overline{\tau}^{+}_{n^{1/3}}, \infty ) \cap B ( n^{1/4} ) \neq \emptyset$. However, by \eqref{mukashino}, we have
\begin{equation*}
\overline{P} \big( \overline{S} [\overline{\tau}^{+}_{n^{1/3}}, \infty ) \cap B ( n^{1/4} ) \neq \emptyset \big) \le C n^{-\frac{1}{24}},
\end{equation*}
which gives \eqref{chotonoyatsu}. So we can assume that $\overline{\tau}^{+}_{n^{1/4}} \le \overline{T}_{n}$ with probability $\ge 1- C n^{-\frac{1}{24}}$.

Combining \eqref{chotonoyatsu} with \eqref{mukashino}, we see that for all $n \ge L^{16}$
\begin{equation}\label{transi2}
\overline{P} \big( \overline{S} [ \overline{T}_{n} , \infty) \cap B ( L^{2} ) \neq \emptyset \big) \le C L^{-\frac{2}{3}}.
\end{equation}
However, if we assume that $\overline{S} [ \overline{T}_{n} , \infty) \cap B ( L^{2} ) = \emptyset$, whether $B^{n}$ holds or not does not depend on $\overline{S}[0, \overline{\tau}^{+}_{L^2}]$. So if we define an event
\begin{equation*}
F^{n} := B^{n} \cap \{ \overline{S} [ \overline{T}_{n} , \infty) \cap B ( L^{2} ) = \emptyset \},
\end{equation*}
then $F^{n}$ is measurable for $\overline{S}[ \overline{\tau}^{+}_{L^2} , \infty )$, and we have
\begin{equation*}
| \overline{P} \big( A \cap B^{n} \big) -  \overline{P} \big( A^{L} \cap F^{n} \big) | \le C L^{-\frac{1}{2}}.
\end{equation*}
Let 
\begin{equation*}
\Gamma'(L) = \{ \overline{\gamma} \in \Gamma(L) \ | \ \overline{P} \big( A^{L}, \ (\overline{S}^{1}[0, \overline{\tau}^{1}_{L} ], \overline{S}^{2}[0, \overline{\tau}^{2}_{L} ] )= \overline{\gamma} \big) > 0 \}
\end{equation*}
be the set of pairs of paths $\overline{\gamma} \in \Gamma(L)$ such that with positive probability, $A^{L}$ and $(\overline{S}^{1}[0, \overline{\tau}^{1}_{L} ], \overline{S}^{2}[0, \overline{\tau}^{2}_{L} ] )= \overline{\gamma}$ occur. Then by conditioning $(\overline{S}^{1}[0, \overline{\tau}^{1}_{L} ], \overline{S}^{2}[0, \overline{\tau}^{2}_{L} ] )$, we see that $\overline{P} \big( A^{L} \cap F^{n} \big)$ is equal to
\begin{equation}\label{pukosuke}
\sum_{ \overline{\gamma} \in \Gamma'(L) } \overline{P} \big( (\overline{S}^{1}[0, \overline{\tau}^{1}_{L} ], \overline{S}^{2}[0, \overline{\tau}^{2}_{L} ] )= \overline{\gamma} \big) \overline{P} \big( F^{n} \ | \ (\overline{S}^{1}[0, \overline{\tau}^{1}_{L} ], \overline{S}^{2}[0, \overline{\tau}^{2}_{L} ] )= \overline{\gamma} \big).
\end{equation}
Applying Theorem \ref{coupling3} to the conditional probability in the right hand side of \eqref{pukosuke}, we see that for all $\overline{\gamma}$,
\begin{equation*}
| \overline{P} \big( F^{n} \ | \ (\overline{S}^{1}[0, \overline{\tau}^{1}_{L} ], \overline{S}^{2}[0, \overline{\tau}^{2}_{L} ] )= \overline{\gamma} \big) - \overline{P} \big( F^{n} ) | \le c L^{-\beta},
\end{equation*}
for some absolute constants $0 < c, \beta < \infty$. Therefore, taking sum for $\overline{\gamma} \in \Gamma' (L)$, we have
\begin{align*}
& \Big| \sum_{ \overline{\gamma} \in \Gamma'(L) } \overline{P} \big( (\overline{S}^{1}[0, \overline{\tau}^{1}_{L} ], \overline{S}^{2}[0, \overline{\tau}^{2}_{L} ] )= \overline{\gamma} \big) \overline{P} \big( F^{n} \ | \ (\overline{S}^{1}[0, \overline{\tau}^{1}_{L} ], \overline{S}^{2}[0, \overline{\tau}^{2}_{L} ] )= \overline{\gamma} \big) \\
&- \overline{P} \big( A^{L} \big) \overline{P} \big( F^{n} ) \Big| \le c L^{-\beta}.
\end{align*}
Therefore, by using \eqref{transi1} and \eqref{transi2} again, for all $n > L^{16}$
\begin{equation*}
| \overline{P} \big( A \cap B^{n} \big) - \overline{P} \big( A \big) \overline{P} \big( B^{n} \big) | \le c L^{-\beta},
\end{equation*}
which finishes the proof since it follows from Theorem \ref{stationary-o} that $\overline{P} \big( B^{n} \big) = \overline{P} \big( B \big)$.
\end{proof}

\section{Application}
Recall that we call a random walk path between consecutive cut points a piece. When we study a random walk path using the pieces, some issues come from the fact that each piece has no common distribution and they are strongly correlated. This is the one of main reason that we are interested in non-intersecting random walks. Theorems \ref{stationary-o} and \ref{mixing} allow to use results of ergodic theory when we study the non-intersecting random walk. In this section we will consider some application of Theorem \ref{stationary-o} and \ref{mixing} along with ergodic theory to estimate quantities generated by the path of $\overline{S}$ in Theorem \ref{critical exp}. The quantities that we are interested in are the length of the loop-erasure, graph distance and effective resistance of $\overline{S} [ 0, \overline{T}_{n} ]$. We will give the definitions of these quantities and briefly explain backgrounds of them in Section 4.1. Then we will apply Aaronson's results derived in \cite{A} to analyze the quantities in Section 4.2.


\subsection{LERW, graph distance and effective resistance}
In this subsection, we will introduce three quantities generated by random walk paths. Those quantities are loop-erased random walk (LERW), shortest path graph distance and effective resistance. We will consider the growth rate of these quantities along with ergodic theory in Section 4.2.

The first quantity that we are interested in is loop-erased random walk (LERW). Loop-erased random walk is a model for a random simple path, which is created by running a simple random walk and, whenever the random walk hits its path, removing the resulting loop and continuing. We begin with the precise definition of loop-erasing procedure of a given path in $\mathbb{Z}^{d}$.
For a deterministic path $\lambda$ with length $m$, we denote the loop-erasure of $\lambda$ by $LE (\lambda)$. More precisely, let $\lambda = [\lambda_{0}, \lambda_{1}, \cdots , \lambda_{m}]$ be a path in $\mathbb{Z}^{d}$. We let 
\begin{equation*}
s_{0} = \sup \{ j : \lambda_{j} = \lambda_{0} \}.
\end{equation*}
and, for $i > 0$,
\begin{equation*}
s_{i} = \sup \{ j : \lambda_{j} = \lambda_{s_{i-1} + 1} \}.
\end{equation*}
Let 
\begin{equation*}
n= \inf \{ i : s_{i} = m \}.
\end{equation*}
Then
\begin{equation}\label{def-lew}
LE (\lambda ) = [\lambda_{s_{0}}, \lambda_{s_{1}}, \cdots , \lambda_{s_{n}} ].
\end{equation}

We are interested in a loop-erasure of a random walk path and we call it loop-erased random walk (LERW). Let us give brief backgrounds of LERW here. Since Lawler \cite{lew} introduced LERW, this process has played an important role both in the statistical physics and mathematics literature. It is closely related to the uniform spanning tree (UST). Let $u$ and $v$ be two vertices on UST. Then UST contains precisely one simple path between $u$ and $v$. Pemantle \cite{Pem} proved that the distribution of this simple path is identical to the distribution of the LERW from $u$ to $v$. Furthermore, the UST  can be generated using LERWs by Wilson's algorithm \cite{Wil}. Concerning a scaling limit of LERW on $\mathbb{Z}^{d}$, the followings are known. For $d \ge 4$, Lawler \cite{Law b, Law-3} showed that the scaling limit of the LERW is Brownian motion (note that Brownian motion is a simple curve almost surely for $d \ge 4$). Lawler, Schramm and Werner \cite{LSW2} showed that LERW has a conformally invariant scaling limit for $d=2$, SLE. Indeed, SLE was introduced by Schramm \cite{Sch} as a candidate for the scaling limit of LERW. For $d=3$, Kozma \cite{Koz} showed that the scaling limit of LERW exists and is invariant to dilations and rotations.

Let $M_{n}$ be the number of steps of $\text{LE} ( S[0, \tau_{n} ] )$, the loop-erasure of $S[0, \tau_{n} ]$. In \cite{Ken}, using domino tilings, it was proved that for $d=2$,
\begin{equation}
\lim_{n \to \infty} \frac{ \log E(M_{n}) }{ \log n} = \frac{5}{4}.
\end{equation}
Recently, Lawler \cite{Lawler} showed that 
\begin{equation}
 E(M_{n}) \asymp n^{ \frac{5}{4} },
\end{equation}
(see \eqref{asymp} for the definition of $\asymp$). The quantity $\frac{5}{4}$ is called the growth exponent for planar loop-erased random walk. 

In 3 dimensions, physicists conjecture that there exists $\beta$ such that 
\begin{equation}\label{big problem}
\lim_{n \to \infty} \frac{ \log E(M_{n}) }{ \log n} = \beta,
\end{equation}
and did numerical experiments to show that $\beta = 1.62 \pm 0. 01$ (\cite{GB}, \cite{Wil2}). However, rigorously the existence of $\beta$ is not proved. The best rigorous bounds are (\cite{Law5}) 
\begin{equation*}
1 < \beta \le \frac{5}{3},
\end{equation*}
if $\beta$ exists. We will prove the existence of the exponent $\beta$ in Section 7 and Section 8 (see Theorem \ref{main result-kaetta}, Theorem \ref{leipzig} and Proposition \ref{zhan}).

While LERW is not a Markov
chain, it satisfies the following ``domain Markov property'': for any Markov chain
$X$, if we condition that the first $k$ steps of $LE (X)$ is equal to a given path $\omega$, the conditional distribution of the rest part of $LE(X)$ is same as the loop-erasure of $X$ starting from the endpoint of $\omega$ conditioned to avoid $\omega$. More precisely, we have the following proposition.

\begin{prop}\label{DMP} (Domain Markov Property \cite{Law b}) 
Let $X$ be a Markov chain in $\mathbb{Z}^{d}$, $A \subset \mathbb{Z}^{d}$ and $\omega = [\omega_{0}, \omega_{1}, \cdots , \omega_{m}]$ be a path in $A$. We let $\tau^{X}_{A} := \inf \{ k \ | \ X (k) \notin A \}$ be the first time that $X$ exits from $A$. Define a new Markov chain
$Y$ to be $X$ started at $\omega_{m}$ conditioned that $X [1, \tau^{X}_{A}] \cap \omega = \emptyset$. We write $\tau^{Y}_{A}$ for the first time that $Y$ exits from $A$.
Suppose that
$\omega' = [\omega'_{0}, \cdots ,  \omega'_{m'} ]$ is a path satisfying that $\omega'_{0} = \omega_{m}$ and $\omega + \omega' := [ \omega_{0}, \omega_{1}, \cdots , \omega_{m}, \omega'_{1}, \cdots ,  \omega'_{m'} ]$ is a path from $\omega_{0}$ to $\partial A$. Then,
\begin{equation*}
P \Big( LE \big( X [0, \tau^{X}_{A}] \big) = \omega + \omega' \ \big| \ LE \big( X [0, \tau^{X}_{A}] \big) [0, m] = \omega \Big) = P \Big( LE \big( Y [0, \tau^{Y}_{A} ] \big) = \omega' \Big).
\end{equation*}
\end{prop}

The second quantity that we are interested in is the graph distance.
For a graph $G$, let $d_{G}( \cdot , \cdot )$ be the shortest path graph distance on $G$. 

Finally, we will introduce the effective resistance on a graph $G$.
To define it, we first introduce a quadratic form $\cal E$ by
\begin{equation*}
{\cal E} (f,g) = \frac{1}{2} \textstyle\sum\limits_{\begin{subarray}{c} x,y \in V, \\   \{ x,y \} \in E \end{subarray}} ( f(x) -f(y))(g(x)-g(y)).
\end{equation*}
If we regard $G$ as an electrical network with a unit resistor on each edge in $E$, then ${\cal E} (f,f)$ is the energy dissipation when the vertices of $V$ are at a potential $f$. Set 
\begin{equation*}
H^{2}= \{ f \in \mathbb{R} ^{V}: {\cal E} (f,f) < \infty \}.
\end{equation*}
Let $A,B$ be disjoint subsets of $V$. The effective resistance between $A$ and $B$ is defined by 
\begin{equation}\label{resistanceteigi}
R_{G}(A,B) ^{-1} = \inf \{ {\cal E}(f,f) : f \in H^{2}, f|_{A} =1, f|_{B}=0 \}.
\end{equation}
We write $R_{G} (x,y)= R_{G} ( \{ x \}, \{ y \})$ for the effective resistance between two points $x$ and $y$.

\subsection{Critical exponents}
Recall that $\overline{T}_{n}$ stands for the $n$-th global cut time for $\overline{S}$. We are interested in the growth rate of the following three quantities;
\begin{itemize}
\item $\text{len} \big( LE (\overline{S} [0, \overline{T}_{n} ] ) \big)$, length of the loop-erasure of $\overline{S} [0, \overline{T}_{n} ]$,

\item $d_{\overline{S} [0, \overline{T}_{n} ]} (0, \overline{S} ( \overline{T}_{n} ) )$, graph distance between the origin and $\overline{S} ( \overline{T}_{n} )$ on $\overline{S} [0, \overline{T}_{n} ]$,

\item $R_{\overline{S} [0, \overline{T}_{n} ]} (0, \overline{S} ( \overline{T}_{n} ) )$, effective resistance between the origin and $\overline{S} ( \overline{T}_{n} )$ on $\overline{S} [0, \overline{T}_{n} ]$,
\end{itemize}
where for the second and third quantities, we think of $\overline{S} [0, \overline{T}_{n} ]$ as a (random) graph whose vertex set is $\{ \overline{S} (k) \ | \ k \in [0, \overline{T}_{n} ] \}$ and edge set is $\{ [ \overline{S} (k), \overline{S} (k+1) ] \ | \ k \in [0, \overline{T}_{n} -1 ] \}$.

If we let $f = \text{len} \big( LE (\overline{S} [0, \overline{T}_{1} ] ) \big)$ be the length of the loop-erasure of $\overline{S} [0, \overline{T}_{1} ]$, then we see that 
\begin{equation}\label{ergoergo}
\text{len} \big( LE (\overline{S} [0, \overline{T}_{n} ] ) \big) = \sum_{k=0}^{n-1} f \circ \overline{\theta}^{k}.
\end{equation}
The graph distance and effective resistance also can be written in terms of the sum along with the shift $\overline{\theta}$ similarly. Recall that by Theorem \ref{stationary-o} and Theorem \ref{mixing}, the law of $\overline{S}$ is invariant under the shift $\overline{\theta}$ and $\overline{\theta}$ is mixing. Therefore, if $f$ had a finite first moment, we could apply Birkhoff's theorem to show that the right hand side of \eqref{ergoergo} grows like $c n$ for some constant $c$. However, this is not the case for three quantities above. In order to study the growth rate of the sum in \eqref{ergoergo} for the case that $f$ does not have a finite first moment, we will use results from \cite{A}. In the next theorem, we will show that there exists a deterministic constant $\alpha$ such that the sum in \eqref{ergoergo} divided by $n^{a}$ converges to 0 almost surely when $a > \alpha$, and it diverges when $a < \alpha$. Same results hold for the graph distance and effective resistance.

\begin{thm}\label{critical exp}
Let $d=2,3$. There exist $\alpha_{\ell}(d), \alpha_{g}(d)$ and $\alpha_{r}(d)$ such that the following holds;
\begin{enumerate}
\item[$(1)$] $1 \le \alpha_{r}(d) \le \alpha_{g}(d) \le \alpha_{\ell}(d) < \infty$.
\item[$(2)$] for every $\alpha_{1} > \alpha_{\ell}(d)$, $\alpha_{2} > \alpha_{g}(d)$ and $\alpha_{3} > \alpha_{r}(d)$, we have 
\begin{align}
&\lim_{n \to \infty} \frac{ \text{len} \big( LE (\overline{S} [0, \overline{T}_{n} ] ) \big) } { n^{\alpha_{1}} } = 0, \ \ \overline{P}\text{-a.s.,} \label{loop-loop1} \\
&\lim_{n \to \infty} \frac{ d_{\overline{S} [0, \overline{T}_{n} ]} (0, \overline{S} ( \overline{T}_{n} ) )  }{n^{\alpha_{2}} } = 0, \ \ \overline{P}\text{-a.s.,} \\
&\lim_{n \to \infty} \frac{ R_{\overline{S} [0, \overline{T}_{n} ]} (0, \overline{S} ( \overline{T}_{n} ) ) }{n^{\alpha_{3}} } = 0, \ \ \overline{P}\text{-a.s.} 
\end{align}
\item[$(3)$] for every $\alpha_{1} < \alpha_{\ell}(d)$, $\alpha_{2} < \alpha_{g}(d)$ and $\alpha_{3} < \alpha_{r}(d)$, we have 
\begin{align}
&\limsup_{n \to \infty} \frac{ \text{len} \big( LE (\overline{S} [0, \overline{T}_{n} ] ) \big) } { n^{\alpha_{1}} } = \infty, \ \ \overline{P}\text{-a.s.,} \label{loop-loop2} \\
&\limsup_{n \to \infty} \frac{ d_{\overline{S} [0, \overline{T}_{n} ]} (0, \overline{S} ( \overline{T}_{n} ) )  }{n^{\alpha_{2}} } = \infty, \ \ \overline{P}\text{-a.s.,} \\
&\limsup_{n \to \infty} \frac{ R_{\overline{S} [0, \overline{T}_{n} ]} (0, \overline{S} ( \overline{T}_{n} ) )  }{n^{\alpha_{3}} } = \infty, \ \ \overline{P}\text{-a.s.} 
\end{align}
\end{enumerate}
\end{thm}

\begin{proof}
By Rayleigh's monotonicity law (see Section 1.4 of \cite{DS}), we see that 
\begin{equation*}
d_{\overline{S} [0, \overline{T}_{n} ]} (0, \overline{S} ( \overline{T}_{n} ) ) \ge R_{\overline{S} [0, \overline{T}_{n} ]} (0, \overline{S} ( \overline{T}_{n} ) ) \ge n.
\end{equation*}
 Since the loop-erasure of $\overline{S} [0, \overline{T}_{n} ]$ is a path from the origin to $\overline{S} ( \overline{T}_{n} )$ contained in $\overline{S} [0, \overline{T}_{n} ]$, it is clear that the length of the loop-erasure of $\overline{S} [0, \overline{T}_{n} ]$ is bounded below by $d_{\overline{S} [0, \overline{T}_{n} ]} (0, \overline{S} ( \overline{T}_{n} ) )$. Therefore, $1 \le \alpha_{r}(d) \le \alpha_{g}(d) \le \alpha_{\ell}(d)$ if these exponents exist. On the other hand, by Theorem 1.1 in \cite{S}, it follows that
\begin{equation*}
\lim_{n \to \infty} \frac {\log  \overline{T}_{n}} {\log n} = \frac{2}{2- \xi_{d}},  \ \overline{P}\text{-a.s.,}
\end{equation*}
where $\xi_{d}$ is the constant as in \eqref{intersection-exp}. Since $\text{len} \big( LE (\overline{S} [0, \overline{T}_{n} ] ) \big) \le \overline{T}_{n}$, we see that $\alpha_{\ell}(d) < \infty$ if it exists.

We will prove the existence of $\alpha_{\ell}(d)$ such that the claims \eqref{loop-loop1} and \eqref{loop-loop2} hold. The existence of $\alpha_{g}(d)$ and $\alpha_{r}(d)$ can be proved similarly. By Theorems \ref{stationary-o} and \ref{mixing}, the law of $\overline{S}$ is invariant under the shift $\overline{\theta}$ and $\overline{\theta}$ is mixing. Therefore, by using Theorem A' in \cite{A}, we see that for all $\alpha > 1$ either
\begin{equation*}
\lim_{n \to \infty} \frac{ \text{len} \big( LE (\overline{S} [0, \overline{T}_{n} ] ) \big) } { n^{\alpha} } = 0 \ \ \overline{P}\text{-a.s.,}
\end{equation*}
or
\begin{equation}\label{aaro-aaro}
\limsup_{n \to \infty} \frac{ \text{len} \big( LE (\overline{S} [0, \overline{T}_{n} ] ) \big) } { n^{\alpha} } = \infty \ \ \overline{P}\text{-a.s.}
\end{equation}
With this in mind, we define
\begin{equation}\label{exponent}
\alpha_{\ell}(d) := \inf \Big\{ \alpha > 1 \ \Big| \ \overline{P} \Big( \lim_{n \to \infty} \frac{ \text{len} \big( LE (\overline{S} [0, \overline{T}_{n} ] ) \big) } { n^{\alpha} } = 0 \Big) =1 \Big\}.
\end{equation}
Then this definition immediately gives \eqref{loop-loop1}. In order to see \eqref{loop-loop2}, take $\alpha_{1} < \alpha_{\ell}(d)$. By \eqref{aaro-aaro}, with probability one we have
\begin{equation*}
\limsup_{n \to \infty} \frac{ \text{len} \big( LE (\overline{S} [0, \overline{T}_{n} ] ) \big) } { n^{\alpha_{1}} } = \infty.,
\end{equation*}
which gives \eqref{loop-loop2}. For the exponents $\alpha_{g}(d)$ and $\alpha_{r}(d)$, we can define them by replacing $\text{len} \big( LE (\overline{S} [0, \overline{T}_{n} ] ) \big)$ by the graph distance and effective resistance in \eqref{exponent}. So we finish the proof. 
\end{proof}

\section{LERW in two dimensions}
Theorem \ref{critical exp} shows that $\text{len} \big( LE (\overline{S} [0, \overline{T}_{n} ] ) \big)$ divided by $n^{\alpha_{\ell}(d) + \epsilon}$ converges to zero almost surely for all $\epsilon > 0$, and $\text{len} \big( LE (\overline{S} [0, \overline{T}_{n} ] ) \big)$ divided by $n^{\alpha_{\ell}(d) - \epsilon}$ diverges in the sense that the $\limsup$ of the ratio goes to infinity. It is natural to expect that $\text{len} \big( LE (\overline{S} [0, \overline{T}_{n} ] ) \big) = n^{\alpha_{\ell}(d) + o(1)}$ a.s. as $n \to \infty$. Unfortunately \eqref{loop-loop2} is not sufficient to show it. In order to prove that $\text{len} \big( LE (\overline{S} [0, \overline{T}_{n} ] ) \big) = n^{\alpha_{\ell}(d) + o(1)}$, we need to show that the limit (not $\limsup$) of $\text{len} \big( LE (\overline{S} [0, \overline{T}_{n} ] ) \big)$ divided by $n^{\alpha_{\ell}(d) - \epsilon}$ is infinity for all $\epsilon > 0$. In this section, we will prove this for $d=2$. We will also give the exact value of $\alpha_{\ell}(2)$. The goal of this section is the following theorem.

\begin{thm}\label{2-dim}
Let $d=2$. Then we have
\begin{equation}\label{2-dim exp}
\alpha_{\ell}(2) = \frac{5}{3}.
\end{equation}
Furthermore it follows that with probability one,
\begin{equation}\label{regular}
\lim_{n \to \infty} \frac{ \log  \text{len} \big( LE (\overline{S} [0, \overline{T}_{n} ] ) \big) }{ \log n}  = \frac{5}{3}.
\end{equation}
\end{thm}

We will prove this theorem in Section 5.1 and 5.2. In Section 5.1, we will show that $\alpha_{\ell}(2) \le \frac{5}{3}$ by proving that $\text{len} \big( LE (\overline{S} [0, \overline{T}_{n} ] ) \big)$ divided by $n^{\frac{5}{3} + \epsilon }$ converges to zero for all $\epsilon > 0$, see Proposition \ref{upper-2dim-loop}. In Section 5.2, we will show that $\alpha_{\ell}(2) \ge \frac{5}{3}$ by proving that $\text{len} \big( LE (\overline{S} [0, \overline{T}_{n} ] ) \big)$ divided by $n^{\frac{5}{3} - \epsilon }$ goes to infinity for all $\epsilon > 0$, see Proposition \ref{lower-2dim-loop}. Theorem \ref{2-dim} immediately follows from Proposition \ref{upper-2dim-loop} and \ref{lower-2dim-loop}.

Before going to the proof, we will explain the reason that $\alpha_{\ell}(2) = \frac{5}{3}$ intuitively. Theorem 1.1 in \cite{S} gives that $n$-th global cut time $\overline{T}_{n}$ is of order $n^{\frac{2}{2- \xi_{2}}}$ when $d=2$. In 2 dimensions, we have $\xi_{2} = \frac{5}{4}$, see (1.4). Therefore $\overline{T}_{n}$ is of order $n^{\frac{8}{3}}$. It is known that the length of the loop-erasure of $S[0, n]$ is of order $n^{\frac{5}{8}}$ in 2 dimensions (see \cite{Ken} for this). It turns out that the length of the loop-erasure of $\overline{S} [0, n]$ is also of order $n^{\frac{5}{8}}$. Thus we expect that $\text{len} \big( LE (\overline{S} [0, \overline{T}_{n} ] ) \big) \approx (n^{\frac{8}{3}})^{\frac{5}{8}} = n^{\frac{5}{3}}$.



\subsection{Upper bound for $\alpha_{\ell}(2)$}
In this subsection, we will show that $\alpha_{\ell}(2) \le \frac{5}{3}$ by proving that $\text{len} \big( LE (\overline{S} [0, \overline{T}_{n} ] )$ divided by $n^{\frac{5}{3} + \epsilon }$ converges to zero for all $\epsilon > 0$ in Proposition \ref{upper-2dim-loop}. The proof is based on the following steps. We first compare $\overline{T}_{n}$ with the first time that $\overline{S}$ exits from a ball so that $\overline{\tau}^{+}_{n^{\frac{4}{3} - \epsilon}} \le \overline{T}_{n} \le \overline{\tau}^{+}_{n^{\frac{4}{3} + \epsilon}}$. Since $\overline{T}_{n}$ is a cut time, we see that $\text{len} \big( LE (\overline{S} [0, \overline{T}_{n} ] ) \big)$ is bounded above by $\text{len} \big( LE (\overline{S} [0, \overline{\tau}^{+}_{n^{\frac{4}{3} + \epsilon}} ] ) \big)$. Now we use results from \cite{BM} which give exponential tail bounds for the length of the loop-erasure of the usual simple random walk $S$ in 2 dimensions. Theorem 1.1 of \cite{BM} gives that $\text{len} \big( LE (S [0, \tau_{n^{\frac{4}{3} + \epsilon}} ] ) \big)$ is bounded above by $n^{\frac{5}{3} + 3 \epsilon}$ with high probability. Since the probability that $\text{len} \big( LE (S [0, \tau_{n^{\frac{4}{3} + \epsilon}} ] ) \big) \ge n^{\frac{5}{3} + 3 \epsilon}$ is much smaller than the probability that $S^{1}$ and $S^{2}$ do not intersect up to the first time that they exit from $B (\tau_{n^{\frac{4}{3} + \epsilon}})$, we can conclude that $\text{len} \big( LE (\overline{S} [0, \overline{\tau}^{+}_{n^{\frac{4}{3} + \epsilon}} ] ) \big)$ is also bounded above by $n^{\frac{5}{3} + 3 \epsilon}$ with high probability.

\begin{prop}\label{upper-2dim-loop}
Let $d=2$. For all $\alpha > \frac{5}{3}$,
\begin{equation}
\overline{P} \Big( \lim_{n \to \infty} \frac{ \text{len} \big( LE (\overline{S} [0, \overline{T}_{n} ] ) \big) } { n^{\alpha} } = 0 \Big) =1.
\end{equation}
In particular, $\alpha_{\ell}(2) \le \frac{5}{3}$.
\end{prop}

\begin{proof}
Fix $\epsilon > 0$. We write $\overline{K}^{+}_{n}$ for the number of global cut times of $\overline{S}$ in $[0, \overline{\tau}^{+}_{n} ]$. In the proof of Theorem 1.1 of \cite{S}, it was shown that 
\begin{equation*}
n^{\frac{3}{4} - \epsilon } \le \overline{K}^{+}_{n} \le n^{\frac{3}{4} + \epsilon } \text{ for large } n, \ \overline{P}\text{-a.s.}
\end{equation*}
This gives that 
\begin{equation}\label{jirousan}
\overline{\tau}^{+}_{n^{\frac{4}{3} - 2 \epsilon }} \le \overline{T}_{n} \le \overline{\tau}^{+}_{n^{\frac{4}{3} + 2 \epsilon }} \text{ for large } n, \ \overline{P}\text{-a.s.}
\end{equation}

On the other hand, Theorem 1.1 of \cite{BM} gives the following upper tail bound of the length of the loop-erasure of the usual simple random walk $S$ for $d=2$;
\begin{equation}\label{uruseina}
P \Big( \text{len} \big( LE ( S[0, \tau_{ n^{\frac{4}{3}+ 2 \epsilon} } ]  ) \big) \ge n^{\frac{5}{3} + 3 \epsilon } \Big) \le c_{0} e^{- c_{1} n^{\frac{\epsilon}{4}}},
\end{equation}
for some $0 < c_{0}, c_{1} < \infty$. 

Recall that Corollary 4.6 of \cite{L} gives that for $N > n^{\frac{4}{3}+ 2 \epsilon}$
\begin{equation*}
\max_{x, y \in B ( n^{\frac{4}{3}+ 2 \epsilon} ) } P^{x, y} \big( S^{1} [0, \tau^{1}_{N}] \cap S^{2} [0, \tau^{2}_{N}] = \emptyset \big) \le c \big( \frac{N}{n^{\frac{4}{3}+2 \epsilon}} \big)^{-\frac{5}{4}}.
\end{equation*}
Using this along with \eqref{uruseina} and the strong Markov property, we see that
\begin{align}\label{kietane}
&P \Big( S^{1} [0, \tau^{1}_{N}] \cap S^{2} [1, \tau^{2}_{N}] = \emptyset, \ \text{len} \big( LE ( S^{2} [0, \tau^{2}_{ n^{\frac{4}{3}+ 2 \epsilon} } ]  ) \big) \ge n^{\frac{5}{3} + 3 \epsilon } \Big) \notag \\
& \le c_{0} e^{- c_{1} n^{ \frac{\epsilon}{4}}} c \big( \frac{N}{n^{\frac{4}{3}+2 \epsilon}} \big)^{-\frac{5}{4}} \le c N^{-\frac{5}{4}} e^{- \frac{c_{1}}{2} n^{\frac{\epsilon}{4}}}.
\end{align}
Theorem 1.3 of \cite{L} gives that $P \Big( S^{1} [0, \tau^{1}_{N}] \cap S^{2} [1, \tau^{2}_{N}] = \emptyset \Big) \asymp N^{-\frac{5}{4}}$. By dividing both sides of \eqref{kietane} by $P \Big( S^{1} [0, \tau^{1}_{N}] \cap S^{2} [1, \tau^{2}_{N}] = \emptyset \Big)$ first and then by letting $N$ go to infinity, we have
\begin{equation*}
\overline{P} \Big( \text{len} \big( LE ( \overline{S} [0, \overline{\tau}^{+}_{ n^{\frac{4}{3}+ 2 \epsilon} } ] )  \big) \ge n^{\frac{5}{3} + 3 \epsilon } \Big) \le c e^{- \frac{c_{1}}{2} n^{\frac{\epsilon}{4}}}.
\end{equation*}
By the Borel-Cantelli lemma, we have
\begin{equation}\label{jiro-2}
\text{len} \big( LE ( \overline{S} [0, \overline{\tau}^{+}_{ n^{\frac{4}{3}+ 2 \epsilon} } ] )  \big) \le n^{\frac{5}{3} + 3 \epsilon } \text{ for large } n, \ \overline{P}\text{-a.s.}
\end{equation}
Combining this by \eqref{jirousan}, with probability one, $\overline{T}_{n} \le \overline{\tau}^{+}_{n^{\frac{4}{3} + 2 \epsilon }}$ and the length of the loop-erasure of $\overline{S} [0, \overline{\tau}^{+}_{ n^{\frac{4}{3}+ 2 \epsilon} } ]$ is bounded above by $n^{\frac{5}{3} + 3 \epsilon }$ for large $n$. Since $\overline{T}_{n} \le \overline{\tau}^{+}_{n^{\frac{4}{3} + 2 \epsilon }}$, we see that $\text{len} \big( LE (\overline{S} [0, \overline{T}_{n} ] ) \big)$ is bounded above by $\text{len} \big( LE ( \overline{S} [0, \overline{\tau}^{+}_{ n^{\frac{4}{3}+ 2 \epsilon} } ] )  \big)$. Since $\epsilon > 0$ is an arbitrary positive number, we finish the proof.
\end{proof}

\subsection{Lower bound for $\alpha_{\ell}(2)$}
In this subsection, we will show that $\alpha_{\ell}(2) \ge \frac{5}{3}$ by proving that $\text{len} \big( LE (\overline{S} [0, \overline{T}_{n} ] )$ divided by $n^{\frac{5}{3} - \epsilon }$ goes to infinity for all $\epsilon > 0$ in Proposition \ref{lower-2dim-loop}. The proof is based on the same ideas as in the proof of Proposition \ref{upper-2dim-loop}. We compare $\overline{T}_{n}$ with the first time that $\overline{S}$ exits from a ball as in \eqref{jirousan}. Then we will give a lower bound on the length of the loop-erasure of $\overline{S}$ up to $\overline{\tau}^{+}_{n^{\frac{4}{3} - 2 \epsilon }}$. However, there is an issue to achieve it. Since $\text{len} \overline{S} [0, t_{1} ]$ may be larger than $\text{len} \overline{S} [0, t_{2} ]$ even if $t_{1} < t_{2}$, we are not able to conclude that $\text{len} \big( LE (\overline{S} [0, \overline{T}_{n} ] ) \big)$ is bigger than $\text{len} \big( LE (\overline{S} [0, \overline{\tau}^{+}_{n^{\frac{4}{3} - 2 \epsilon }} ] ) \big)$. In order to deal with this issue, we will consider $LE (\overline{S} [0, \overline{\tau}^{+}_{n^{2}}] )$ up to the first time it exits from $B ( n^{\frac{4}{3} - 2 \epsilon } )$ instead of $LE (\overline{S} [0, \overline{\tau}^{+}_{n^{\frac{4}{3} - 2 \epsilon }} ] )$.

\begin{prop}\label{lower-2dim-loop}
Let $d=2$. For all $\alpha < \frac{5}{3}$,
\begin{equation}
\overline{P} \Big( \lim_{n \to \infty} \frac{ \text{len} \big( LE (\overline{S} [0, \overline{T}_{n} ] ) \big) } { n^{\alpha} } = \infty \Big) =1.
\end{equation}
In particular, $\alpha_{\ell}(2) \ge \frac{5}{3}$.
\end{prop}

 
\begin{proof}
Fix $\epsilon > 0$. Recall that it follows from \eqref{jirousan} that $\overline{\tau}^{+}_{n^{\frac{4}{3} - 2 \epsilon }} \le \overline{T}_{n} \le \overline{\tau}^{+}_{n^{\frac{4}{3} + 2 \epsilon }}$ for large $n$ with probability one. Suppose that $\overline{\tau}^{+}_{n^{\frac{4}{3} - 2 \epsilon }} \le \overline{T}_{n}$. We first show that $ | \overline{S} ( \overline{T}_{n} ) | \ge n^{\frac{4}{3} - 3 \epsilon }$ with high probability. To show it, suppose that $ | \overline{S} ( \overline{T}_{n} ) | \le n^{\frac{4}{3} - 3 \epsilon }$. This implies that $\overline{S} [\overline{\tau}^{+}_{n^{\frac{4}{3} - 2 \epsilon }}, \infty ) \cap B (n^{\frac{4}{3} - 3 \epsilon } ) \neq \emptyset$. However, (1.9) of \cite{S} shows that this return probability is bounded above by $c n^{-\frac{\epsilon}{2}}$. Therefore we have
\begin{equation}\label{dareda}
\overline{P} \big( n^{\frac{4}{3} - 3 \epsilon } < | \overline{S} ( \overline{T}_{n} ) | < n^{\frac{4}{3} + 2 \epsilon } \big) \ge 1- c n^{-\frac{\epsilon}{2}}.
\end{equation}

We are interested in $LE (\overline{S} [0, \overline{\tau}^{+}_{n^{2}}] )$ up to the first time that it exits from $B (n^{\frac{4}{3} - 3 \epsilon })$. Let 
\begin{equation}\label{kaetta}
\overline{u} := \inf \{ k \ | \ LE (\overline{S} [0, \overline{\tau}^{+}_{n^{2}}] ) (k) \in B (n^{\frac{4}{3} - 3 \epsilon })^{c} \}.
\end{equation}
Suppose that $n^{\frac{4}{3} - 3 \epsilon } < | \overline{S} ( \overline{T}_{n} ) | < n^{\frac{4}{3} + 2 \epsilon }$. Then we see that $\overline{S} ( \overline{T}_{n} )$ lies in $LE (\overline{S} [0, \overline{\tau}^{+}_{n^{2}}] )$, and that $\overline{S} ( \overline{T}_{n} )$ appears in $LE (\overline{S} [0, \overline{\tau}^{+}_{n^{2}}] )$ after time $\overline{u}$. So there exists an unique time $t$ such that $LE (\overline{S} [0, \overline{\tau}^{+}_{n^{2}}] ) (t) = \overline{S} ( \overline{T}_{n} )$ with $t > \overline{u}$. Since $\overline{T}_{n}$ is a global cut time, we see that $LE (\overline{S} [0, \overline{\tau}^{+}_{n^{2}}] ) [0, t] = LE (\overline{S} [0, \overline{T}_{n} ])$. Consequently it follows that with probability at least $1- c n^{-\frac{\epsilon}{2}}$,
\begin{equation}\label{kaerunone}
\text{len} \big( LE (\overline{S} [0, \overline{T}_{n} ] ) \big) = t > \overline{u}.
\end{equation}
Thus we need to estimate $\overline{u}$ which was defined as in \eqref{kaetta}.

In order to estimate $\overline{u}$, we will again use tail bounds on the length of LERW derived in \cite{BM}. We are interested in $LE ( S^{2} [0, \tau^{2}_{n^{2}}])$ up to the first time that it exits from $B (n^{\frac{4}{3} - 3 \epsilon })$. Let 
\begin{equation}\label{kaetta-1}
u := \inf \{ k \ | \ LE ( S^{2} [0, \tau^{2}_{n^{2}}]) (k) \in B (n^{\frac{4}{3} - 3 \epsilon })^{c} \}.
\end{equation}
Then Theorem 1.2 of \cite{BM} gives that 
\begin{equation*}
P ( u < n^{\frac{5}{3} - 6 \epsilon} ) \le C e^{-c n^{\epsilon}},
\end{equation*}
for some $0 < c, C < \infty$. Using this, same estimates as in \eqref{kietane} gives that
\begin{equation*}
\overline{P} \big( \overline{u} < n^{\frac{5}{3} - 6 \epsilon} \big) \le C e^{- \frac{c}{2} n^{\epsilon}}.
\end{equation*}
Combining this with \eqref{kaerunone}, we can conclude that with probability at least $1- c n^{-\frac{\epsilon}{2}}$, $\text{len} \big( LE (\overline{S} [0, \overline{T}_{n} ] ) \big)$ is bounded below by $n^{\frac{5}{3} - 6 \epsilon}$. Now we apply the Borel-Cantelli lemma for $n = 2^{k}$ to see that
\begin{equation*}
\text{len} \big( LE (\overline{S} [0, \overline{T}_{2^{k}} ] ) \big) \ge (2^{k})^{\frac{5}{3} - 6 \epsilon} \text{ for large } k \ \text{ a.s.}
\end{equation*}
For a general index $n$, by considering $k$ with $2^{k} \le n < 2^{k+1}$, we see that with probability one, 
\begin{equation*}
\text{len} \big( LE (\overline{S} [0, \overline{T}_{n} ] ) \big) \ge \text{len} \big( LE (\overline{S} [0, \overline{T}_{2^{k}} ] ) \big) \ge (2^{k})^{\frac{5}{3} - 6 \epsilon} \ge c n^{\frac{5}{3} - 6 \epsilon},
\end{equation*}
for large $n$. Since $\epsilon$ is an arbitrary positive number, we finish the proof.
\end{proof}

\section{Estimates on escape probabilities}
From this section we will focus on loop-erased random walks in 3 dimensions. In the rest of the present article, the goal is to establish an analog of Theorem \ref{2-dim} in three dimensions. Namely we want to prove that for $d=3$ 
\begin{equation}\label{3-dim-mokuhyo}
\lim_{n \to \infty} \frac{ \log  \text{len} \big( LE (\overline{S} [0, \overline{T}_{n} ] ) \big) }{ \log n}  = \alpha_{\ell}(3) \text{ a.s.,}
\end{equation}
where $\alpha_{\ell}(3)$ is the exponent as in Theorem \ref{critical exp}. We will prove \eqref{3-dim-mokuhyo} in Section 9. Section 6 -- Section 8 will be devoted to establish various results for LERW in 3 dimensions to show \eqref{3-dim-mokuhyo}. The purpose of this section is to give various relations between escape probabilities on various scales (see Section 6.2 for the escape probabilities).

The proof of Theorem \ref{2-dim} was based on Proposition \ref{upper-2dim-loop} and \ref{lower-2dim-loop}. In order to prove these two propositions, we strongly relied on results of \cite{BM} which give exponential tail bounds on the length of LERW in 2 dimensions. Therefore we need to establish similar tail bounds in 3 dimensions. The key ingredient in \cite{BM} is the probability that a random walk and an independent LERW do not intersect up to the first time that they exit from a large ball, which is referred to as an escape probability (see Section 6.2 for the precise definition of the escape probability). We recall that one of the main step in \cite{BM} is to give bounds on the $k$-th moment of the length of LERW in terms of escape probabilities. Such moment estimates allow to establish the exponential tail bounds on the length of LERW, see (1.5) of \cite{BM}.

Several estimates on the escape probability derived in \cite{Mas} were used to give the tail bounds on the length of LERW in \cite{BM}. In this section, we will establish such estimates on the escape probability in 3 dimensions that will be needed later. We will give various relations between the escape probabilities on various scales in Propositions \ref{up-to-const lew}, \ref{up to const indep1} and \ref{up to const indep2}. These propositions are analogs of Lemma 5.1, Proposition 5.2 and Proposition 5.3 of \cite{Mas}. We point out that the separation lemma (see Theorem 4.7 of \cite{Mas}) was a key result in order to prove these results in \cite{Mas}. The separation lemma (Theorem 4.7 of \cite{Mas}) roughly claims that a random walk and an independent LERW that are conditioned not to intersect are likely to be not very close at their endpoints, which is an analog of Proposition \ref{sep lemma} for a random walk and an independent LERW. Unfortunately, the separation lemma was proved only in 2 dimensions in \cite{Mas}, and to our knowledge it has not been proved in 3 dimensions. 
So we need to prove it in 3 dimensions.

In the next subsection, we will prove the separation lemma in 3 dimensions (see Theorem \ref{sep lem 3dim}). Using this lemma, we will give various relations between the escape probabilities on various scales in Section 6.2.

\subsection{Separation lemma --- SRW v.s. LERW}
As we discussed above, in order to give various relations between the escape probabilities on various scales, we need to prove the separation lemma (see Theorem \ref{sep lem 3dim}). The lemma says that a random walk and an independent LERW that are conditioned not to intersect are likely to be ``well-separated". We start by giving preliminary results to show the separation lemma.

Let ${\cal D} = \{ (x,y,z) \in \mathbb{R}^{3} \ : \ x=1, \ y^{2} + z^{2} \le 1 \}$ and $D_{n} = \partial B (n) \cap \{ rw \ : \ r \ge 0, \ w \in {\cal D} \}$. We write $x_{n} =  (n, 0, 0) $. 

Suppose that we have a random walk conditioned that it exits from a ball without hitting a given set $K$ contained in the ``left" side of the ball. The next proposition says that this conditioned random walk exits from the ``right" side of the ball with positive probability. This is an analog of Proposition 3.5 in 3 dimensions. Claim 3.4 of \cite{SS} gave the proof of the proposition in 3 dimensions.

\begin{prop}\label{Masson}
Let $d=3$. There exist $N$ and $c > 0$ such that for all $n \ge N$, we have the following. Suppose that $K \subset \mathbb{Z}^{3} \setminus B (x_{n} , n )$. Recall that $\sigma_{K} := \inf \{ j \ge 1 \ | \ S(j) \in K \}$. Then, 
\begin{equation}
P \Big( S \big( \tau_{n} \big) \in D_{n} \ \Big| \ \tau_{n} < \sigma_{K} \Big) \ge c,
\end{equation}
where $D_{n}$ stands for the ``right" side of the boundary of $B (n)$ defined as above.
\end{prop}

For a subset $A \subset \mathbb{Z}^{3}$, we write $A^{+} = \{ x = (x_{1}, x_{2}, x_{3} ) \in A \ | \ x_{1} > 0 \}$ and $A^{-} = \{ x = (x_{1}, x_{2}, x_{3} ) \in A \ | \ x_{1} < 0 \}$ for the ``right" and ``left" side of $A$. For $x = (x_{1}, x_{2}, x_{3} ) \in \mathbb{Z}^{3}$, we write $\overline{x} = (-x_{1}, x_{2}, x_{3} ) \in \mathbb{Z}^{3}$ for the reflection of $x$ with respect to the $yz$-plane. We let $\overline{A} = \{ \overline{x} \ | \ x \in A \}$ be the reflection of $A$ with respect to the $yz$-plane. We will need the following lemma which is an analog of Lemma 4.4 of \cite{BM} in 3 dimensions.




\begin{lem}\label{reflect1}
Let $d=3$. Take two subsets $A \subset B \subset \mathbb{Z}^{3}$ satisfying that $B^{+} \subset \overline{B^{-}}$ and $A^{+} \subset \overline{A^{-}}$. Then it follows that for all $x \in B^{-}$
\begin{equation}
P^{x} \big( \tau_{B} < \sigma_{A} \big) \le P^{\overline{x} } \big( \tau_{B} < \sigma_{A} \big).
\end{equation}
\end{lem}

\begin{proof}
The proof is same as the proof of Lemma 4.4 of \cite{BM}. Thus we will give only the sketch of it here. The lemma follows from a reflection argument as follows. Take $x \in B^{-}$. When the event $\tau_{B} < \sigma_{A}$ occurs, Either of the following two cases must occur: 
\begin{itemize}
\item The random walk exits from $B$ without hitting $A$ and the $yz$-plane.

\item The random walk hits the $yz$-plane before hitting $A$ and exiting $B$. Then it exits from $B$ without hitting $A$.
\end{itemize} 
Since we assume that $B^{+} \subset \overline{B^{-}}$ and $A^{+} \subset \overline{A^{-}}$, if we consider the reflection of the random walk path in the first case, the reflected path starts from $\overline{x}$ and it exits from $B$ without hitting $A$ and the $yz$-plane. For the second case, we consider the reflection of the random walk path up to the first time that it hits the $yz$-plane. Then the reflected path starts from $\overline{x}$ and it hits $yz$-plane without hitting $A$. After hitting $yz$-plane, it exits from $B$ without hitting $A$. Therefore the reflected path for both cases will be a random walk path started at $\overline{x}$ which satisfies $\tau_{B} < \sigma_{A}$. So we get the lemma.
\end{proof}

In order to state the next lemma, we need the following definition.

\begin{dfn}\label{set-near-cube}
Take integers $m, n , N$ with $\sqrt{3} m + n \le N$. We set $A_{m} := [-m, m]^{3}$ for the cube of side length $2m$ centered at the origin. We take a point $x$ lying in a face of $A_{m}$. We write $\ell$ for the infinite half line started at $x$ which lies in $A_{m}^{c}$ and is orthogonal to the face of $A_{m}$ containing $x$ (we choose one such faces arbitrarily if $x$ lies in a edge of $A_{m}$). We write $y$ for the unique point which lies in $\ell$ and satisfies $|x-y| = \frac{n}{2} $. Then we let $A_{n} (x):= \prod_{i=1}^{3} [y_{i} - \frac{n}{4}, y_{i} + \frac{n}{4}]$ be the cube of length $\frac{n}{2}$ centered at $y$. The assumption $\sqrt{3} m + n \le N$ ensures that $B(x, n) \subset B (N)$.
\end{dfn}

When we relate a random walk conditioned not to hit a given set to an usual simple random walk, the next lemma is used many times.

\begin{lem}\label{reflect2}
Let $d=3$. We take integers $m, n , N$ with $\sqrt{3} m + n \le N$. We suppose that $A_{m}$ is the cube as in Definition \ref{set-near-cube} and take a point $x$ in the face of $A_{m}$. Let $K \subset A_{m}$ be a subset of the cube. We also suppose that $A_{n} (x)$ is the cube as in Definition \ref{set-near-cube}. Then there exists a universal constant $C < \infty$ such that
\begin{equation}\label{reflection}
\max_{z \in \partial B (x, \frac{n}{8})} P^{z} ( \tau_{N} < \sigma_{K} ) \le C P^{w} ( \tau_{N} < \sigma_{K} ),
\end{equation}
for all $w \in A_{n}(x)$.
\end{lem}

\begin{proof}
Recall that the half line $\ell$ started at $x$ was defined as in Definition \ref{set-near-cube}. We write $y'$ for the unique point lying on $\ell$ such that $|x-y'| = \frac{n}{4}$. 
Let $\pi_{1}$ be the plane containing the middle point of $x$ and $y'$ which is orthogonal to $\ell$. Applying Lemma \ref{reflect1} to the plane $\pi_{1}$, we see that
\begin{equation*}
\max_{z \in \partial B (x, \frac{n}{8})} P^{z} ( \tau_{N} < \sigma_{K} ) \le \max_{z \in \partial B (y', \frac{n}{8})} P^{z} ( \tau_{N} < \sigma_{K} ). 
\end{equation*}
The discrete Harnack principle (see Theorem 1.7.6 of \cite{Law b}) gives that there exists a universal constant $C < \infty$ such that 
\begin{equation*}
\max_{z \in \partial B (y', \frac{n}{8})} P^{z} ( \tau_{N} < \sigma_{K} ) \le C P^{w} ( \tau_{N} < \sigma_{K} ),
\end{equation*}
for all $w \in A_{n}(x)$. So we finish the proof.
\end{proof}

Consider two independent simple random walks $S^{1}$ and $S^{2}$ in $\mathbb{Z}^{3}$. We are interested in the conditional probability that the distance between $S^{1} (\tau^{1}_{n})$ and $LE (S^{2} [0, \tau^{2}_{n}] )$ and the distance between $S^{2} (\tau^{2}_{n})$ and $S^{1} [0, \tau^{1}_{n}]$ is bounded below by $c n$ conditioned that $S^{1} [1, \tau^{1}_{n}]$ and $LE (S^{2} [0, \tau^{2}_{n}] )$ do not intersect. With this in mind, let 
\begin{equation}\label{tabaco}
A^{n} := \{  S^{1} [1, \tau^{1}_{n}] \cap LE (S^{2} [0, \tau^{2}_{n}] ) = \emptyset \}
\end{equation}
be the event that a simple random walk and an independent LERW do not intersect. We also consider the infinite LERW as follows. Since $S^{2}$ is transient, we may consider the loop-erasure of $S^{2}[0, \infty )$. So we let $\gamma^{\infty} := LE ( S^{2}[0, \infty ) )$ be its loop-erasure and we call it the infinite LERW. We set 
\begin{equation}\label{san-bon}
\tau^{\infty}_{n} = \inf \{ j \ | \ \gamma^{\infty} (j) \notin B(n) \}
\end{equation}
for the first time that the infinite LERW exits from $B (n)$. We denote the event that $S^{1}$ and $\gamma^{\infty}$ do not intersect up to the first time that they exit from $B (n )$ by 
\begin{equation}\label{higanai}
A^{n}_{\infty} := \{  S^{1} [1, \tau^{1}_{n}] \cap \gamma^{\infty} [0, \tau^{\infty}_{n} ] = \emptyset \}.
\end{equation}

We choose a ``separation" event as in Proposition \ref{sep lemma}. Recall that $I (r)$ and $I^{\prime}(r)$ were defined as in \eqref{migihidari}. With \eqref{sep} in mind, we define
\begin{align}
&\overline{\textsf{Sep}}(n) = \Big\{ S^{1}[0, \tau^{1}_{n}] \subset B \big( \frac{3n}{4} \big) \cup I \big( \frac{2n}{3} \big)  \Big\} \cap \Big\{ LE (S^{2} [0, \tau^{2}_{n}] ) \subset B \big( \frac{3n}{4} \big) \cup I^{\prime} \big( \frac{2n}{3} \big)  \Big\} \label{sep-for-lerwtachi-1} \\
&\overline{\textsf{Sep}}_{\infty} (n) = \Big\{ S^{1}[0, \tau^{1}_{n}] \subset B \big( \frac{3n}{4} \big) \cup I \big( \frac{2n}{3} \big)  \Big\} \cap \Big\{ \gamma^{\infty} [0, \tau^{\infty}_{n} ] \subset B \big( \frac{3n}{4} \big) \cup I^{\prime} \big( \frac{2n}{3} \big)  \Big\} \label{sep-for-lerwtachi-2}.
\end{align}
Namely, $\overline{\textsf{Sep}}(n)$ stands for the event that $S^{1}[0, \tau^{1}_{n}]$ and $LE (S^{2} [0, \tau^{2}_{n}])$ are well-separated. $\overline{\textsf{Sep}}_{\infty} (n)$ stands for the event that $S^{1}[0, \tau^{1}_{n}]$ and $\gamma^{\infty} [0, \tau^{\infty}_{n} ]$ are well-separated. Then the separation lemma for SRW and LERW states the following.



\begin{thm}(Separation Lemma)\label{sep lem 3dim}
Let $d=3$. There exists a constants $c > 0$ such that for all $n$, 
\begin{align}
&P ( \overline{\textsf{Sep}}(n) \ | \ A^{n} ) \ge c, \label{Sep-3dim-1} \\
&P ( \overline{\textsf{Sep}}_{\infty} (n) \ | \ A^{n}_{\infty} ) \ge c \label{Sep-3dim-2}.
\end{align}
\end{thm}

\begin{proof}
We will prove only \eqref{Sep-3dim-1}. The second inequality \eqref{Sep-3dim-2} can be proved similarly. In the proof of \eqref{Sep-3dim-1}, we will use the same ideas based on the induction as in the proof of Proposition 2.1 in \cite{S}. 

We let  
\begin{equation*}
\Gamma' (n) = \{ \overline{\gamma} = (\gamma^{1}, \gamma^{2}) \in \Gamma (\frac{n}{2}) \ : \ \gamma^{2} \text{ is a simple path } \}
\end{equation*}
be the set of pairs of $\overline{\gamma} = (\gamma^{1}, \gamma^{2})$ such that $\gamma^{1}$ and $\gamma^{2}$ do not intersect, and that $\gamma^{2}$ is a simple path (see Section 1.4 for $\Gamma (n)$). Take $\overline{\gamma} = (\gamma^{1}, \gamma^{2}) \in \Gamma' (n)$. We write $w^{i} = \gamma^{i} (\text{len} \gamma^{i})$ for the endpoint of $\gamma^{i}$. Note that $w^{i}$ lies in the boundary of $B (\frac{n}{2} )$. We consider a simple random walk $S^{3}$ started at $w^{1}$ and an independent random walk $X$ started at $w^{2}$ which is conditioned that $X[1, \tau^{X}_{n} ] \cap \gamma^{2} = \emptyset$, where we write $\tau^{X}_{n}$ for the first time that $X$ exits from $B (n )$. We set 
\begin{equation*}
\eta = LE ( X[0, \tau^{X}_{n} ] )
\end{equation*}
for the loop-erasure of $x$ up to $\tau^{X}_{n}$. With \eqref{tabaco} in mind, we let 
\begin{eqnarray}\label{tabaco-2}
A^{n} ( \overline {\gamma} )=\left\{ \begin{array}{ll}
S_{3}[0, \tau^{3}_{n} ] \cap \gamma^{2} = \emptyset , \\
\eta \cap \gamma^{1} = \emptyset , \\
S_{3}[0, \tau^{3}_{n} ] \cap \eta = \emptyset  
\end{array}
\right\}.
\end{eqnarray}
We are interested in the conditional probability that $S^{3}$ and $\eta$ are well-separated conditioned on $A^{n} ( \overline {\gamma} )$. So we denote the separation event for $S^{3}$ and $\eta$ by
\begin{equation*}
\overline{\textsf{Sep}}(n, \overline{\gamma}) = \Big\{ S_{3}[0, \tau^{3}_{n} ] \subset B \big( \frac{3n}{4} \big) \cup I \big( \frac{2n}{3} \big)  \Big\} \cap \Big\{ \eta \subset B \big( \frac{3n}{4} \big) \cup I^{\prime} \big( \frac{2n}{3} \big)  \Big\}.
\end{equation*}

In order to prove \eqref{Sep-3dim-1}, it suffices to show that there exists a $c > 0$ such that for all $n$ and $\overline{\gamma} = (\gamma^{1}, \gamma^{2}) \in \Gamma' (n)$,
\begin{equation}\label{seminar-darui}
P^{w^{1}, w^{2}} \big( \overline{\textsf{Sep}}(n, \overline{\gamma}) \ | \ A^{n} ( \overline {\gamma} ) \big) \ge c.
\end{equation}
In order to see that \eqref{Sep-3dim-1} follows from \eqref{seminar-darui}, we set $u_{1} = \inf \{ j \ | \ LE (S^{2} [0, \tau^{2}_{n}] ) (j) \notin B (\frac{n}{2} ) \}$ and $u_{2} = \text{len} \big( LE (S^{2} [0, \tau^{2}_{n}] ) \big)$. Then by the strong Markov property of $S^{1}$ and the domain Markov property of LERW (see Proposition \ref{DMP}), we see that
\begin{align}\label{hayakusi}
&P \Big( \big( S^{1} [0, \tau^{1}_{\frac{n}{2}}], LE (S^{2} [0, \tau^{2}_{n}] ) [0, u_{1}] \big) = \overline{\gamma}, \ \overline{\textsf{Sep}}(n), \ A^{n} \Big) \notag \\
&= P \Big( \big( S^{1} [0, \tau^{1}_{\frac{n}{2}}], LE (S^{2} [0, \tau^{2}_{n}] ) [0, u_{1}] \big) = \overline{\gamma} \Big) P^{w^{1}, w^{2}} \Big( \overline{\textsf{Sep}}(n, \overline{\gamma}), \ A^{n} ( \overline {\gamma} ) \Big).
\end{align}
However, by \eqref{seminar-darui}, the left hand side of \eqref{hayakusi} is bounded below by 
\begin{equation*}
c P \Big( \big( S^{1} [0, \tau^{1}_{\frac{n}{2}}], LE (S^{2} [0, \tau^{2}_{n}] ) [0, u_{1}] \big) = \overline{\gamma} \Big) P^{w^{1}, w^{2}} \Big(  A^{n} ( \overline {\gamma} ) \Big),
\end{equation*}
which is, by the strong Markov property and the domain Markov property again, equal to 
\begin{equation*}
c P \Big( \big( S^{1} [0, \tau^{1}_{\frac{n}{2}}], LE (S^{2} [0, \tau^{2}_{n}] ) [0, u_{1}] \big) = \overline{\gamma}, \ A^{n} \Big).
\end{equation*}
By taking sum for $\overline{\gamma} = (\gamma^{1}, \gamma^{2}) \in \Gamma' (n)$, we get \eqref{Sep-3dim-1}. 

We will give a stronger estimate than \eqref{seminar-darui} as follows. We write $\Gamma"(n)$ for the set of pairs $\overline{\gamma} = (\gamma^{1}, \gamma^{2})$ such that the following conditions are fulfilled:
\begin{itemize}
\item $\gamma^{1}$ is a path started at the origin. $\gamma^{1}$ lies in $B( \frac{n}{2})$ except its endpoint. The endpoint $\gamma^{1} (\text{len} \gamma^{1} )$ lies in $\partial B( \frac{n}{2})$.

\item $\gamma^{2}$ is a simple path started at the origin. $\gamma^{2}$ lies in $B( \frac{n}{2})$ except its endpoint. The endpoint $\gamma^{2} (\text{len} \gamma^{2} )$ lies in $\partial B( \frac{n}{2})$.

\item $\gamma^{1} (\text{len} \gamma^{1} ) \neq \gamma^{2} (\text{len} \gamma^{2} )$.
\end{itemize}
Clearly, $\Gamma' (n) \subset \Gamma"(n)$. We will show that there exists a $c > 0$ such that for all $n$ and $\overline{\gamma} = (\gamma^{1}, \gamma^{2}) \in \Gamma" (n)$,
\begin{equation}\label{asitayasumi}
P^{w^{1}, w^{2}} \big( \overline{\textsf{Sep}}(n, \overline{\gamma}) \ | \ A^{n} ( \overline {\gamma} ) \big) \ge c,
\end{equation}
which gives \eqref{seminar-darui}. 

We will prove \eqref{asitayasumi} by induction. To achieve it, let
\begin{equation*}
u_{k} = \textstyle\sum\limits _{j= k} ^{\infty } j^{2} 2^{-j}.
\end{equation*}
We take $N$ sufficiently Large so that $u_{N} \le \frac{1}{8}$. For $\overline{\gamma} = (\gamma^{1}, \gamma^{2}) \in \Gamma" (n)$ with $w^{i} = \gamma^{i} ( \text{len}\gamma ^{i} )$, we set $D( \overline{\gamma} ) = \text{dist} ( w^{1} , \gamma ^{2} ) \wedge \text{dist} ( w^{2} , \gamma ^{1} )$. The definition of $\Gamma" (n)$ gives that $D( \overline{\gamma} ) \ge 1 $ for all $\overline{\gamma} \in \Gamma" (n)$. For $k \ge N$, we let $h_{k}$ be the infimum of 
\begin{equation}\label{cheese-kue}
\frac{ P^{w^{1}, w^{2}} \big( \overline{\textsf{Sep}}(n, \overline{\gamma}), \ A^{n} ( \overline {\gamma} ) \big) }{ P^{w^{1}, w^{2}} \big( A^{n} ( \overline {\gamma} ) \big) },
\end{equation}
where the infimum is over $\frac{n}{2} \ge 2^{k-1}$; $0 \le r \le u_{k}$; and all $\overline{\gamma} = (\gamma ^{1} , \gamma ^{2} ) \in \Gamma" ((1+r) n)$ such that $\frac{ D( \overline{\gamma} ) }{\frac{n}{2}} \ge 2^{-k}$. Then in order to prove \eqref{asitayasumi} it suffices to show that
\begin{equation}\label{deduce}
\inf _{k \ge N} h_{k} > 0.
\end{equation}
Indeed, suppose that $\overline{\gamma} \in \Gamma" (n)$ with $\frac{n}{2} \ge 2^{N-1}$. Consider the unique integer $k$ such that $2^{k-1} \le \frac{n}{2} < 2^{k}$. Then we see that $k \ge N$ and $\frac{ D( \overline{\gamma} ) }{\frac{n}{2}} \ge 2^{-k}$ (we choose $r=0$ in the definition of $h_{k}$). Therefore the ratio of \eqref{cheese-kue} for this $\overline{\gamma}$ is bounded below by $\inf _{k \ge N} h_{k}$. For $\overline{\gamma} \in \Gamma" (n)$ with $\frac{n}{2} < 2^{N-1}$, it is easy to see that the ratio of \eqref{cheese-kue} can be bounded below by some universal constant uniformly. So \eqref{asitayasumi} follows from \eqref{deduce}.

We will prove \eqref{deduce} by showing that $h_{k} > 0$ for each $k \ge N$, and that there exists a summable sequence $\delta _{k} < 1$ such that
\begin{equation}\label{induction}
h_{k+1} \ge h_{k}(1-\delta _{k} ).
\end{equation}
To achieve it, we start by proving that there exist $0 < c, \delta < \infty $ such that
\begin{equation}\label{hn}
h_{k} \ge c 2^{-\delta k}.
\end{equation}
We take $\overline{\gamma} = (\gamma ^{1} , \gamma ^{2} ) \in \Gamma" ((1+r) n)$ with $\frac{n}{2} \ge 2^{k-1}$, $0 \le r \le u_{k}$, and $\frac{ D( \overline{\gamma} ) }{\frac{n}{2}} \ge 2^{-k}$. We write $w^{i}$ for the endpoint of $\gamma^{i}$ again. To show \eqref{hn}, we consider two cones $O_{1},O_{2}$ starting from $z_{1}, z_{2} \in \mathbb{R}^{3}$ as follows. Suppose $U$ is a relatively open subset of $\{ z \in \mathbb{R} ^{3} : |z|=1 \}$. We let $O$ denote the corresponding cone 
\begin{equation}\label{cone}
O= \{ rw : r > 0, \ w \in U \}.
\end{equation}
Then it is easy to see that we can find two cones $O_{1},O_{2}$ as in \eqref{cone} and vertices $z_{1}, z_{2} \in \mathbb{R}^{d}$ such that the following hold for $i=1, 2$:
\begin{align*}
&\text{(a) }\frac{D( \overline{\gamma} )}{100} \le |z_{j} - w^{i}| \le \frac{D( \overline{\gamma} )}{20}. \\
&\text{(b) }w^{i} \in O_{i} + z_{i} \text{ and } \frac{D( \overline{\gamma} )}{100} \le \text{dist} \big( w^{i}, \partial (z_{i} + O_{j}) \big) \le \frac{D( \overline{\gamma} )}{20}. \\
&\text{(c) }(O_{i}+z_{i}) \cap B (\frac{n}{2}) \subset B ( w^{i} , \frac{D( \overline{\gamma} )}{10} ) \\
&\text{(d) }\text{If } V_{i} = (O_{i}+z_{i}) \cap B \big( \frac{9 n}{14} \big)^{c}, \text{ then dist}( V_{i}, (O_{3 -i}+z_{3 - i}) ) \ge \frac{n}{1000}.
\end{align*}
We leave it to the reader to see that such cones can be found. Using the strong Markov property, we see that there exist $0 < c, \delta < \infty $ such that 
\begin{equation}\label{kusocity-1}
P^{w^{1}} \big( S^{3} [ 0, \tau^{3}  ( \frac{ 9 n}{14} ) ] \subset O_{1} + z_{1} \big) \ge c 2^{-\delta k}.
\end{equation}
Recall that $\eta = LE ( X[0, \tau^{X}_{n} ] )$ is the loop-erasure of the random walk $X$ conditioned to avoid $\gamma^{2}$. We let $\tau^{\eta}_{m} := \inf \{ j \ | \ \eta (j) \notin B (m ) \}$ be the first time that $\eta$ exits from $B (m )$. We want to show that there exist $0 < c, \delta < \infty $ such that 
\begin{equation}\label{kusocity}
P^{w^{2}} \big( X [ 0, \tau^{X}  ( \frac{ 9 n}{14} ) ] \subset O_{2} + z_{2} \big) \ge c 2^{-\delta k}.
\end{equation}
The definition of $X$ gives that 
\begin{equation*}
P^{w^{2}} \big( X [ 0, \tau^{X}  ( \frac{ 9 n}{14} ) ] \subset O_{2} + z_{2} \big) \ge \frac{P^{w^{2}} \big( S^{4} [0, \tau^{4}  ( \frac{ 9 n}{14} ) ] \subset O_{2} + z_{2}, \ S^{4} [1, \tau^{4}_{n} ] \cap \gamma^{2} = \emptyset \big) }{P^{w^{2}} \big( S^{4} [1, \tau^{4}_{\frac{D( \overline{\gamma} )}{1000}} ] \cap \gamma^{2} = \emptyset \big) }.
\end{equation*}
However, by using Proposition \ref{Masson}, we see that the right hand side of the inequality above is bounded below by
\begin{equation*}
c \min_{y} P^{y} \Big( S^{4} [0, \tau^{4}  ( \frac{ 9 n}{14} ) ] \subset O_{2} + z_{2}, \ S^{4} [0, \tau^{4}_{n} ] \cap B \big( (1+ r) n \big) = \emptyset \Big),
\end{equation*}
where the minimum is over $y$ such that $\text{dist} (y, O_{2} + z_{2} ) \ge \frac{D( \overline{\gamma} )}{2000}$ and $\text{dist} (y, B \big( (1+ r) n \big) ) \ge \frac{D( \overline{\gamma} )}{2000}$. It is easy to see that this minimum is bounded below by $c 2^{-\delta k}$ for some $0 < c, \delta < \infty $. So we get \eqref{kusocity}. Once $S^{3}$ and $X$ lie in cones as in \eqref{kusocity-1} and \eqref{kusocity}, with positive probability we can attach paths to $S^{3} [ 0, \tau^{3}  ( \frac{ 9 n}{14} ) ]$ and $X [ 0, \tau^{X}  ( \frac{ 9 n}{14} ) ]$ so that $\overline{\textsf{Sep}}(n, \overline{\gamma})$ and $ A^{n} ( \overline {\gamma} )$ are fulfilled. Thus we have
\begin{equation*}
P^{w^{1}, w^{2}} \big( \overline{\textsf{Sep}}(n, \overline{\gamma}), \ A^{n} ( \overline {\gamma} ) \big) \ge c 2^{-\delta k},
\end{equation*}
which gives \eqref{hn}.

We will next prove \eqref{induction}. Suppose that $\overline{\gamma} = (\gamma ^{1} , \gamma ^{2} ) \in \Gamma" ((1+r) n)$ satisfies $\frac{n}{2} \ge 2^{k}$, $0 \le r \le u_{k+1}$, and $\frac{ D( \overline{\gamma} ) }{\frac{n}{2}} \ge 2^{-k-1}$. We write $w^{i}$ for the endpoint of $\gamma^{i}$ again. We define a sequence of balls $\{ B^{j} \} _{j \ge 0}$ by
\begin{equation*}
B^{j} = B ( a_{j} ),
\end{equation*}
where $a_{j} = (1+r) \frac{n}{2} + 4j 2^{-k} n$. We let 
\begin{equation*}
\rho ^{\prime} = \inf \Big\{ j : \text{ dist } \Big( S^{3} ( \tau ^{3}_{ a_{j} } ),  \eta [0, \tau^{\eta}_{a_{j}}] \cup  \gamma ^{2} \Big) \wedge \text{ dist} \Big( \eta (\tau^{\eta}_{a_{j}} ),  S^{3} [0, \tau ^{3}_{ a_{j} } ] \cup \gamma^{1} \Big) \ge 2^{-k} n \Big\}
\end{equation*}
be the first index $j$ such that both endpoints of $S^{3}$ and $\eta$ up to the first time they exit from $B ( a_{j} )$ have a distance $2^{-k} n$ from the other path. We write $\rho = \rho ^{\prime} \wedge \frac{k^{2}}{4}$. We set 
\begin{equation*}
D_{j} = \text{ dist } \Big( S^{3} ( \tau ^{3}_{ a_{j} } ),  \eta [0, \tau^{\eta}_{a_{j}}] \cup  \gamma ^{2} \Big) \wedge \text{ dist} \Big( \eta (\tau^{\eta}_{a_{j}} ),  S^{3} [0, \tau ^{3}_{ a_{j} } ] \cup \gamma^{1} \Big)
\end{equation*}
for the distance from the endpoints of $S^{3}$ and $\eta$ up to the first time they exit from $B ( a_{j} )$ to the other path. We will show that there is a universal constant $p > 0$ such that conditioned on $S^{3} [0, \tau ^{3}_{ a_{j} } ] $ and $ \eta [0, \tau^{\eta}_{a_{j}}]$, the conditional probability that $D_{j+1} \ge 2^{-k} n$ is at least $p$. To show it, take two paths $\lambda^{1}$ and $\lambda^{2}$ such that $\lambda^{i} (0) = w^{i} $, $\lambda^{i}$ lies in $B ( a_{j} )$ except its endpoint, and the endpoint of $\lambda^{i}$ lies in $\partial B ( a_{j} )$. We denote the endpoint of $\lambda^{i}$ by $v^{i}$. We are interested in the conditional probability that $D_{j+1} \ge 2^{-k} n$ conditioned that $S^{3} [0, \tau ^{3}_{ a_{j} } ] = \lambda^{1}$ and $ \eta [0, \tau^{\eta}_{a_{j}}] = \lambda^{2}$. Under this conditioning, by the strong Markov property, the conditional law of $S^{3}$ after time $\tau ^{3}_{ a_{j} }$ is just the law of a simple random walk started at $v^{1}$. For $\eta$, using the domain Markov property (see Proposition \ref{DMP}), conditioned that $ \eta [0, \tau^{\eta}_{a_{j}}] = \lambda^{2}$, the conditional law of $\eta$ after time $\tau^{\eta}_{a_{j}}$ is same as the law of the loop-erasure of a random walk $Y$ up to the first time that it exits from $B( n)$ conditioned that $Y [1, \tau^{Y}_{n} ] \cap (\gamma^{2} + \lambda^{2} ) = \emptyset$ (we write $\tau^{Y}_{n}$ for the first time that $Y$ exits from $B (n )$). We denote the loop-erasure of $Y [0, \tau^{Y}_{n} ]$ by $\eta'$. We attach $S^{3} [0, \tau ^{3}_{ a_{j+ 1} } ] $ (we assume $S^{3} (0) = v^{1}$ here) and $\eta' [0, \tau^{\eta'}_{ a_{j+ 1}} ]$ to $\gamma^{1} + \lambda^{1}$ and $\gamma^{2} + \lambda^{2}$ respectively in two cones as in \eqref{kusocity-1} and \eqref{kusocity}. By choosing those cones suitably and using Proposition \ref{Masson} as in \eqref{kusocity}, we see that there exists a universal constant $p > 0$ such that for all $\lambda^{i}$ as above,
\begin{equation}\label{perla}
P^{w^{1}, w^{2}} \Big( D_{j+1} \ge 2^{-k} n \ \big| \ S^{3} [0, \tau ^{3}_{ a_{j} } ] = \lambda^{1}, \ \eta [0, \tau^{\eta}_{a_{j}}] = \lambda^{2} \Big) \ge p.
\end{equation}
Using \eqref{perla} $\frac{k^{2}}{4}$ times, we see that there exist $0 < c, \delta < \infty$ such that 
\begin{equation}\label{iterate}
P^{w^{1},w^{2}} ( \rho = \frac{k^{2}}{4} ) \le c 2^{-\delta k^{2} }.
\end{equation}
On the event $\rho < \frac{k^{2}}{4}$, we have $D_{\rho} \ge 2^{-k} n$. The definition of $h_{k}$ gives that 
\begin{align*}
P^{w^{1},w^{2}}  \big( \overline{\textsf{Sep}}(n, \overline{\gamma}), \ A^{n} ( \overline {\gamma} ) \big) &\ge P^{w^{1},w^{2}}  \big( \overline{\textsf{Sep}}(n, \overline{\gamma}), \ A^{n} ( \overline {\gamma} ), \ \{ \rho < \frac{k^{2}}{4} \} \big) \\ 
&\ge h_{k} P^{w^{1},w^{2}}  \big( A^{n} ( \overline {\gamma} ), \ \{ \rho < \frac{k^{2}}{4} \} \big).
\end{align*}
However, \eqref{hn} and \eqref{iterate} imply that
\begin{equation*}
P^{w^{1},w^{2}}  \big( A^{n} ( \overline {\gamma} ), \ \{ \rho < \frac{k^{2}}{4} \} \big) \ge P^{w^{1},w^{2}}  \big( A^{n} ( \overline {\gamma} ) \big) - c 2^{-\delta k^{2} } \ge P^{w^{1},w^{2}} \big( A^{n} ( \overline {\gamma} ) \big) \big( 1- c 2^{-\delta k^{2} + \delta k } \big).
\end{equation*}
Therefore, \eqref{induction} follows with $\delta _{k} = c 2^{-\delta k^{2} + \delta k }$ and we finish the proof.
\end{proof}

As in Theorem 4.10 of \cite{Mas}, using a similar technique, one can prove a ``reverse" separation lemma as follows. Let $Z$ be a random walk started uniformly on $\partial B_{n}$ and conditioned to hit 0 before hitting the boundary of $B_{n}$. Let $W$ be the time reversal of $LE (S [0, \tau_{n}] )$ which is independent of $Z$. Note that $W$ is a process starting from $\partial B_{n}$ and its endpoint is the origin. We write $\sigma^{Z}_{k}$ for the first time that $Z$ hits $B (k )$ and define $\sigma^{W}_{k}$ similarly. For $k \le n$, we define the event $A (k)$ by
\begin{equation*}
A (k) = \{ Z [0, \sigma^{Z}_{k} ] \cap W [0, \sigma^{W}_{k} ] = \emptyset \}.
\end{equation*}
We are interested in the distance defined by
\begin{equation*}
D (k) = \min \{ \text{dist} ( Z( \sigma^{Z}_{k} ) , W [0, \sigma^{W}_{k} ] ), \text{dist} ( W (\sigma^{W}_{k}) , Z [0, \sigma^{Z}_{k} ] ) \}.
\end{equation*}

The next theorem says that the time reverse of a simple random walk and the time reverse of LERW that are conditioned not to intersect are likely to be not very close at their endpoints, which is referred to as the reverse separation lemma in \cite{Mas}.

\begin{thm}(Reverse Separation Lemma)\label{rev sep}
Let $d=3$. There exists a $c > 0$ such that for all $n$
\begin{equation}
P \Big( D (k) \ge c k \ \Big| \ A (k) \Big) \ge c.
\end{equation}
\end{thm}

\subsection{Escape probabilities}
In this subsection, we will study the probability that a simple random walk and an independent LERW do not intersect up to the first time that they exit from a large ball. This probability is referred to as an escape probability in \cite{Mas} and \cite{BM}. In order to establish exponential tail bounds on the length of LERW, escape probabilities are key tools and those were used to give a bound on the $k$-th moment of the length of LERW for $d=2$ in \cite{BM}, see (1.5) of \cite{BM}. Several estimates on escape probabilities derived in \cite{Mas} were needed to give such higher moment estimates of the length of LERW in \cite{BM}. The purpose of this subsection is to establish those estimates on escape probabilities on various scales in 3 dimensions (see Proposition \ref{up-to-const lew}, Proposition \ref{up to const indep1} and Proposition \ref{up to const indep2}). The separation lemmas (Theorem \ref{sep lem 3dim} and Theorem \ref{rev sep}) allow to achieve them.

In order to define escape probabilities, we start by giving some definitions. We consider a path $\lambda$ and a point $z \in \mathbb{Z}^{3}$. Take two integers $m \le n$. We write $\eta^{1}_{z, m} (\lambda )$ for the path $\lambda$ up to the first time that $\lambda$ exits from $B (z, m)$. Namely if we let $u = \inf \{ j \ | \ \lambda (j) \notin B (z, m) \}$ then $\eta^{1}_{z, m} (\lambda ) = \lambda [0, u]$. We write $\eta^{1}_{m} (\lambda )$ for $\eta^{1}_{0, m} (\lambda )$. We set $s = \inf \{ j \ | \ \lambda (j) \notin B (z, n) \}$ for the first time $\lambda$ exits from $B (z, n)$ and set $t = \sup \{ j \le s \ | \ \lambda (j) \in B (z, m ) \}$ for its last visit to $B (z, m )$ up to time $s$. We let $\eta^{2}_{z, m, n} ( \lambda ) = \lambda [t, s]$ be the path $\lambda$ between the last visit to $B (z, m )$ and the first time that it exits from $B (z, n)$. Again we write $\eta^{2}_{m, n} ( \lambda )$ for $\eta^{2}_{0, m, n} ( \lambda )$ (see Figure 3 for $\eta^{1}$ and $\eta^{2}$).

Now we define escape probabilities as follows. Suppose that $S^{1}$ and $S^{2}$ are independent simple random walks in $\mathbb{Z}^{3}$. Recall that $A^{n}$ is the event that $S^{1}$ up to the first time that $S^{1}$ exits from $B (n)$ and the loop-erasure of $S^{2}$ up to the first time it exits from $B (n)$ do not intersect, see \eqref{tabaco}. We also recall that we write $\gamma^{\infty}$ for the infinite loop-erased random walk and that $A^{n}_{\infty}$ stands for the event that $S^{1}$ up to the first time that $S^{1}$ exits from $B (n)$ and the infinite LERW up to the first time it exits from $B (n )$ do not intersect, see \eqref{higanai}. We set
\begin{equation*}
Es (n) = P ( A^{n} ),
\end{equation*}
and set 
\begin{equation*}
Es^{\infty} (n) = P ( A^{n}_{\infty} ).
\end{equation*}
We take two integers $m \le n$. We are also interested in the probability that $S^{1}$ up to $\tau^{1}_{n}$ and the loop-erasure of $S^{2}$ up to $\tau^{2}_{n}$ from its last visit to $B (m )$ do not intersect. Namely, we let 
\begin{equation*}
Es (m, n) = P \Big( S^{1} [1, \tau^{1}_{n}] \cap \eta^{2}_{m, n} \big( LE ( S^{2} [0, \tau^{2}_{n} ] ) \big) = \emptyset \Big).
\end{equation*}

We will give various relations between the escape probabilities on various scales. We start by proving the following proposition which says that $Es (n)$ is comparable to $Es (4n)$ and $Es^{\infty} (n)$. This proposition is an analog of Lemma 5.1 of \cite{Mas} in 3 dimensions.

\begin{figure}
\begin{center}

 \includegraphics[width=8cm]{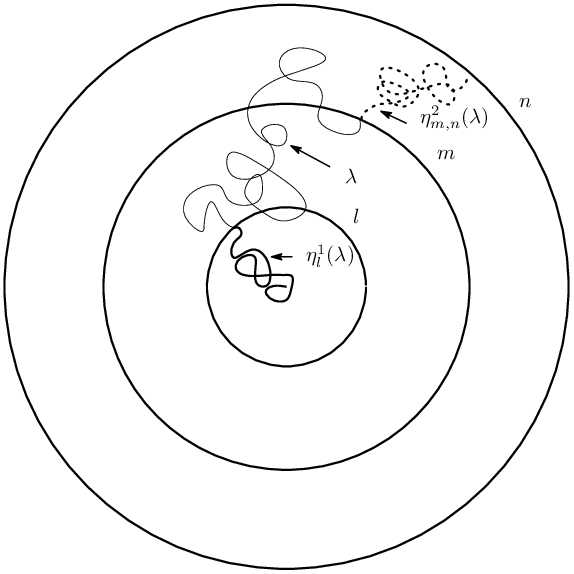}

\end{center}

 \caption{A thick curve is $\eta^{1}_{l} (\lambda )$. A thick dotted curve is $\eta^{2}_{m,n} ( \lambda )$. }

 \end{figure}

\begin{prop}\label{up-to-const lew}
Let $d=3$. Then we have
\begin{equation}
Es^{\infty} (n) \asymp Es (4n) \asymp Es^{\infty} (4n).
\end{equation}
\end{prop}

\begin{proof}
We will first show that $Es^{\infty} (n) \asymp Es^{\infty} (4n)$. The definition of $Es^{\infty} (n)$ immediately gives that $Es^{\infty} (n) \ge Es^{\infty} (4n)$. So we need to show that $Es^{\infty} (4n) \ge c Es^{\infty} (n)$ for some $c > 0$. Recall that we write $\gamma^{\infty}$ for the infinite loop-erased random walk and write $\tau^{\infty }_{n}$ for the first time that it exits from $B (n )$ (see \eqref{san-bon}). We also recall that $\overline{\textsf{Sep}}_{\infty} (n) $ stands for the event that $S^{1}[0, \tau^{1}_{n}]$ and $\gamma^{\infty} [0, \tau^{\infty}_{n} ]$ are ``well-separated" (see \eqref{sep-for-lerwtachi-2} for $\overline{\textsf{Sep}}_{\infty} (n) $). By Theorem \ref{sep lem 3dim}, we see that 
\begin{equation*}
P \big( \overline{\textsf{Sep}}_{\infty} (n), \ A^{n}_{\infty} \big) \ge c P ( A^{n}_{\infty} ).
\end{equation*}
With this in mind, we condition that $( S^{1}[0, \tau^{1}_{n}], \gamma^{\infty} [0, \tau^{\infty}_{n} ]) = (\gamma^{1}, \gamma^{2} )$ so that $(\gamma^{1}, \gamma^{2} )$ satisfies $\overline{\textsf{Sep}}_{\infty} (n)$ and $A^{n}_{\infty}$. Namely $\gamma^{1}$ and $\gamma^{2}$ are well separated and they do not intersect. We write $w^{i}$ for the endpoint of $\gamma^{i}$. By the strong Markov property, under this conditioning, the law of $S^{1}$ after $\tau^{1}_{n}$ is same as the law of a simple random walk started at $w^{1}$. The domain Markov property (see Proposition \ref{DMP}) ensures that conditioned that $\gamma^{\infty} [0, \tau^{\infty}_{n} ] = \gamma^{2}$, the law of $\gamma^{\infty}$ after $\tau^{\infty}_{n}$ is given by the law of the loop-erasure of a random walk $X$ starting from $w^{2}$ conditioned that $X[1, \infty )$ do not hit $\gamma^{2}$. We write $\ell^{i}$ for the infinite half line starting from $w^{i}$ which is orthogonal to the $yz$-plane and lies in $B (n )^{c}$. We let $G^{i} := \{ x \in \mathbb{R}^{3} \ | \ \text{dist} (x, \ell^{i} ) \le \frac{n}{8} \}$ be a $\frac{n}{8}$-neighborhood of $\ell^{i}$. Then it is easy to see that 
\begin{equation}\label{hideosoi}
P^{w^{1}} \big( S^{1} [0, \tau^{1}_{4n} ] \subset G^{1} \big) \ge c,
\end{equation}
for some $c > 0$. We let $\eta := LE ( X [0, \infty ) )$ be the loop-erasure of the conditioned random walk $X$. We want to show that 
\begin{equation}\label{hidehamadanano}
P^{w^{2}} \big( \eta [0, \tau^{\eta}_{4n} ] \subset G^{2} \big) \ge c,
\end{equation}
where $\tau^{\eta}_{4n}$ stands for the first time that $\eta$ exits from $B ( 4n)$. We let $\tau^{X}_{m}$ be the first exit time for $X$ similarly. Suppose that $X [0, \tau^{X}_{8n}] \subset G^{2}$ and that $X[\tau^{X}_{8n}, \infty ) \cap B (4n) = \emptyset$. Then we have $\eta [0, \tau^{\eta}_{4n} ] \subset G^{2}$. Therefore the probability in \eqref{hidehamadanano} is bounded below by
\begin{equation}\label{naniyatteru}
P^{w^{2}} \Big( X [0, \tau^{X}_{8n}] \subset G^{2}, \ X[\tau^{X}_{8n}, \infty ) \cap B (4n) = \emptyset \Big).
\end{equation}
By definition of $X$, the probability of \eqref{naniyatteru} is bounded below by
\begin{equation}\label{huzakeruna}
\frac{P^{w^{2}} \Big( S^{2} [1, \infty) \cap \gamma^{2} = \emptyset, \ S^{2} [0, \tau^{2}_{8 n} ] \subset G^{2}, \ S^{2}[\tau^{2}_{8 n}, \infty) \cap B (4 n) = \emptyset \Big)}{P^{w^{2}} \big( S^{2} [1, \tau^{2}_{\frac{n}{16}}] \cap \gamma^{2} = \emptyset \big) }.
\end{equation}
Using Proposition \ref{Masson} along with Proposition 1.5.10 of \cite{Law b}, we see that the ratio of \eqref{huzakeruna} is bounded below by a constant $c > 0$. So we get \eqref{hidehamadanano}. The separation event ensures that $G^{i} \cap (\gamma^{3-i} \cup G^{3-i} ) = \emptyset$ for each $i =1, 2$. Therefore by attaching $S^{1} [0, \tau^{1}_{4n} ]$ (we assume $S^{1} (0) = w^{1}$ here) and $\eta [0, \tau^{\eta}_{4n} ]$ to $S^{1}[0, \tau^{1}_{n}]$ and $\gamma^{\infty} [0, \tau^{\infty}_{n} ]$ respectively as in \eqref{hideosoi} and \eqref{hidehamadanano}, we see that $Es^{\infty} (4n) \ge c Es^{\infty} (n)$. So we get $Es^{\infty} (n) \asymp Es^{\infty} (4n)$.

We will next show $Es^{\infty} (n) \asymp Es (4n)$. To achieve it, by Corollary 4.5 of \cite{Mas}, it suffices to prove that 
\begin{equation}\label{taimingu}
P \Big( S^{1} [1, \tau^{1}_{n} ] \cap \eta^{1}_{n} \big( LE (S^{2} [0, \tau^{2}_{4n}] ) \big) = \emptyset \Big) \asymp P (A^{4n} ).
\end{equation}
(Recall that $\eta^{1}$ was defined at the beginning of this subsection.) It is clear that the left hand side of \eqref{taimingu} is bounded below by the right hand side. To prove the other inequality, we use the separation lemma (see Theorem \ref{sep lem 3dim}) again. By Theorem \ref{sep lem 3dim}, conditioned on $S^{1} [1, \tau^{1}_{n} ] \cap \eta^{1}_{n} \big( LE (S^{2} [0, \tau^{2}_{4n}] ) \big) = \emptyset$, with positive conditional probability they are separated. Then we can attach paths to them as above, and conclude \eqref{taimingu}. We leave the details to the reader.
\end{proof}

Take two integers $m \le n$. We will next relates $Es (n)$ with the product of $Es (m )$ and $Es (m, n )$. Namely we will show that $Es (n ) \asymp Es (m ) Es (m, n)$ in Proposition \ref{up to const indep1} and Proposition \ref{up to const indep2}. We start by proving $Es (n)$ is bounded above by $C Es (m ) Es (m, n)$ for some $C  < \infty$.

\begin{prop}\label{up to const indep1}
Let $d=3$. There exists $C < \infty$ such that for all $m, n$ with $m \le n$,
\begin{equation}
Es (n) \le C Es (m) Es (m,n).
\end{equation}
\end{prop}

\begin{proof}
We set $\eta^{1} = \eta^{1}_{\frac{m}{4}} \big( LE ( S^{2} [0, \tau^{2}_{n} ] ) \big)$ for $LE ( S^{2} [0, \tau^{2}_{n} ] )$ up to its first exit time of $B (\frac{m}{4} )$. We also write $\eta^{2} = \eta^{2}_{m, n} \big( LE ( S^{2} [0, \tau^{2}_{n} ] ) \big)$ for $LE ( S^{2} [0, \tau^{2}_{n} ])$ from its last visit to $B ( m)$. The definition of $Es (n)$ gives that
\begin{equation}\label{nagaburo}
Es (n ) \le P \Big( S^{1} [1, \tau^{1}_{\frac{m}{4}} ] \cap \eta^{1} = \emptyset, \ S^{1} [\tau^{1}_{\frac{m}{4}}, \tau^{1}_{n} ] \cap \eta^{2} = \emptyset \Big).
\end{equation}
Using the strong Markov property first and then applying the discrete Harnack principle (see Theorem 1.7.6 of \cite{Law b}), we see that the probability in the right hand side of \eqref{nagaburo} is bounded above by
\begin{equation}\label{nagaburo-1}
 C E_{2} \Big\{ P_{1} \big( S^{1} [1, \tau^{1}_{\frac{m}{4}} ] \cap \eta^{1} = \emptyset \big)  P_{1} \big( S^{1} [1, \tau^{1}_{n} ] \cap \eta^{2} = \emptyset \big) \Big\}.
\end{equation}
(Recall that we write $P_{i}$ for the probability of $S^{i}$ and write $E_{i}$ for its expectation.) However, Proposition 4.6 of \cite{Mas} shows that $\eta^{1}$ and $\eta^{2}$ are independent ``up to constant" (see Proposition 4.6 of \cite{Mas} for it). From this, it follows that the quantity of \eqref{nagaburo-1} is bounded above by
\begin{equation}\label{nagaburo-2}
 C P \big( S^{1} [1, \tau^{1}_{\frac{m}{4}} ] \cap \eta^{1} = \emptyset \big)  P \big( S^{1} [1, \tau^{1}_{n} ] \cap \eta^{2} = \emptyset \big) = C P \big( S^{1} [1, \tau^{1}_{\frac{m}{4}} ] \cap \eta^{1} = \emptyset \big) Es (m, n).
\end{equation}
Corollary 4.5 of \cite{Mas} gives that the distribution of $\eta^{1}$ is same as the distribution of $\gamma^{\infty} [0, \tau^{\infty}_{\frac{m}{4}}]$ up to multiplicative constants (see Corollary 4.5 of \cite{Mas} for this). Therefore we have 
\begin{equation}\label{nagaburo-3}
P \big( S^{1} [1, \tau^{1}_{\frac{m}{4}} ] \cap \eta^{1} = \emptyset \big) \asymp P \big( S^{1} [1, \tau^{1}_{\frac{m}{4}} ] \cap \gamma^{\infty} [0, \tau^{\infty}_{\frac{m}{4}}] = \emptyset \big) = Es^{\infty} ( \frac{m}{4} ).
\end{equation}
Finally, Proposition \ref{up-to-const lew} shows that $Es^{\infty} ( \frac{m}{4} ) \asymp Es (m )$ and we finish the proof. 
\end{proof}

In order to prove that $Es (n) \ge c Es (m) Es (m,n)$, we need the next lemma. The next lemma estimates the conditional probability that a random walk lies in a given set conditioned that the random walk avoids some sets and that the endpoint of the random walk is equal to a given point. This lemma is an analog of Corollary 3.8 of \cite{Mas} in 3 dimensions. To state the lemma, we start by giving some definitions.

Take $\kappa \in (0, 1)$. We write $\pi (\kappa ) = \{ x = (x_{1}, x_{2}, x_{3} ) \in \mathbb{R}^{3} \ | \ x_{1} = \kappa \}$ for the plane which is orthogonal to the $x_{1}$-axis and has a distance $\kappa$ from the origin. We set $H ( \kappa ) = \{ |x| < 1 \} \cap \pi ( \kappa )$ for the intersection of $\pi ( \kappa )$ and the unit open ball. Using this set, we define a cone $O (\kappa )$ by $O (\kappa ) = \{ r x \ | \ r > 0, \ x \in H ( \kappa ) \}$.

\begin{lem}\label{koremohitsuyou}
Let $d=3$. Suppose that $\kappa \in (0, 1)$ and $0 < a < 1 < b < \infty$. There exists a constant $c = c (\kappa, a, b )$ which depends on constants $\kappa$, $a$ and $b$ such that the following holds. We consider a subset $W$ of a cone defined by 
\begin{equation*}
W = \{ x \in O ( \kappa ) \ | \ an \le |x| \le 4 b n \}.
\end{equation*}
We take two subsets $K_{1}$ and $K_{2}$ satisfying that $K_{1} \subset B (n )$ and $K_{2} \cap B (4 n) = \emptyset$. We set $K = K_{1} \cup K_{2} \cup B (4 b n )^{c}$. Then it follows that for all $z \in \partial B (n )$ and $y \in \partial_{i} B ( 4 n)$ with $z, y \in O ( \frac{\kappa + 1}{2} )$, 
\begin{equation}\label{kataitai}
P^{z} \big( S [0, \sigma_{y} ] \subset W \ | \ \sigma_{y} < \sigma_{K} \big) \ge c. 
\end{equation}

\end{lem} 

\begin{proof}
We consider a subset $W'$ defined by
\begin{equation*}
W' = \{ x \in O ( \frac{3 \kappa + 1}{4} ) \ | \ an \le |x| \le 2 n \}.
\end{equation*}
Using Proposition \ref{Masson}, we see that 
\begin{equation*}
P^{z} \big( S[0, \tau_{2n} ] \subset W' \ | \ \tau_{2n} < \sigma_{K} \big) \ge c,
\end{equation*}
for some $c > 0$. Therefore, in order to prove \eqref{kataitai}, it suffice to show that there exists a $c > 0$ such that for all $w \in \partial B ( 2n ) \cap W'$
\begin{equation}\label{amedarui}
P^{w} \big( S [0, \sigma_{y} ] \subset W \ | \ \sigma_{y} < \sigma_{K} \big) \ge c. 
\end{equation}
However, by Lemma 3.1 of \cite{Mas}, we see that the probability in the left hand side of \eqref{amedarui} is equal to 
\begin{equation}\label{kasanai}
\frac{ G \big( w, w, W \cap (\{ y \} \cup K )^{c} \big)  P^{y} \big( \sigma_{w} < \tau_{W} \wedge \sigma_{K} \wedge \sigma_{y} \big) }{G \big( w, w, (\{ y \} \cup K )^{c} \big) P^{y} \big( \sigma_{w} <  \sigma_{K} \wedge \sigma_{y} \big)  }.
\end{equation}
(Recall that $G (\cdot, \cdot, W )$ stands for Green's function in $W$ and that $\sigma_{y} = \inf \{ j \ge 1 \ | \ S (j) = y \}$.) Note that by the transience of $S$ in 3 dimensions both Green's functions in \eqref{kasanai} are constants. Therefore, we need to show that 
\begin{equation}\label{kasanai-1}
\frac{   P^{y} \big( \sigma_{w} < \tau_{W} \wedge \sigma_{K} \wedge \sigma_{y} \big) }{ P^{y} \big( \sigma_{w} <  \sigma_{K} \wedge \sigma_{y} \big)  } \ge c.
\end{equation}
Using Proposition \ref{Masson} along with Proposition 1.5.10 of \cite{Law b}, we get \eqref{kasanai-1} and finish the proof of the lemma.
\end{proof}

Now we are ready to show that $Es (n ) \ge c Es (m ) Es (m, n)$ in the next proposition. In the proof of the proposition, we will use the separation lemmas (Theorem \ref{sep lem 3dim} and Theorem \ref{rev sep}) along with Lemma \ref{koremohitsuyou}.

\begin{prop}\label{up to const indep2}
Let $d=3$. There exists $c > 0$ such that for all $m, n$ with $m \le n$,
\begin{equation}
Es (n) \ge c Es (m) Es (m,n).
\end{equation}
\end{prop}

\begin{proof}
Recall that for $\kappa \in (0, 1)$ we define a cone $O (\kappa )$ just before stating Lemma \ref{koremohitsuyou}. Let 
\begin{equation*}
W := O (\frac{3}{4}) \cap \big( B (\frac{5 m}{4} ) \setminus B (\frac{m}{5} ) \big)
\end{equation*}
be a subset of a cone $O (\frac{3}{4})$ restricted to $\big( B (\frac{5 m}{4} ) \setminus B (\frac{m}{5} ) \big)$. We write $W_{-} = \{ (-x_{1}, x_{2}, x_{3} ) \ | \ (x_{1}, x_{2}, x_{3} ) \in W \}$ for the reflection of $W$ with respect to the $x_{2} x_{3}$-plane. We also define a set $W'$ by 
\begin{equation*}
W' := O (\frac{1}{2}) \cap \big( B (\frac{5 m}{4} ) \setminus B (\frac{m}{5} ) \big).
\end{equation*}
Note that $W \subset W'$. We write $W'_{-}$ for the reflection of $W'$ with respect to the $x_{2} x_{3}$-plane similarly.

Throughout the proof of this proposition, we write $\lambda = S^{1} [0, \tau^{1}_{n} ]$ for the path of $S^{1}$ up to $\tau^{1}_{n}$ and $\gamma = LE (S^{2} [0, \tau^{2}_{n} ] )$ for the loop-erasure of $S^{2}$ up to its first exit time of $B (n)$. We next define several random times as follows. We write $u_{1} = \tau^{1}_{\frac{m}{4}}$ and let $u_{2} = \max \{ j \le \tau^{1}_{n} \ | \ S^{1} (j) \in B (m ) \}$ be its last visit to $B ( m)$ up to $ \tau^{1}_{n}$. Similarly we write $t_{1} = \tau^{\gamma}_{\frac{m}{4}}$ and let $t_{2} = \max \{ j \le \tau^{\gamma}_{n} \ | \ \gamma (j) \in B (m ) \}$ be its last visit to $B ( m)$ up to the first time that $\gamma$ exits from $B (n)$. 

Suppose that all of the following three events are fulfilled:
\begin{itemize}
\item[(a)] $\lambda [1, u_{1}] \cap \gamma [0, t_{1} ] = \emptyset$, $\lambda [0, u_{1}] \subset B ( \frac{m}{5} ) \cup W_{-}$ and $\gamma [0, t_{1} ] \subset B ( \frac{m}{5} ) \cup W$.

\item[(b)] $\lambda [u_{1}, u_{2} ] \subset W'_{-}$ and $\gamma [t_{1}, t_{2} ] \subset W'$.

\item[(c)] $\lambda (u_{2} ) \in W_{-}$, $\gamma (t_{2} ) \in W$ and $\Big( \lambda [u_{2}, \text{len} (\lambda ) ] \cup W'_{-} \Big) \cap \Big( \gamma [t_{2}, \text{len} (\gamma ) ] \cup W' \Big) = \emptyset$.
\end{itemize}
Then the definitions of $W, W', W_{-}$ and $W'_{-}$ ensure that $\lambda [1, \text{len} (\lambda ) ] \cap \gamma [0, \text{len} (\gamma ) ] = \emptyset$. Therefore we see that $Es (n)$ is bounded below by the probability that the events (a), (b) and (c) hold. So we need to estimate $P \Big( \text{(a)}, \ \text{(b)}, \ \text{(c)} \Big)$.

Conditioned that $\lambda [0, u_{1}]$ and $\gamma [0, t_{1} ]$ satisfy (a), and that $\lambda [u_{2}, \text{len} (\lambda ) ]$ and $\gamma [t_{2}, \text{len} (\gamma ) ]$ satisfy (c), then by using the strong Markov property and the domain Markov property (see Proposition \ref{DMP}) along with Lemma \ref{koremohitsuyou}, the conditional probability that the event (b) holds is bounded below by some constant $ c > 0$. Namely we have
\begin{equation*}
P \Big( \text{(b)} \ \big| \ \text{(a)}, \ \text{(c)} \Big) \ge c,
\end{equation*}
which gives that $Es (n ) \ge c P \Big(  \text{(a)}, \ \text{(c)} \Big)$. However Proposition 4.6 of \cite{Mas} ensures that the event (a) and (c) are independent up to multiplicative constants (see Proposition 4.6 of \cite{Mas} for this). Thus we see that $Es (n ) \ge c P \big(  \text{(a)}  \big) P \big(  \text{(c)} \big)$. So in order to prove this proposition, it suffices to show that 
\begin{align}
&P \big(  \text{(a)}  \big) \ge c Es (m) \label{pukosuke-sanpo} \\
&P \big(  \text{(c)}  \big) \ge c Es (m, n) \label{pukosuke-sanpo-1}.
\end{align}

We will first show \eqref{pukosuke-sanpo} using Theorem \ref{sep lem 3dim}.
Recall that Theorem \ref{sep lem 3dim} shows that conditioned on $\lambda [1, u_{1}] \cap \gamma [0, t_{1} ] = \emptyset$, then they are well-separated with positive conditional probability. Therefore Theorem \ref{sep lem 3dim} gives that $P \big(  \text{(a)}  \big)$ is bounded below by $c P \big( \lambda [1, u_{1}] \cap \gamma [0, t_{1} ] = \emptyset )$. But using Corollary 4.5 of \cite{Mas}, we can replace $\gamma [0, t_{1} ]$ by $\gamma^{\infty} [0, \tau^{\infty}_{\frac{m}{4}} ]$ (recall that $\gamma^{\infty}$ stands for the infinite LERW). So $P \big( \lambda [1, u_{1}] \cap \gamma [0, t_{1} ] = \emptyset )$ is bounded below by $c P \big( \lambda [1, u_{1}] \cap \gamma^{\infty} [0, \tau^{\infty}_{\frac{m}{4}}] = \emptyset \big)$ which is equal to $c Es^{\infty} (\frac{m}{4})$. Finally, it follows from Proposition \ref{up-to-const lew} that $Es^{\infty} (\frac{m}{4}) \asymp Es (m)$ which gives \eqref{pukosuke-sanpo}.

We will prove \eqref{pukosuke-sanpo-1} by using Theorem \ref{rev sep}. To achieve it, we first set $u_{3} = \max \{ j \le \text{len} (\lambda ) \ | \ \lambda (j) \in B ( 2 m ) \}$ and $t_{3} = \max \{ j \le \text{len} (\gamma ) \ | \ \gamma (j) \in B ( 2 m ) \}$ for the last time that $\lambda$ and $\gamma$ visit to $B (2m)$. We let $d := \text{dist} \big( \lambda (u_{3} ), \gamma [t_{3}, \text{len} (\gamma ) ] ) \wedge \text{dist} \big( \gamma (t_{3} ), \lambda [u_{3}, \text{len} (\lambda ) ] )$ be the distance between the endpoint and the other path. Then Theorem \ref{rev sep} gives that 
\begin{equation*}
P \Big( \lambda [u_{3}, \text{len} (\lambda ) ] \cap \gamma [t_{3}, \text{len} (\gamma ) ] = \emptyset, \ d \ge c m \Big) \ge c P \Big( \lambda [u_{3}, \text{len} (\lambda ) ] \cap \gamma [t_{3}, \text{len} (\gamma ) ] = \emptyset \Big),
\end{equation*}
for some $c > 0$. Conditioned that $\lambda [u_{3}, \text{len} (\lambda ) ] \cap \gamma [t_{3}, \text{len} (\gamma ) ] = \emptyset$ and $d \ge c m$, by using the strong Markov property and the domain Markov property (see Proposition \ref{DMP}) along with Proposition \ref{Masson}, we can find two subset $J_{1}$ and $J_{2}$ which satisfy the following conditions:
\begin{itemize}
\item $\lambda  ( u_{3} ) \in J_{1}$ and $\gamma  ( t_{3} ) \in J_{2}$.

\item $\big( J_{1} \cup W'_{-} \big) \cap \big( J_{2} \cup W' \big) = \emptyset$, $\big( J_{1} \cap B (\frac{5 m}{4} ) \big) \subset W_{-}$ and $\big( J_{2} \cap B (\frac{5 m}{4} ) \big) \subset W$.

\item $P \Big( \lambda [u_{2}, u_{3} ] \subset J_{1}, \ \gamma [ t_{2}, t_{3} ] \subset J_{2} \ \big| \ \lambda [u_{3}, \text{len} (\lambda ) ] \cap \gamma [t_{3}, \text{len} (\gamma ) ] = \emptyset, \ d \ge c m \Big) \ge c $.
\end{itemize}
Suppose that $\lambda [u_{3}, \text{len} (\lambda ) ] \cap \gamma [t_{3}, \text{len} (\gamma ) ] = \emptyset$, $d \ge c m$, and $\lambda [u_{2}, u_{3} ] \subset J_{1}, \ \gamma [ t_{2}, t_{3} ] \subset J_{2}$. Then we see that the event (c) holds. Therefore the probability of the event (c) is bounded below by
\begin{equation*}
P \Big( \lambda [u_{2}, u_{3} ] \subset J_{1}, \ \gamma [ t_{2}, t_{3} ] \subset J_{2}, \ \lambda [u_{3}, \text{len} (\lambda ) ] \cap \gamma [t_{3}, \text{len} (\gamma ) ] = \emptyset, \ d \ge c m \Big).
\end{equation*}
But we have already proved that the probability above is bounded below by
\begin{equation}\label{osanpo-i}
P \Big( \lambda [u_{3}, \text{len} (\lambda ) ] \cap \gamma [t_{3}, \text{len} (\gamma ) ] = \emptyset, \ d \ge c m \Big) \ge c P \Big( \lambda [u_{3}, \text{len} (\lambda ) ] \cap \gamma [t_{3}, \text{len} (\gamma ) ] = \emptyset \Big).
\end{equation}
The definition of the escape probability immediately gives that the right hand side of \eqref{osanpo-i} is bounded below by $Es (2m , n)$. Consequently we get $Es (n ) \ge c Es (m ) Es (2m, n)$ (we let $Es (k, l) = 1$ for $k > l$ here). But Proposition \ref{up-to-const lew} shows that $Es (m ) \asymp Es (2m )$ and thus we have $Es (n ) \ge c Es ( 2m ) Es (2m, n)$. Replacing $m$ by $\frac{m}{2}$, we finish the proof. 
\end{proof}

From Proposition \ref{up to const indep1} and Proposition \ref{up to const indep2}, we see that $Es (n) \asymp Es (m) Es (m, n)$. Recall that $Es (m, n) $ stands for the probability that a simple random walk up to its first exit of $B (n)$ and the loop-erasure of an independent simple random walk up to the first exit of $B (n)$ after last time that the loop-erasure visits to $B (m)$ do not intersect. We want to show that this probability is comparable to the probability that a simple random walk starting from $ (-m, 0, 0)$ up to its first exit of $B (n)$ and the loop-erasure of an independent simple random walk starting from $(m, 0, 0)$ up to the first time that the simple random walk exits from $B (n)$ do not intersect. Using the separation lemma and attaching paths as in the proof of Proposition \ref{up to const indep2}, we have the following proposition, which will be used in the next section.

\begin{prop}\label{step1}
Let $d=3$. We write $x^{n} = (2^{n}, 0, 0)$ for a pole of $B (2^{n} )$. Then there exists $c>0$ such that for each $k, m, n$, 
\begin{align}\label{jikotta}
&c P^{-x^{k}, x^{k}} \Big( S^{1} [0, \tau^{1}_{2^{k + n}}] \cap LE ( S^{2} [0, \tau^{2}_{2^{k + n}}] ) = \emptyset \Big) \notag \\
&  \ \ \ \ \times P^{-x^{k+n}, x^{k+n}} \Big( S^{1} [0, \tau^{1}_{2^{k + n + m}}] \cap LE ( S^{2} [0, \tau^{2}_{2^{k + n + m}}] ) = \emptyset \Big) \notag \\
&\le P^{-x^{k}, x^{k}} \Big( S^{1} [0, \tau^{1}_{2^{k + n + m}}] \cap LE ( S^{2} [0, \tau^{2}_{2^{k + n + m}}] ) = \emptyset \Big) \notag \\
&\le \frac{1}{c} P^{-x^{k}, x^{k}} \Big( S^{1} [0, \tau^{1}_{2^{k + n}}] \cap LE ( S^{2} [0, \tau^{2}_{2^{k + n}}] ) = \emptyset \Big) \notag \\
&  \ \ \ \ \times P^{-x^{k+n}, x^{k+n}} \Big( S^{1} [0, \tau^{1}_{2^{k + n + m}}] \cap LE ( S^{2} [0, \tau^{2}_{2^{k + n + m}}] ) = \emptyset \Big).
\end{align}
\end{prop}

\begin{proof}
Recall that $\eta^{2}_{\cdot , \cdot} ( \cdot )$ was defined in the beginning of this subsection. We let 
\begin{equation*}
Es (2^{k}, 2^{k+n}, 2^{k+n+m}) = P^{-x^{k}, x^{k}}  \Big( S^{1} [0, \tau^{1}_{2^{k + n + m}}] \cap \eta^{2}_{2^{k+n}, 2^{k+n+m}} \big( LE ( S^{2} [0, \tau^{2}_{2^{k + n + m}}] ) \big) = \emptyset \Big)
\end{equation*}
be the probability that $S^{1} [0, \tau^{1}_{2^{k + n + m}}]$ and $LE ( S^{2} [0, \tau^{2}_{2^{k + n + m}}] )$ after its last visit to $B (2^{k + n} )$ do not intersect. Similar arguments as in the proof of Proposition \ref{up to const indep1} and Proposition \ref{up to const indep2} allow to show that 
\begin{align}\label{okaeri}
&P^{-x^{k}, x^{k}} \Big( S^{1} [0, \tau^{1}_{2^{k + n + m}}] \cap LE ( S^{2} [0, \tau^{2}_{2^{k + n + m}}] ) = \emptyset \Big) \notag \\
&\asymp P^{-x^{k}, x^{k}} \Big( S^{1} [0, \tau^{1}_{2^{k + n}}] \cap LE ( S^{2} [0, \tau^{2}_{2^{k + n}}] ) = \emptyset \Big) Es (2^{k}, 2^{k+n}, 2^{k+n+m}).
\end{align}
Therefore in order to prove \eqref{jikotta}, it suffices to show that
\begin{equation}\label{omukae}
Es (2^{k}, 2^{k+n}, 2^{k+n+m}) \asymp P^{-x^{k+n}, x^{k+n}} \Big( S^{1} [0, \tau^{1}_{2^{k + n + m}}] \cap LE ( S^{2} [0, \tau^{2}_{2^{k + n + m}}] ) = \emptyset \Big).
\end{equation}

To show \eqref{omukae}, we will introduce two cubes as follows. Let $y = (y_{1}, 0, 0) := \frac{x^{k} + x^{k +n}}{2}$ be the middle point of $x^{k}$ and $x^{k+n}$. We write $W^{1} = \{ z = (z_{1}, z_{2}, z_{3} ) \in \mathbb{Z}^{3} \ \big| \ |z_{1} -y_{1} | \le 2^{k + n+ 2}, \ |z_{2} | \le 2^{k + n+ 2}, \ |z_{3}| \le 2^{k + n+ 2} \}$ for the cube of side length $2^{k + n+ 3}$ centered at $y$. We also write $W^{2}$ for the cube of side length of $2^{k + n+ m+ 3}$ centered at $y$. These definitions of $W^{1}$ and $W^{2}$ ensure that $B ( 2^{k+n} ) \subset W^{1} \subset B (2^{k+ n + 4} )$ and $B ( 2^{k+n + m} ) \subset W^{2} \subset B ( 2^{k+ n + m+ 4} )$. For a path $\lambda$, we define $\eta^{2}_{W^{1}, W^{2}} (\lambda)$ as follows. We let $s_{1} := \inf \{ j \ | \ \lambda (j) \notin W^{2} \}$ be the first time that $\lambda$ exits from $W^{2}$. We set $s_{2} := \sup \{ j \le s_{1} \ | \ \lambda (j) \in W^{1} \}$ for the last visit to $W^{1}$ up to time $s_{1}$. Then we write $\eta^{2}_{W^{1}, W^{2}} (\lambda) = \lambda [s_{2}, s_{1} ]$. We let $\tau^{l}_{W^{i}} = \inf \{ j \ | \ S^{l} (j) \notin W^{i} \}$ be the first time that $S^{l}$ exits from $W^{i}$ for each $i =1, 2$ and $l=1, 2$. Then similar arguments as in the proof of Proposition \ref{up to const indep1} and Proposition \ref{up to const indep2} give that
\begin{align*}
&P^{-x^{k}, x^{k}} \Big( S^{1} [0, \tau^{1}_{W^{2}} ] \cap LE \big( S^{2} [0, \tau^{2}_{W^{2}} ] \big) = \emptyset \Big) \\
&\asymp P^{-x^{k}, x^{k}} \Big( S^{1} [0, \tau^{1}_{W^{1}} ] \cap LE \big( S^{2} [0, \tau^{2}_{W^{1}} ] \big) = \emptyset \Big) \\
&\times P^{-x^{k}, x^{k}} \Big( S^{1} [0, \tau^{1}_{W^{2}} ] \cap \eta^{2}_{W^{1}, W^{2}} \big( LE \big( S^{2} [0, \tau^{2}_{W^{2}} ] \big) \big) = \emptyset \Big).
\end{align*}
This comparison is an analog of \eqref{okaeri} for cubes. 

We will next see that the non-intersecting probability up to the first exit of the cube is comparable to the non-intersecting probability up to the first exit of the ball. Namely Proposition \ref{up-to-const lew} gives that $P^{-x^{k}, x^{k}} \Big( S^{1} [0, \tau^{1}_{W^{2}} ] \cap LE \big( S^{2} [0, \tau^{2}_{W^{2}} ] \big) = \emptyset \Big)$ is comparable to $P^{-x^{k}, x^{k}} \Big( S^{1} [0, \tau^{1}_{2^{k + n + m}}] \cap LE ( S^{2} [0, \tau^{2}_{2^{k + n + m}}] ) = \emptyset \Big)$. Similarly we see that $P^{-x^{k}, x^{k}} \Big( S^{1} [0, \tau^{1}_{W^{1}} ] \cap LE \big( S^{2} [0, \tau^{2}_{W^{1}} ] \big) = \emptyset \Big)$ is comparable to 
\begin{equation*}
P^{-x^{k}, x^{k}} \Big( S^{1} [0, \tau^{1}_{2^{k + n}}] \cap LE ( S^{2} [0, \tau^{2}_{2^{k + n}}] ) = \emptyset \Big). 
\end{equation*}
Therefore it follows that $Es (2^{k}, 2^{k+n}, 2^{k+n+m})$ is comparable to 
\begin{equation}\label{obaa}
P^{-x^{k}, x^{k}} \Big( S^{1} [0, \tau^{1}_{W^{2}} ] \cap \eta^{2}_{W^{1}, W^{2}} \big( LE \big( S^{2} [0, \tau^{2}_{W^{2}} ] \big) \big) = \emptyset \Big).
\end{equation}
Namely we may consider the non-intersecting probability of $S^{1}$ up to $\tau^{1}_{W^{2}}$ and the loop-erasure of $S^{2}$ up to $\tau^{2}_{W^{2}}$ after its last visit to $W^{1}$ instead of $Es (2^{k}, 2^{k+n}, 2^{k+n+m})$. Since $P^{-x^{k+n}, x^{k+n}} \Big( S^{1} [0, \tau^{1}_{W^{1}} ] \cap LE \big( S^{2} [0, \tau^{2}_{W^{1}} ] \big) = \emptyset \Big)$ is bounded below by a constant, similar arguments as above show that 

$P^{-x^{k+n}, x^{k+n}} \Big( S^{1} [0, \tau^{1}_{2^{k + n + m}}] \cap LE ( S^{2} [0, \tau^{2}_{2^{k + n + m}}] ) = \emptyset \Big)$ is comparable to 
\begin{equation}\label{obaa-1}
P^{-x^{k+n}, x^{k+n}} \Big( S^{1} [0, \tau^{1}_{W^{2}} ] \cap \eta^{2}_{W^{1}, W^{2}} \big( LE \big( S^{2} [0, \tau^{2}_{W^{2}} ] \big) \big) = \emptyset \Big).
\end{equation}

Consequently in order to finish the proof, it suffices to show that the probability of \eqref{obaa} is comparable to the probability of \eqref{obaa-1}. To achieve it, we will use a simple reflection argument as follows. We start by replacing a staring point of $S^{1}$ by $y$. Since $\eta^{2}_{W^{1}, W^{2}} \big( LE \big( S^{2} [0, \tau^{2}_{W^{2}} ] \big)$ lies in $(W^{1})^{c}$, the discrete Harnack principle (see Theorem 1.7.6 of \cite{Law b}) shows that the probability of \eqref{obaa} is comparable to
\begin{equation}\label{obaa-2}
P^{y, x^{k}} \Big( S^{1} [0, \tau^{1}_{W^{2}} ] \cap \eta^{2}_{W^{1}, W^{2}} \big( LE \big( S^{2} [0, \tau^{2}_{W^{2}} ] \big) \big) = \emptyset \Big).
\end{equation}
Similarly we see that the probability of \eqref{obaa-1} is comparable to 
\begin{equation}\label{obaa-3}
P^{y, x^{k+n}} \Big( S^{1} [0, \tau^{1}_{W^{2}} ] \cap \eta^{2}_{W^{1}, W^{2}} \big( LE \big( S^{2} [0, \tau^{2}_{W^{2}} ] \big) \big) = \emptyset \Big).
\end{equation}
However, by using a reflection of paths with respect to the plane $\{ (y_{1}, z_{2}, z_{3} )  \ | \ (z_{2}, z_{3} ) \in \mathbb{Z}^{2} \}$ (recall that $y= (y_{1}, 0, 0)$ is the middle point of $x^{k}$ and $x^{k+n}$), we can conclude that the probability of \eqref{obaa-2} is equal to the probability of \eqref{obaa-3}, and thus we finish the proof.
\end{proof}

\section{Rate of convergence for $Es (n)$}
As we discussed at the beginning of the previous section, we want to establish exponential tail bounds on the length of LERW in 3 dimensions. In order to achieve it, we need to give bounds on the higher moments of the length of LERW in terms of the escape probability. To derive such moment estimates, it turns out that we need to show the existence of some exponent $\alpha > 0$ such that 
\begin{equation}\label{rate-of-conve-esp}
\lim_{n \to \infty} \frac{ \log Es (n)}{\log n} = - \alpha,
\end{equation}
for $d=3$. We point out that $\frac{1}{3} \le \alpha < 1$ if $\alpha$ exists (see \cite{Law5}). In this section we will prove that the limit in \eqref{rate-of-conve-esp} exists.

Before going to the proof, we will explain the strategy of it here. To show the existence of $\alpha$, we first consider 
\begin{equation}\label{ososugiru}
a_{m, n} = P^{-x^{n}, x^{n}} \Big( S^{1} [0, \tau^{1}_{2^{m + n}}] \cap LE ( S^{2} [0, \tau^{2}_{2^{m + n}}] ) = \emptyset \Big)
\end{equation}
the probability that $S^{1}$ starting from $-x^{n}$ up to $\tau^{1}_{2^{m + n}}$ and the loop-erasure of $S^{2}$ starting from $x^{n}$ up to $\tau^{2}_{2^{m + n}}$ do not intersect, which was considered in Proposition \ref{step1} (recall that $x^{n} = (2^{n}, 0, 0)$ is a pole of $B (2^{n})$). Fixing $m$, we are interested in the existence of limit of $a_{m, n}$ as $n \to \infty$. Suppose that $\frac{S^{1} [0, \tau^{1}_{2^{m + n}}]}{2^{n}}$ stands for the rescaled simple random walk obtained by multiplying $S^{1} [0, \tau^{1}_{2^{m + n}}]$ by $2^{-n}$. We also consider the rescaled LERW $\frac{LE ( S^{2} [0, \tau^{2}_{2^{m + n}}] )}{2^{n}}$ similarly. Note that the rescaled simple random walk and LERW start from $(-1, 0, 0)$ and $(1, 0, 0)$ respectively and run until the first exit from $B (2^{m})$. We also mention that the probability $a_{m, n}$ is equal to the probability that these rescaled walks do not intersect. Fixing $m$ and letting $n \to \infty$, we know that the rescaled simple random walk converges to a Brownian motion and that the limit of the rescaled LERW exists (see \cite{Koz}). Therefore it is natural to expect that the limit of $a_{m, n}$ as $n \to \infty$ exists for each fixed $m$. Unfortunately the existence of the limit of $a_{m, n}$ is not an immediate consequence of the existence of the scaling limit of LERW shown in \cite{Koz}. For the scaling limit of LERW in 3 dimensions, the topology considered in \cite{Koz} is the topology of the space of compact subsets with the Hausdorff metric, which is somewhat weak for our purpose. In addition to that, little is known about the scaling limit when $d=3$. Therefore in oder to prove the existence of the limit of $a_{m, n}$ as $n \to \infty$, we need more work and it will be done in the next subsection.

Once we have showed that the limit of $a_{m, n}$ exists as $n \to \infty$ for each $m$, we may write $a_{m} = \lim_{n \to \infty} a_{m, n}$ for its limit. Then Proposition \ref{step1} immediately gives that $a_{k + m} \asymp a_{k} a_{m}$. A standard subadditive argument shows that there exists $\alpha > 0$ such that $a_{m} \asymp 2^{- \alpha m}$. Using this, we will prove that $Es (n ) \approx n^{-\alpha}$ in Theorem \ref{main result-kaetta}.

 \subsection{$\lim_{ n \to \infty} a_{m, n}$ exists}
Recall that $a_{m , n}$ was defined as in \eqref{ososugiru}. the goal of this subsection is to prove that the limit of $a_{m , n}$ exists as $n \to \infty$ for each fixed $m$ in the next proposition. In the proof we will compare $a_{m , n}$ with the probability that a Wiener sausage and the loop-erasure of an independent simple random walk do not intersect. Some results derived in \cite{Koz} will be used to compare them.

 \begin{prop}\label{main lem}
 Let $d=3$. Recall that $a_{m , n}$ was defined as in \eqref{ososugiru}. Then for each $m$, the limit 
\begin{equation}
\lim_{n \to \infty} a_{m,n} =: a_m
\end{equation}
exists.
\end{prop}

\begin{proof}
For a path $\lambda$, $t \ge 0$ and $L>0 $, we write 
\begin{equation*}
\lambda_{L} [0, t] = \{ \lambda (s) + B (L) : 0 \le s \le t \}.
\end{equation*}
for a ``sausage" of $\lambda$, which is a set of points $x$ with $\text{dist} (x, \lambda) \le L$. 

Fix $m \in \mathbb{N}$. We start by showing that for each $0 < \epsilon < 1$, there exists $\delta > 0$ such that 
\begin{equation}\label{fat-11}
\big| a_{m, n}  - P^{-x^n , x^n } \Big( S^{1}_{2^{(1- \epsilon ) n}} [0 , \tau^{1}_{2^{n+m}}] \cap LE ( S^{2} [0, \tau^{2}_{2^{n+m}}] )  = \emptyset \Big) \big| \le c 2^{-\delta n},
\end{equation}
for large $n$. Namely we want to replace $S^{1}$ by its $2^{(1- \epsilon ) n}$-neighborhood in the non-intersecting event.

To show \eqref{fat-11}, note that the difference in the left side of \eqref{fat-11} is equal to the probability that $S^{1}$ and the loop-erasure of $S^{2}$ do not intersect while the sausage of $S^{1}$ and the loop-erasure of $S^{2}$ intersect. Therefore we have
\begin{align*}
&\big| a_{m, n}  - P^{-x^n , x^n } \Big( S^{1}_{2^{(1- \epsilon ) n}} [0 , \tau^{1}_{2^{n+m}}] \cap LE ( S^{2} [0, \tau^{2}_{2^{n+m}}] )  = \emptyset \Big) \big|  \\
&= P^{-x^n , x_n } \Big( S^{1} [0, \tau^{1}_{2^{m + n}}] \cap LE ( S^{2} [0, \tau^{2}_{2^{m + n}}] ) = \emptyset, \ S^{1}_{2^{(1- \epsilon ) n}} [0 , \tau^{1}_{2^{n+m}}] \cap LE ( S^{2} [0, \tau^{2}_{2^{n+m}}] )  \neq \emptyset \Big).
\end{align*}
With this in mind, we define two events $F_{1}$ and $F_{2}$ by
\begin{align*}
&F_{1} = \{ S^{1} [0, \tau^{1}_{2^{m + n}}] \cap LE ( S^{2} [0, \tau^{2}_{2^{m + n}}] ) = \emptyset \} \\
&F_{2} = \{ S^{1}_{2^{(1- \epsilon ) n}} [0 , \tau^{1}_{2^{n+m}}] \cap LE ( S^{2} [0, \tau^{2}_{2^{n+m}}] )  \neq \emptyset \}.
\end{align*}

We write $\gamma = LE ( S^{2} [0, \tau^{2}_{2^{m + n}}] )$ for the loop-erasure of $S^{2} [0, \tau^{2}_{2^{m + n}}]$. Suppose that $F_{1}$ and $F_{2}$ occur. Then $S^{1}$ has a point within a distance $2^{(1- \epsilon ) n}$ of $\gamma$ while $S^{1}$ and $\gamma$ do not intersect. Thus we see that the following event $F_{3}$ occurs:
\begin{equation*}
F_{3} = \{ \exists w \in B (2^{n+m})  \text{ s.t. } \text{dist} ( w,  \gamma  ) \le 2^{(1- \epsilon ) n}, \  \sigma^{1}_{w} \le \tau^{1}_{2^{n+m}}, \ S^{1}[0 , \tau^{1}_{2^{n+m}}] \cap \gamma = \emptyset \},
\end{equation*}
where we recall that $\sigma^{1}_{w} = \inf \{ j \ge 1 \ | \ S^{1} (j) = w \}$ stands for the first time that $S^{1}$ hits $w$. 

We first want to show that with high probability, $w$ is not very close to $\partial B (2^{n+m})$. To show this, let $y := S^{1} \big( \tau^{1}_{2^{n+m} - 2^{(1- \frac{\epsilon}{2})n } } \big)$. We define a sequence of stopping times $t_{k}$ ($k=0, 1, \cdots , \frac{\epsilon n}{4}$) by
\begin{equation*}
t_{k} = \inf \{ j \ge \tau^{1}_{2^{n+m} - 2^{(1- \frac{\epsilon}{2})n } } \ | \ |S^{1}(j) - y| \ge 2^{(1- \frac{\epsilon}{2})n + k} \}.
\end{equation*}
It is easy to check that for each $k$, the probability that $S^{1} [t_{k}, t_{k+1} ]$ hits the boundary of $B (2^{n + m} )$ is bounded below by some universal constant $c > 0$. Therefore iterating this along with the strong Markov property, we see that there exists $\delta > 0$ such that 
\begin{equation*}
P^{-x^{n}} \Big( S^{1} [ \tau^{1}_{2^{n+m} - 2^{(1- \frac{\epsilon}{2})n } } , \tau^{1}_{2^{n+m}} ] \subset B (y, 2^{(1- \frac{\epsilon}{4} )n} ) \Big) \ge 1- 2^{-\delta \epsilon n }.
\end{equation*}
Now suppose that $S^{1} [ \tau^{1}_{2^{n+m} - 2^{(1- \frac{\epsilon}{2})n } } , \tau^{1}_{2^{n+m}} ]$ lies in $B (y, 2^{(1- \frac{\epsilon}{4} )n} )$ and that $F_{3}$ holds with $w$ lying in $B ( 2^{n+m} - 2^{(1- \frac{\epsilon}{2})n } )^{c}$. This implies that $S^{1}$ hits $w$ after $\tau^{1}_{2^{n+m} - 2^{(1- \frac{\epsilon}{2})n } }$ and that the distance between $w$ and $y$ is bounded above by $2^{(1- \frac{\epsilon}{4} )n}$. Since $w$ is a point within a distance $2^{(1- \epsilon ) n}$ of $\gamma$, we see that the distance between $y$ and $\gamma$ is bounded above by $2^{(1- \frac{\epsilon}{8} )n}$. However, using Proposition 1.5.10 of \cite{Law b}, we have
\begin{equation*}
P^{-x^{n}, x^{n}} \Big( S^{2} [0, \tau^{2}_{2^{n + m} } ] \cap B (y, 2^{(1- \frac{\epsilon}{8} )n} ) \neq \emptyset \Big) \le c \frac{ 2^{(1- \epsilon/8)n} }{2^{n} } = c 2^{ - \frac{ \epsilon n}{8}}.
\end{equation*}
Consequently if we write
\begin{equation*}
F_{4} = \{ \exists w \in B ( 2^{n+m} - 2^{(1- \frac{\epsilon}{2})n } )^{c} \text{ s.t. } \text{dist} ( w,  \gamma  ) \le 2^{(1- \epsilon ) n}, \ \sigma^{1}_{w} \le \tau^{1}_{2^{n+m}},  \ S^{1}[0 , \tau^{1}_{2^{n+m}}] \cap \gamma = \emptyset \}
\end{equation*}
for the event that $F_{3}$ occurs with $w$ which is close to the boundary of $B (2^{n+m})$, then we see that $P^{-x^{n}, x^{n}}  ( F_{4} ) \le 2^{-\delta \epsilon n } $. 

Therefore in order to prove \eqref{fat-11}, it suffices to show that the probability of the following event $F_{5}$ is bounded above by $c 2^{-\delta n}$:
\begin{equation*}
F_{5} = \{ \exists w \in B ( 2^{n+m} - 2^{(1- \frac{\epsilon}{2})n } )  \text{ s.t. } \text{dist} ( w,  \gamma  ) \le 2^{(1- \epsilon ) n}, \ \sigma^{1}_{w} \le \tau^{1}_{2^{n+m}},  \ S^{1}[0 , \tau^{1}_{2^{n+m}}] \cap \gamma = \emptyset \},
\end{equation*}
the event that $F_{3}$ occurs with $w$ which is not close to the boundary of $B (2^{n+m})$. To estimate the probability of $F_{5}$, we will use Lemma 4.8 in \cite{Koz} as follows. Suppose that $F_{5}$ occurs. Given $\gamma$, we may define the stopping $u$ by $u = \inf \{ j \ | \ \text{dist} \big( S^{1} (j), \gamma \cap B ( 2^{n+m} - 2^{(1- \frac{\epsilon}{2})n -1 } ) \big) \le 2^{(1- \epsilon ) n} \}$. Namely $u$ is the first time that $S^{1}$ hits the $2^{(1- \epsilon ) n}$-neighborhood of $\gamma$ restricted in $B ( 2^{n+m} - 2^{(1- \frac{\epsilon}{2})n -1 } )$. Then $u < \tau^{1}_{2^{n+m}}$ and $S^{1} [u, \tau^{1}_{2^{n+m}}] \cap \gamma = \emptyset$ on $F_{5}$. Thus we are interested in the probability 
\begin{equation}\label{majide}
P^{-x^{n}}_{1} \Big( u < \tau^{1}_{2^{n+m}}, \ S^{1} [u, \tau^{1}_{2^{n+m}}] \cap \gamma = \emptyset \Big),
\end{equation}
which is a function of $\gamma$. Lemma 4.8 in \cite{Koz} along with the strong Markov property gives that there exists $\delta > 0$ such that with very high probability of $\gamma$, the probability in \eqref{majide} is bounded above by $2^{- \delta n}$, i.e., 
\begin{equation*}
P^{x^{n}}_{2} \Big(   P^{-x^{n}}_{1} \Big( u < \tau^{1}_{2^{n +m}}, \ S^{1} [u, \tau^{1}_{2^{n+m}}] \cap \gamma = \emptyset \Big) > 2^{- \delta n} \Big) \le  c 2^{-10 n}.
\end{equation*}
By using this, we see that 
\begin{equation*}
P^{-x^{n}, x^{n}} (F_{5} ) \le P^{-x^{n}, x^{n}} \Big( u < \tau^{1}_{2^{n+m}}, \ S^{1} [u, \tau^{1}_{2^{n+m}}] \cap \gamma = \emptyset \Big) \le 2^{- \delta n},
\end{equation*}
which gives \eqref{fat-11}.

Next we want to show that $| a_{m, n+1} - a_{m, n} |$ is small. To achieve this, we consider a Wiener sausage as follows. Let $W = (W(t) )_{t \ge 0}$ be a Brownian motion in $\mathbb{R}^{3}$ started at $-x^{n}$, which is independent of $S^{2}$. We write $\tau^{W}_{R} = \inf  \{ t \ge 0 \ | \ W (t) \notin B (R) \}$ for the first time that $W$ exits from $B (R)$. We let  
\begin{equation*}
b_{m, n} = P^{-x^{n} , x^{n} } \Big( \gamma \cap W_{ 2^{\frac{2n}{3}} } [0, \tau^{W}_{2^{n+m}}]= \emptyset \Big)
\end{equation*}
be the probability that $\gamma$ and a $2^{\frac{2n}{3}}$-neighborhood of $W [0, \tau^{W}_{2^{n+m}}]$ do not intersect. We want to compare $a_{m, n}$ with $b_{m, n}$. To so it, we consider the following coupling of $S^{1}$ and $W$. Lemma 3.2 of \cite{L} shows that we can couple $W$ and $S^{1}$ in the same probability space such that the following holds: $W (0) = S^{1} (0) = -x^{n}$ and there exists $ \delta > 0$ such that 
\begin{equation*}
P^{-x^{n} } \Big( \max_{0 \le t \le \tau^{W}_{2^{n+m+ 1}}} | W(t) - S^{1} (3t) | \ge 2^{\frac{2n}{3} -1} \Big) \le c e^{- 2^{\delta n} }.
\end{equation*}
Namely, we can couple $S^{1}$ and $W$ up to the first exit of $B (2^{n+m+ 1})$ so that $S^{1}$ lies in a $2^{\frac{2n}{3} -1}$-neighborhood of $W$ with high probability. From now on we will consider $S^{1}$ and $W$ assuming that they are coupled as above. We set $F_{6} = \{ \max_{0 \le t \le \tau^{W}_{2^{n+m+ 1}}} | W(t) - S^{1} (3t) | \le 2^{\frac{2n}{3} -1} \}$ for the event that $S^{1}$ stays close to $W$. Suppose that $F_{6}$ holds. Then it is easy to see that $S^{1}_{ 2^{\frac{2n}{3} -2 } } [0, \tau^{1}_{2^{n + m}} ]$ is contained in a Wiener sausage $W_{2^{\frac{2n}{3} }} [0, \tau^{W}_{2^{n+m}}]$. Similarly we see that $W_{2^{\frac{2n}{3} }} [0, \tau^{W}_{2^{n+m}}]$ is contained in $S^{1}_{ 2^{\frac{2n}{3} +2 } } [0, \tau^{1}_{2^{n + m}} ]$. Therefore we have 
\begin{equation*}
\big| b_{m, n} - P^{-x^{n} , x^{n} } \Big( F_{6}, \ \gamma \cap W_{ 2^{\frac{2n}{3}} } [0, \tau^{W}_{2^{n+m}}]= \emptyset \Big) \big| \le c e^{- 2^{\delta n} },
\end{equation*}
and 
\begin{align}\label{sko}
& P^{-x^{n} , x^{n} } \Big(  \gamma \cap S^{1}_{ 2^{\frac{2n}{3} +2 } } [0, \tau^{1}_{2^{n + m}} ] = \emptyset \Big) - ce^{- 2^{\delta n} } \le P^{-x^{n} , x^{n} } \Big( F_{6}, \ \gamma \cap W_{ 2^{\frac{2n}{3}} } [0, \tau^{W}_{2^{n+m}}]= \emptyset \Big) \notag  \\
&\le P^{-x^{n} , x^{n} } \Big(  \gamma \cap S^{1}_{ 2^{\frac{2n}{3} -2 } } [0, \tau^{1}_{2^{n + m}} ] = \emptyset \Big).
\end{align}
Using \eqref{fat-11} with $\epsilon = \frac{1}{3}$, we can conclude that 
\begin{equation}\label{bmsw}
| b_{m,n} - a_{m, n} | \le c 2^{- \delta n}.
\end{equation}

Next we want to compare $b_{m, n}$ with $a_{m, n+1}$ by using a result derived in \cite{Koz}. To achieve it, we define the event $F_{7}$ by
\begin{equation*}
F_{7} = \{ W[0 , \tau^{W}_{2^{n+m}}] \cap B ( x^{n} , 2^{ n - \sqrt{n} } ) = \emptyset \}.
\end{equation*}
Theorem 3.17 of \cite{PP} shows that $P^{- x^{n}} (F_{7}^{c} ) \le c 2^{- \sqrt{n} }$. So we may assume that $W$ does not hit $B ( x^{n} , 2^{ n - \sqrt{n} } )$ with probability at least $1-  c 2^{- \sqrt{n} }$. We recall that for $r > 0$ and a set $D \subset \mathbb{R}^{3}$ we write $rD = \{ rz : z \in D \}$. In order to show that the difference between $b_{m, n}$ and $a_{m, n+1}$ is small, we will use Theorem 5 of \cite{Koz} as follows. Conditioned on the Brownian motion $W[0 , \tau^{W}_{2^{n+m}}]$ which satisfies $F_{7}$, we are interested in the probability $P^{x^{n}} \Big( \gamma \cap W_{ 2^{\frac{2n}{3}} } [0, \tau^{W}_{2^{n+m}}]= \emptyset \Big)$ (this probability is a function of $W[0 , \tau^{W}_{2^{n+m}}]$). Theorem 5 of \cite{Koz} shows that there exist  universal (deterministic) constants $\epsilon > 0$ and $c < \infty$ such that if $W[0 , \tau^{W}_{2^{n+m}}]$ satisfies $F_{7}$ then
\begin{align}\label{care}
&P^{x^{n}} \Big( \gamma \cap W_{ 2^{\frac{2n}{3}} } [0, \tau^{W}_{2^{n+m}}]= \emptyset \Big) =P^{x^{n+1} } \Big( LE (R [0, \tau^{R}_{2^{n + m +1}} ]) \cap 2 \big( W_{ 2^{\frac{2n}{3}} } [0, \tau^{W}_{2^{n+m}}] \big) = \emptyset \Big) \notag \\
&\ge P^{x^{n+1} } \Big( LE ( S^{2} [0, \tau^{2}_{2^{n + m +1}} ])  \cap W'_{ 2^{ (1- \epsilon ) n} } = \emptyset \Big) - c 2^{-\epsilon n},
\end{align}
where $R = (R (j))_{j \ge 0}$ stands for a simple random walk on $2 \mathbb{Z}^{3}$ started at $x^{n+1}$ and $\tau^{R}_{l}$ stands for the first exit of $B (l )$. We also mention that $W' = 2 W[0 , \tau^{W}_{2^{n+m}}] $ stands for an enlargement of the Brownian motion and that we write $W'_{ 2^{ (1- \epsilon ) n} }$ for a $2^{ (1- \epsilon ) n}$-neighborhood of $W'$. Namely, Theorem 5 of \cite{Koz} gives a comparison between the probability that the loop-erasure of a simple random walk on $2 \mathbb{Z}^{3}$ does not hit a set $A$ and the probability that 
the loop-erasure of a simple random walk on $ \mathbb{Z}^{3}$ does not hit $A$. It was shown there that the difference between these probabilities is small for a suitable set $A$, and thus we require $W[0 , \tau^{W}_{2^{n+m}}]$ to satisfy $F_{7}$. Using \eqref{care}, we see that 
\begin{align*}
&b_{m, n} \ge E^{-x^{n} } \Big\{ P^{x^{n}}_{2} \Big( \gamma \cap W_{ 2^{\frac{2n}{3}} } [0, \tau^{W}_{2^{n+m}}]= \emptyset \Big) \ ; \ F_{7} \Big\} \\
&\ge E^{-x^{n} } \Big\{ P^{x^{n+1} } \Big( LE ( S^{2} [0, \tau^{2}_{2^{n + m +1}} ])  \cap W'_{ 2^{ (1- \epsilon ) n} } = \emptyset \Big) - c 2^{-\epsilon n} \ ; \ F_{7} \Big\} \\
&\ge E^{-x^{n} } \Big\{ P^{x^{n+1} } \Big( LE ( S^{2} [0, \tau^{2}_{2^{n + m +1}} ])  \cap W'_{ 2^{ (1- \epsilon ) n} } = \emptyset \Big)  \Big\} -  c2^{-\sqrt{n} }.
\end{align*}
Now we use the scaling property of the Brownian motion to see that $W'$ has the same distribution as $W[0, \tau^{W}_{2^{n + m + 1}} ]$ assuming that $W (0) = -x^{n+1}$. Therefore the law of $W'_{ 2^{ (1- \epsilon ) n} }$ is same as the law of the $2^{ (1- \epsilon ) n}$-neighborhood of $W[0, \tau^{W}_{2^{n + m + 1}} ]$ with $W (0) = -x^{n+1}$. Thus we have 
\begin{align*}
&E^{-x^{n} } \Big\{ P^{x^{n+1} } \Big( LE ( S^{2} [0, \tau^{2}_{2^{n + m +1}} ])  \cap W'_{ 2^{ (1- \epsilon ) n} } = \emptyset \Big)  \Big\} \\
&= P^{-x^{n+1} , x^{n+1} } \Big( LE ( S^{2} [0, \tau^{2}_{2^{n + m +1}} ])  \cap W_{2^{ (1- \epsilon ) n}} [0, \tau^{W}_{2^{n + m + 1}} ] = \emptyset \Big).
\end{align*}
As in the proof of \eqref{bmsw}, we can show that the difference between the probability that the LERW and the Wiener sausage do not intersect and the probability that the LERW and the simple random walk do not intersect is small. So we see that there exists $\delta > 0$ such that 
\begin{equation*}
\big| P^{-x^{n+1} , x^{n+1} } \Big( LE ( S^{2} [0, \tau^{2}_{2^{n + m +1}} ])  \cap W_{2^{ (1- \epsilon ) n}} [0, \tau^{W}_{2^{n + m + 1}} ] = \emptyset \Big) - a_{m, n+1} \big| \le c 2^{-\delta n}.
\end{equation*}
Combining these estimates, we can conclude that 
\begin{align*}
&a_{m,n} \ge  b_{m,n} - c 2^{- \delta n} \\
&\ge P^{-x^{n+1} , x^{n+1} } \Big( LE ( S^{2} [0, \tau^{2}_{2^{n + m +1}} ])  \cap W_{2^{ (1- \epsilon ) n}} [0, \tau^{W}_{2^{n + m + 1}} ] = \emptyset \Big)  -  c2^{-\sqrt{n} } \\
&\ge a_{m, n+1} -  c2^{-\sqrt{n} }.
\end{align*}
Similar arguments as above gives that $a_{m , n+1} \ge a_{m, n} - c2^{-\sqrt{n} } $, and thus we have
\begin{equation}\label{cauchy}
|a_{m,n} - a_{m,n+1}| \le c2^{-\sqrt{n} },
\end{equation}
which implies that $\{ a_{m, n} \}_{n} $ is a Cauchy sequence and we finish the proof.
\end{proof}

\begin{rem}
We expect that $a_{m}$ can be written in terms of non-intersection probability of Brownian motion and scaling limit of loop-erased random walk in \cite{Koz}.
\end{rem}

\begin{cor}\label{subad}
Let $d=3$. There exists $\alpha > 0$ such that 
\begin{equation}
a_{n} \asymp 2^{-\alpha n}.
\end{equation}
\end{cor}

\begin{proof}
By Proposition \ref{step1} and Proposition \ref{main lem}, we have 
\begin{equation*}
a_{m+n} \asymp a_{m} a_{n}.
\end{equation*}
Using Lemma 8.5 of \cite{Law-conf}, we get the result.
\end{proof}

\subsection{Proof of \eqref{rate-of-conve-esp}}

\begin{thm}\label{main result-kaetta}
Suppose that $d=3$. Let $\alpha$ be the positive number as in Corollary \ref{subad}. Then we have 
\begin{equation}\label{Goal1}
\text{Es} (n) \approx n^{- \alpha}.
\end{equation}
In particular, we have 
\begin{equation}\label{Goal1-1}
\frac{1}{3} \le \alpha < 1.
\end{equation}
\end{thm}

\begin{proof}
Note that \eqref{Goal1-1} follows from \eqref{Goal1}. Indeed suppose that we get \eqref{Goal1}. The definition of $Es (n)$ and Theorem 1.3 of \cite{L} give that 
\begin{equation*}
Es (n) \ge P \big( S^{1} [1, \tau^{1}_{n} ] \cap S^{2} [0, \tau^{2}_{n} ] = \emptyset \big) \ge c n^{-\xi_{3}},
\end{equation*}
where we recall that $\xi_{3}$ is the intersection exponent with $\xi_{3} \in (\frac{1}{2}, 1 )$ (see \eqref{intersection-exp} for the intersection exponent). This implies that $\alpha \le \xi_{3} < 1$. For the lower bound on $\alpha$, the estimates (11.10) and (11.11) of \cite{Law b2} give that $Es (n) \le c n^{-\frac{1}{3}}$. This implies that $\alpha \ge \frac{1}{3}$. Thus in order to finish the proof of the theorem, it suffices to show \eqref{Goal1}.

Let $\epsilon > 0$.  By Corollary \ref{subad}, there exist $ 0 < c_{1}, c_{2} < \infty$ such that for all $m$ 
\begin{equation*}
c_{1} 2^{- \alpha  m   } \le a_{m} \le c_{2} 2^{- \alpha  m }.
\end{equation*}
We fix a large constant $M=M_{\epsilon }$ depending on $\epsilon$. The precise form of $M$ will be defined later. By Proposition \ref{main lem}, we know that $a_{M, n}$ converges to $a_{M}$ as $n \to \infty$. Therefore we can take $N = N_{M}$ depending on $M$ such that $\frac{1}{2} a_{M} \le a_{M, n} \le 2 a_{M}$ for all $n \ge N$. Consequently for all $n \ge N$, we have
\begin{equation*}
\frac{c_{1}}{2} 2^{- \alpha  M   } \le a_{M, n} \le 2 c_{2} 2^{- \alpha  M }.
\end{equation*}
On the other hand, using Proposition \ref{step1} for the case that $k=0$, we see that $Es (2^{n} ) a_{M, n} $ is comparable to $Es (2^{n + M})$. However, Proposition \ref{up to const indep1} and Proposition \ref{up to const indep2} show that $Es (2^{n + M})$ is comparable to $Es (2^{n} ) Es (2^{n}, 2^{n + M} )$. Consequently we see that $a_{M, n} $ is comparable to $Es (2^{n}, 2^{n + M} )$. Thus there exists $c_{0} > 0$ such that 
\begin{equation}\label{ikanainoka}
c_{0} 2^{- \alpha  M   } \le Es (2^{n}, 2^{n + M} ) \le \frac{1}{c_{0}} 2^{- \alpha  M   },
\end{equation}
for all $n \ge N$.

For $n \ge N$, we write $n = N + j M + r$ with $j \ge 0$ and $0 \le r < M$. Hence by Proposition \ref{up to const indep1},
\begin{align}\label{kamikazari}
&Es ( 2^{n}) = Es ( 2^{N + jM + r } ) \le C^{j+1} Es ( 2^{N+r})  \prod_{k=1}^{j} Es ( 2^{N+r + (k-1) M} , 2^{N + r + kM } ) \notag \\
&\le (C/c_{0})^{j+1} Es ( 2^{N+r})  2^{- \alpha M j },
\end{align}
where $C$ is a constant as in Proposition \ref{up to const indep1}. Now we choose $M$ so that $\frac{C}{c_{0}} < 2^{\epsilon M}$. This choice of $M$ ensures that the right hand side of \eqref{kamikazari} is bounded above by $C_{M} 2^{-( \alpha - \epsilon ) n}$. This gives that 
\begin{equation*}
\limsup_{n \to \infty} \frac{ \log Es ( 2^{n}) }{ \log 2^{n} } \le -\alpha + \epsilon.
\end{equation*}
Since $\epsilon > 0$ is an arbitrary positive number, we have
\begin{equation*}
\limsup_{n \to \infty} \frac{ \log Es ( 2^{n}) }{ \log 2^{n} } \le -\alpha.
\end{equation*}
Similar arguments as above also show that 
\begin{equation*}
\liminf_{n \to \infty} \frac{ \log Es ( 2^{n}) }{ \log 2^{n} } \ge -\alpha.
\end{equation*}
For a general integer $n$, we find $m$ with $2^{m} \le n < 2^{m+1}$. Using Proposition \ref{up-to-const lew} to see that
\begin{equation*}
Es (2^{m} ) \asymp Es ( n) \asymp Es ( 2^{m+1} ),
\end{equation*}
we get \eqref{Goal1}.
\end{proof}

Before finishing this section, we will give the following two lemmas. These lemmas will be used in the next section when we estimate the $k$-th moment of the length of LERW.
We begin with the next lemma which says that $Es (k, l)$ is of order $(\frac{l}{k})^{-\alpha}$ where $\alpha$ is the constant as in Theorem \ref{main result-kaetta}. This lemma is an analog of Lemma 3.12 of \cite{BM} for $d=3$.

\begin{lem}\label{mendo1}
Let $d=3$. Recall that $\alpha$ is the constant as in Theorem \ref{main result-kaetta}. Then
for all $\epsilon > 0$, there exist $c_{\epsilon} > 0$ and $n_{\epsilon} \in \mathbb{N}$ such that 
\begin{equation}
c_{\epsilon} \Big( \frac{l}{k} \Big)^{- \alpha - \epsilon } \le Es (k, l) \le c_{\epsilon}^{-1} \Big( \frac{l}{k} \Big)^{- \alpha + \epsilon },
\end{equation}
for all $n_{\epsilon} \le k \le l$. 
\end{lem}

\begin{proof}
Since the proof is completely same as the proof of Lemma 3.12 of \cite{BM}, we will give a sketch here. 

Let $\epsilon > 0$.
We take a large constant $j = j_{\epsilon}$ depending on $\epsilon$. The precise form of $j$ will be defined later. As in \eqref{ikanainoka}, we see that $C^{-1} j^{ - \alpha } \le Es (n, j n) \le C j^{ - \alpha }$ for all large $n$. We choose $i$ so that $j^{i} \le \frac{l}{k} < j^{i + 1}$. Then similar estimates as in \eqref{kamikazari} show that 
\begin{equation*}
 Es (k, l) \le C Es (k, j^{i} k ) \le C^{i +1} \prod_{q=0}^{i-1} Es (j^{q} k, j^{q+1} k ) \le C^{2 i +1} (j^{- \alpha } )^{i}.
 \end{equation*}
Now we take $j$ so that $j^{\frac{\epsilon}{2}} \ge C$ where $C$ is a constant in the inequality above. This choice of $j$ ensures that the right hand side of the inequality above is bounded above by $C_{j} \Big( \frac{l}{k} \Big)^{- \alpha + \epsilon }$. The lower bound of $Es (k, l)$ can be proved similarly. So we finish the proof.
\end{proof}

The next lemma gives a bound on the ratio of $Es (k)$ and $Es (l)$. This is an analog of Lemma 3.13 of \cite{BM} for $d=3$. We will use the next lemma many times in the next section.

\begin{lem}\label{mendo2}
Let $d=3$. Recall that $\alpha$ is the constant as in Theorem \ref{main result-kaetta}. Then
for all $\epsilon > 0$, there exists $C_{\epsilon} < \infty$ such that 
\begin{equation}
k^{\alpha + \epsilon} Es (k) \le C_{\epsilon}  l^{\alpha + \epsilon} Es (l),
\end{equation}
for all $1 \le k \le l$.
\end{lem}

\begin{proof}
We will give a sketch here. Take $\epsilon > 0$. By Theorem \ref{main result-kaetta} and Lemma \ref{mendo1}, we see that there exist constants $c > 0$ and $n $ such that for all $n \le k \le l$, $Es (k, l)$ is bounded below by $c ( \frac{l}{k} )^{- \alpha - \epsilon }$ and $Es (k )$ is bounded below by $c k^{- \alpha - \frac{\epsilon}{2} }$. Therefore using this along with Proposition \ref{up to const indep2}, we see that $l^{\alpha + \epsilon} Es (l)$ is bounded below by 
\begin{equation*}
c l^{\alpha + \epsilon} Es (k) Es (k, l) \ge c l^{\alpha + \epsilon} Es (k) ( \frac{l}{k} )^{- \alpha - \epsilon } = c k^{\alpha + \epsilon} Es (k).
\end{equation*}
for $n \le k \le l$.

For the case that $ k \le  l \le n$, it is easy to check that the claim holds. For the case that $k \le n \le l$, we point out that $k^{\alpha + \epsilon} Es (k) \le C_{n}$ for all $k \le n$ where $C_{n}$ is a constant depending on $n$. However, for every $l \ge n$, we know that $l^{\alpha + \epsilon} Es (l)$ is bounded below by $c l^{\alpha + \epsilon} \times l^{- \alpha - \frac{\epsilon}{2} } = c l^{ \frac{\epsilon}{2}} \ge c$. Thus the claim also holds for this case. So we finish the proof.
\end{proof}

\section{Tail estimates of the length of LERW}
Recall that in order to prove \eqref{regular} of Theorem \ref{2-dim}, we used exponential tail bounds on the length of LERW in 2 dimensions derived in \cite{BM} (see Proposition \ref{upper-2dim-loop} and Proposition \ref{lower-2dim-loop}). Unfortunately, such tail bounds in 3 dimensions have not been established up to now. The main goal of this section to derive both upper and lower tail bounds on the length of LERW in 3 dimensions. 

In the next subsection, we will give an upper tail estimate in Theorem \ref{dounaruno}. Then in Section 8.2, we will give a lower tail estimate in Theorem \ref{yattoda}.

\subsection{Upper tail estimates}
Recall that $M_{n} = \text{len} LE ( S[0, \tau_{n} ])$ stands for the length of LERW assuming that $S (0) = 0$. The goal of this subsection is to derive the following upper exponential tail bounds on $M_{n}$ in 3 dimensions: there exist $ 0 < c, C < \infty$ such that for all $n \ge 1$ and $\kappa > 0$, 
\begin{equation}\label{choco-kansou}
P \Big( M_{n} \ge \kappa E (M_{n} ) \Big) \le C e^{- c \kappa }.
\end{equation}
This is an analog of Theorem 5.8 of \cite{BM} in 3 dimensions. To establish \eqref{choco-kansou}, we will follow the same strategy as in the proof of Theorem 5.8 of \cite{BM}. Before going to its proof, we recall the strategy of \cite{BM} here. The first key step is to give an upper bound on the $k$-th moment of $M_{n}$ in terms of the escape probability, i.e., we will show that for all $k \ge 1$
\begin{equation}\label{hide-last}
E ( M_{n}^{k} ) \le C^{k} k ! (n^{2} Es (n) )^{k},
\end{equation}
where $C$ is some universal constant (see Theorem \ref{leipzig}). The term $n^{2} Es (n)$ comes from the following reason. We first point out that the expectation of $M_{n}$ is equal to the sum of the probability that $LE ( S[0, \tau_{n} ])$ hits $z$, where the sum is over all $z \in B (n)$. It turns out that the sum is comparable to 
\begin{equation*}
\sum_{ \frac{n}{3} \le |z| \le \frac{2 n}{3} } P \big( z \in LE ( S[0, \tau_{n}] ) \big),
\end{equation*}
namely, the sum of the same probabilities for $z$ which is not too close to the origin and the boundary of the ball. Suppose that $\frac{n}{3} \le |z| \le \frac{2 n}{3}$. In order for $z$ to lie in $LE ( S[0, \tau_{n} ] )$, it is required that 
\begin{itemize}
\item[(i)] $S$ hits $z$ up to $\tau_{n}$.

\item[(ii)] The loop-erasure of the random walk from the origin to $z$ and the random walk from $z$ up to $\tau_{n}$ do not intersect.
\end{itemize}
Reversing a path and using the time reversibility of LERW (see Lemma 7.2.1 of \cite{Law b}), we see that the probability of $\text{(i)} \cap \text{(ii)}$ is same as the probability that a random walk starting from $z$ up to the first exit of $B (n)$ and the loop-erasure of an independent random walk from $z$ to the origin do not intersect. We will show that this probability is comparable to $\frac{1}{n} Es (n)$, i.e., the product of the probability that the random walk from $z$ hits the origin and the escape probability. Taking the sum for $z$, it follows that $E ( M_{n} )$ is comparable to $n^{2} Es (n)$ and thus we have 
\begin{equation}\label{kataitaino}
E ( M_{n}^{k} ) \le C^{k} k ! \big( E ( M_{n} ) \big)^{k},
\end{equation}
which is sufficient to get \eqref{choco-kansou} (see Theorem \ref{dounaruno}).

\medskip

We begin with the next proposition (Proposition \ref{barlow1}) which gives an exact expression of the probability that LERW hits given $k$ points $z_{1}, \cdots , z_{k}$ in this order in terms of Green's functions and non-intersecting probabilities of random walks and loop-erased random walks. In order to state it, we need some definitions.

Let $z_{0}, z_{1}, \cdots , z_{k}$ be any distinct $k+1$ points in a given set $D \subset \mathbb{Z}^{3}$.  We write $X$ for a Markov chain on $\mathbb{Z}^{3}$ with $X(0)=z_{0}$ and $P^{z_{0}} (\tau^{X}_{D} < \infty ) =1$ (recall that $\tau^{X}_{D}$ stands for the first time that $X$ exits from $D$). We should write $P^{z_{0}}_{X}$ for the probability of $X$ instead of $P^{z_{0}}$. However, to avoid complication of notation, we will omit the superscript $X$ throughout this section. We let $\gamma = LE ( X [0, \tau^{X}_{D}])$ be the loop-erasure of $X$ up to its first exit of $D$. We are interested in the following event;
\begin{equation}\label{peropero}
F^{X}_{z_{0}, \cdots , z_{k}} = \{ \text{there exist } 0\le t_{1} < \cdots < t_{k} \le \text{len}\gamma \ \text{s.t. }  \gamma (t_{i} ) = z_{i}, \ \forall i=1,2, \cdots, k \},
\end{equation}
which is the event that $\gamma$ passes through points $z_{0}, z_{1}, \cdots , z_{k}$ in this order. We write $z_{k+1}= \partial D$.

We consider several independent versions of $X$ as follows. For $i=0, 1, \cdots, k$, we let $X^{i}$ be independent versions of $X$ with $X^{i}(0) = z_{i}$. We set $Z^{i}$ for $X^{i}$ conditioned on the event $\{ \sigma^{X^{i}}_{z_{i+1}} \le \tau^{X^{i}}_{D} \}$ (recall that $\sigma^{X^{i}}_{z} = \inf \{ t \ge 1 \ | \ X^{i} (t) = z \}$). So $Z^{i}$ is a process from $z_{i}$ to $z_{i+1}$. We write $u (i) = \max \{ l \le \tau^{Z^{i}}_{D} \ : \ Z^{i}_{l} = z_{i+1} \}$ for the last time that $Z^{i}$ visits to $z_{i+1}$ up to its first exit of $D$. Finally, define a non-intersecting event by
\begin{equation}\label{piro}
F^{k}_{i} = \Big\{ LE ( Z^{i-1}[0, u (i-1)]) \cap \bigcup_{j=i}^{k} Z^{j}[1, u (j)] = \emptyset \Big\}.
\end{equation}

The next proposition writes the probability of $F^{X}_{z_{0}, \cdots , z_{k}}$ in terms of Green's functions of $X$ and the probability of the non-intersecting event $F^{k}_{i}$.

\begin{prop}\label{barlow1}
Let $d= 3$. Recall that $z_{i}, D,  X, F^{X}_{z_{0}, \cdots , z_{k}} , X^{i}, Z^{i},  u (i)$ and $F^{k}_{i}$ were defined as above. Then we have
\begin{equation}
P ( F^{X}_{z_{0}, \cdots , z_{k}} ) = \Big[ \prod_{i=1}^{k} G^{X} (z_{i-1} , z_{i}, D ) \Big] P \Big( \bigcap_{i=1}^{k} F^{k}_{i} \Big).
\end{equation}
Here $G^{X} (\cdot , \cdot , D )$ is Green's function in $D$ for a Markov chain $X$.
\end{prop}

\begin{proof}
Proposition 5.2 of \cite{BM} claims the same statement as above in 2 dimensions. However clearly the proof of Proposition 5.2 of \cite{BM} also works in 3 dimensions. So we omit the proof.
\end{proof}

In the next proposition, we will give an upper bound on the probability that the loop-erasure of a Markov chain $X$ hits $k$ points $z_{1}, \cdots , z_{k}$ in terms of escape probabilities. To state it, we introduce some notation here. Take a set $D$ and suppose that $z_{0}, \cdots , z_{k}$ are points (not necessarily distinct) lying in $D$. We write $\vec{z} = (z_{0}, \cdots , z_{k})$. For this pair of $k+1$ points, we are interested in the distance between $z_{i}$ and $\{ z_{i -1}, z_{i+1} \} \cup \partial D$. Namely we set 
\begin{equation*}
d^{\vec{z}}_{i} = |z_{i} - z_{i -1} | \wedge |z_{i} - z_{i + 1} | \wedge \text{dist} (z_{i}, \partial D )
\end{equation*}
for $i = 1, \cdots , k$ (recall that $z_{k+1} = \partial D$). We need to consider permutations of $z_{1}, \cdots , z_{k}$ in the next proposition. So we let $\Pi_{k}$ be the symmetric group on $\{ 1, 2, \cdots , k \}$ and for each element $\pi \in \Pi_{k}$ we write $\pi (0) = 0$ and $\pi ( \vec{z} ) = (z_{0}, z_{\pi (1)}, \cdots , z_{\pi (k)} )$ for the corresponding permutation of $\vec{z}$.

Now we are ready to state the next proposition which estimates the probability that points $z_{1}, \cdots , z_{k}$ lie in $LE ( X [0, \tau^{X}_{D} ] )$. This proposition is an analog of Proposition 5.5 of \cite{BM} for $d=3$.

\begin{prop}\label{pukosukegairu}
Let $d=3$. We consider either
\begin{itemize}
\item[(I)] $D = B (n)$, $z_{0} = 0$, $z_{1}, \cdots , z_{k}$ are points in $B (n)$, and $X$ is a simple random walk $S$ started at the origin; or

\item[(II)] Recall that $m, n , N$, $A_{m}$, $x$ and $A_{n} (x)$ were defined as in Definition \ref{set-near-cube}. We consider a subset $K \subset A_{m}$. Suppose that $X$ is a random walk starting from $x$ conditioned that $X [1, \tau^{X}_{N} ] \cap K = \emptyset$. We let $D = B (N)$ and $z_{0} = x$. Points $z_{1}, \cdots , z_{k}$ lie in $A_{n} (x)$. 
\end{itemize} 
Then there exists a universal constant $C < \infty$ such that
\begin{equation}\label{pukosuke-kieta}
P \big( z_{1}, \cdots , z_{k} \in LE ( X [0, \tau^{X}_{D} ] ) \big) \le C^{k} \sum_{ \pi \in \Pi_{k} } \prod_{i =1}^{k} G^{X} (z_{\pi (i-1)}, z_{\pi (i)}, D ) Es ( d^{\pi ( \vec{z} )}_{\pi (i)} ).
\end{equation}
\end{prop}  

\begin{proof}
We will follow the same strategy as in the proof of Proposition 5.5 of \cite{BM}. We will only consider the first case (I). The claim for the second case will be proved similarly. 

Suppose that $X$ is a simple random walk $S$ started at the origin, $z_{0} = 0$, and $D = B (n)$. We first consider the case that $z_{1}, \cdots , z_{k}$ lie in $D$ and they are distinct. Recall that the event $F^{X}_{z_{0}, \cdots , z_{k}}$ was defined as in \eqref{peropero}. This definition immediately gives that the probability in the left hand side of \eqref{pukosuke-kieta} is equal to 
\begin{equation}\label{pero-1}
\sum_{\pi \in \Pi_{k} } P \big( F^{X}_{z_{0}, z_{\pi (1)}, \cdots , z_{\pi (k)} } \big).
\end{equation}
Thus, by using Proposition \ref{barlow1}, in oder to prove \eqref{pukosuke-kieta}, it suffices to show that 
\begin{equation}\label{pero-2}
P \Big( \bigcap_{i=1}^{k} F^{k}_{i} \Big) \le C^{k} Es ( d^{ \vec{z} }_{i} ),
\end{equation}
where $F^{k}_{i}$ was defined as in \eqref{piro}. But the definition of $F^{k}_{i}$ immediately gives that the probability in the left side of \eqref{pero-2} is bounded above by 
\begin{equation}\label{pero-3}
P \Big( \bigcap_{i=1}^{k} \big\{ LE ( Z^{i-1}[0, u (i-1)]) \cap  Z^{i}[1, u (i)] = \emptyset \big\} \Big).
\end{equation}
In order to estimate the probability above, we will consider a time reverse of a path. For a path $\lambda = [\lambda (0), \lambda (1), \cdots , \lambda (l)]$, we write $\lambda^{R} = [\lambda (l), \lambda (l-1), \cdots , \lambda (0)]$ for its time reversal. Then by the time reversibility of LERW (see Lemma 7.2.1 of \cite{Law b}), we see that the probability in \eqref{pero-3} is equal to 
\begin{equation}\label{pero-4}
P \Big( \bigcap_{i=1}^{k} \big\{ LE ( Z^{i-1}[0, u (i-1)]^{R} ) \cap  Z^{i}[1, u (i)] = \emptyset \big\} \Big).
\end{equation}
Namely we can replace the loop-erasure of $Z^{i-1}$ by the loop-erasure of its time reversal. We write $B^{i} = B (z_{i}, d^{ \vec{z} }_{i}/4 )$ and write $\gamma^{i}$ for $LE ( Z^{i-1}[0, u (i-1)]^{R} )$ from $z_{i}$ up to its first exit of $B^{i}$ for each $i= 1, \cdots , k$ (so that $\gamma^{i}$ is a subset of $LE ( Z^{i-1}[0, u (i-1)]^{R} )$). 

The domain Markov property (see Proposition \ref{DMP}) shows that conditioned on $\gamma^{i}$, the conditional distribution of $Z^{i-1}[0, u (i-1)]$ is same as the distribution of a random walk starting from $z_{i-1}$ conditioned that it hits $\gamma^{i} ( \text{len} \gamma^{i} )$ before hitting $\partial B (n)$ and $\gamma^{i} \setminus \{ \gamma^{i} ( \text{len} \gamma^{i} ) \}$. Using this fact along with the discrete Harnack principle (see Theorem 1.7.6 of \cite{Law b}), we see that the probability of \eqref{pero-4} is bounded above by
\begin{equation}\label{pero-5}
C^{k} \prod_{i=1}^{k} P \Big(  \gamma^{i} \cap  Z^{i}[1, \tau^{Z^{i}}_{B^{i}}] = \emptyset \Big),
\end{equation}
where we recall that $\tau^{Z^{i}}_{B^{i}}$ stands for the first time that $Z^{i}$ exits from $B^{i}$. But using the discrete Harnack principle (Theorem 1.7.6 of \cite{Law b})) again, we see that the law of $Z^{i}$ up to $\tau^{Z^{i}}_{B^{i}}$ is same as the law of a simple random walk staring from $z_{i}$ up to its first exit of $B^{i}$ up to multiplicative constants. In addition to that Proposition 4.2 and Proposition 4.4 of \cite{Mas} show that the distribution of $\gamma^{i}$ is same as the distribution of an infinite LERW starting from $z_{i}$ up to the first time that the infinite LERW exits from $B^{i}$ up to multiplicative constants. Therefore we see that the quantity of \eqref{pero-5} is bounded above by 
\begin{equation}\label{pero-6}
C^{k} \prod_{i=1}^{k} Es^{\infty} ( d^{ \vec{z} }_{i} /4 ).
\end{equation}
Finally, using Proposition \ref{up-to-const lew}, we see that the quantity of \eqref{pero-6} is bounded above by
\begin{equation}\label{pero-7}
C^{k} \prod_{i=1}^{k} Es( d^{ \vec{z} }_{i} ),
\end{equation}
which gives \eqref{pero-2}. 

For general points $z_{1}, \cdots , z_{k}$ (not necessarily distinct), the same argument as in the proof of Proposition 5.5 of \cite{BM} works here. So we omit the proof for that case.  
\end{proof}

In order to estimate the $k$-th moment of the length of LERW, we need the following Green's function estimates. The next lemma shows that for every point $x \in B(n)$ the sum of Green's functions $G (x, y, B_{n} )$ is bounded above by $C r^{2}$, where the sum is over all $y \in B (n)$ whose distance from $\partial B (n )$ is less than $r$.

\begin{lem}\label{leip}
Let $d=3$ and take $n$ and $r \ge 1$ with $n > r$. We write 
\begin{equation*}
G (r) = \{ y \in B (n) \ | \ \text{dist} ( y, \partial B (n) ) \le r \}
\end{equation*}
for the set of points in $B (n)$ whose distance from $\partial B (n)$ is bounded above by $r$. Then there exists a universal constant $ C < \infty$ such that for all $x \in B(n)$,
\begin{equation}\label{leip-1}
\sum_{y \in G (r) } G ( x, y, B (n) ) \le C r^{2}.
\end{equation}
\end{lem}  

\begin{proof}
We will follow the same idea as in the proof of Lemma 4.1 of \cite{BM}. We define entrance and exit times $t_{i}$ and $u_{i}$ by $t_{1} = \min \{ j \ | \ S (j) \in G (r) \}$ and 
\begin{align*}
&u_{i} = \min \{ j \ge t_{i} \ | \ | S(j) - S( t_{i} ) | \ge 2r \}, \\
&t_{i + 1} = \min \{ u_{i} \le j < \tau_{n} \ | \ S (j) \in G (r) \},
\end{align*} 
for $i \ge 1$, where we take $t_{i+1} = \infty$ if the set as above is empty. Conditioned on $t_{i} < \infty$, the conditional expectation $E^{S ( t_{i} ) } \big( | u_{i} - t_{i} | \big)$ is bounded above by $C r^{2}$. Therefore, we see that 
\begin{equation*}
\sum_{y \in G (r) } G ( x, y, B (n) ) = E^{x} \Big( \sum_{j = 1}^{\tau_{n} -1} {\bf 1} \{ S (j) \in G (r) \} \Big) \le E^{x} \Big( \sum_{i = 1}^{\infty} (u_{i} - t_{i} ) \Big) \le C r^{2} \sum_{i = 1}^{\infty} P^{x} ( t_{i} < \infty ).
\end{equation*}
We write $\tau_{z, 2 r} = \min  \{ j \ | \ S(j) \notin B (z, 2r) \}$ for the first exit of $B (z, 2r )$. It is easy to check that there exists a universal constant $p > 0$ such that for all $z \in G (r)$, $P^{z} ( \tau_{n} < \tau_{z, 2 r} ) \ge p$. Thus we have $P^{x} (t_{i} < \infty ) \le (1 - p )^{i -1}$, which finishes the proof.
\end{proof}

Recall that $M_{n}$ stands for the length of $LE ( S[ 0, \tau_{n} ] )$. The next theorem gives a bound on the $k$-th moment of $M_{n}$ in terms of the escape probability. This theorem is an analog of Theorem 5.6 of \cite{BM} in 3 dimensions. 

\begin{thm}\label{leipzig}
Let $d=3$. It follows that there exists $C < \infty$ such that for all $n \ge 1$ and $k \ge 1$, 
\begin{equation}\label{leipzig-1}
E ( M_{n}^{k} ) \le C^{k} k ! \big( n^{2} Es (n) \big)^{k}.
\end{equation}
\end{thm}

\begin{proof}
Recall that $\Pi_{k}$ is the symmetric group on $\{ 1, 2, \cdots , k \}$ and that for each element $\pi \in \Pi_{k}$ we write $\pi (0) = 0$ and $\pi ( \vec{z} ) = (z_{0}, z_{\pi (1)}, \cdots , z_{\pi (k)} )$. Proposition \ref{pukosukegairu} shows that
\begin{align*}
&E ( M_{n}^{k} ) = \sum_{z_{1} \in B (n) } \sum_{z_{2} \in B (n) } \cdots \sum_{z_{k} \in B (n) } P \big( z_{1}, \cdots , z_{k} \in LE ( S [0, \tau_{n} ] ) \big) \\
&\le C^{k} \sum_{ \pi \in \Pi_{k} } \sum_{z_{1} \in B (n) } \sum_{z_{2} \in B (n) } \cdots \sum_{z_{k} \in B (n) } \prod_{i =1}^{k} G (z_{\pi (i-1)}, z_{\pi (i)}, B (n) ) Es ( d^{\pi ( \vec{z} )}_{\pi (i)} ) \\
&= C^{k} k ! \sum_{z_{1} \in B (n) } \sum_{z_{2} \in B (n) } \cdots \sum_{z_{k} \in B (n) } \prod_{i =1}^{k} G (z_{i-1}, z_{i}, B (n) ) Es ( d^{ \vec{z} }_{i} ).
\end{align*}
Thus in order to prove \eqref{leipzig-1}, it suffices to show that 
\begin{equation}\label{bath-tab}
\sum_{z_{1} \in B (n) } \sum_{z_{2} \in B (n) } \cdots \sum_{z_{k} \in B (n) } \prod_{i =1}^{k} G (z_{i-1}, z_{i}, B (n) ) Es ( d^{ \vec{z} }_{i} ) \le C^{k} \big( n^{2} Es (n) \big)^{k}.
\end{equation}

To achieve \eqref{bath-tab}, as in the proof of Theorem 5.6 of \cite{BM}, we are interested in only the terms involving $z_{k}$ in the sum of \eqref{bath-tab}. With this in mind, we let 
$g_{i} = G (z_{i-1}, z_{i}, B (n) ) Es ( d^{ \vec{z} }_{i} )$ and we write $G_{j} = \prod_{i=1}^{j} g_{i}$. We also set $d (z) = \text{dist} ( z, \partial B (n) )$. Then the definition of $d^{ \vec{z} }_{i}$ gives that 
\begin{equation*}
\prod_{i =1}^{k} G (z_{i-1}, z_{i}, B (n) ) Es ( d^{ \vec{z} }_{i} )  = G_{k-1} G (z_{k-1}, z_{k}, B (n) ) Es \Big( |z_{k} - z_{k-1} | \wedge d (z_{k} ) \Big).
\end{equation*}
Now we expand the sum of \eqref{bath-tab} and collect all terms involving $z_{k}$ as follows:
\begin{align*}
&\sum_{z_{1} \in B (n) } \sum_{z_{2} \in B (n) } \cdots \sum_{z_{k} \in B (n) } \prod_{i =1}^{k} G (z_{i-1}, z_{i}, B (n) ) Es ( d^{ \vec{z} }_{i} ) \\
&\le \sum_{z_{1} \in B (n) } \sum_{z_{2} \in B (n) } \cdots \sum_{z_{k-1} \in B (n) } G_{k-2} G (z_{k-2}, z_{k-1}, B (n) ) \\
& \ \ \times \sum_{z_{k} \in B (n ) } G (z_{k-1}, z_{k}, B (n) ) \Big( Es \big( |z_{k-1} - z_{k-2} | \wedge d (z_{k-1} ) \big) + Es (|z_{k-1} - z_{k} | ) \Big) \\
& \ \ \times \Big( Es (|z_{k} - z_{k-1} |) + Es (d ( z_{k} ) ) \Big),
\end{align*}
where we used $Es ( r \wedge R ) \le Es (r ) + Es (R)$ in the inequality above. Therefore, we need to estimate the following four terms:
\begin{itemize}
\item  $I_{1} = Es \big( |z_{k-1} - z_{k-2} | \wedge d (z_{k-1} ) \big) \sum_{z_{k} \in B (n ) } G (z_{k-1}, z_{k}, B (n) ) Es (|z_{k} - z_{k-1} |)$,

\item  $I_{2} = Es \big( |z_{k-1} - z_{k-2} | \wedge d (z_{k-1} ) \big) \sum_{z_{k} \in B (n ) } G (z_{k-1}, z_{k}, B (n) ) Es (d ( z_{k} ) )$,

\item  $I_{3} = \sum_{z_{k} \in B (n ) } G (z_{k-1}, z_{k}, B (n) ) Es (|z_{k} - z_{k-1} |)^{2} $,

\item  $I_{4} = \sum_{z_{k} \in B (n ) } G (z_{k-1}, z_{k}, B (n) ) Es (|z_{k} - z_{k-1} |) Es (d ( z_{k} ) ) $.

\end{itemize}
Using $2 r R \le r^{2} + R^{2}$, if we define $I_{5}$ by 
\begin{equation*}
I_{5} = \sum_{z_{k} \in B (n ) } G (z_{k-1}, z_{k}, B (n) ) Es (d ( z_{k} ) )^{2},
\end{equation*}
then $I_{4}$ is bounded above by $I_{3} + I_{5}$. We start by estimating $I_{3}$. To do it, we set $B^{1} = B (n) \cap B (z_{k-1}, \frac{n}{2} )$ and $B^{2} = B (n) \setminus B^{1}$. Note that Green's function $G (z_{k-1}, z_{k}, B (n) )$ is bounded above by $G (z_{k-1}, z_{k}, \mathbb{Z}^{3} )$ which is comparable to $ ( |z_{k -1} - z_{k}| + 1 )^{-1}$ (see Theorem 4.3.1 of \cite{Law b2} for it). Thus we have a bound on the sum for $z_{k} \in B^{1}$ as follows:
\begin{align*}
& \sum_{z_{k} \in B^{1} } G (z_{k-1}, z_{k}, B (n) ) Es (|z_{k} - z_{k-1} |)^{2} \le C \sum_{z_{k} \in B^{1} } ( |z_{k -1} - z_{k}| + 1 )^{-1} Es (|z_{k} - z_{k-1} |)^{2} \\ 
&\le C \sum_{r=1}^{n} r Es (r)^{2}.
\end{align*}
To estimate the last sum in the inequality above, we use Lemma \ref{mendo2}. By \eqref{Goal1-1}, we know that $\frac{1}{3} \le \alpha < 1$ and thus we can take $\epsilon > 0$ so that $1 - 2 \alpha - 2 \epsilon > -1$. Choosing $\epsilon > 0$ with this condition and applying Lemma \ref{mendo2} to this $\epsilon$, we see that the last sum in the inequality above is bounded above by
\begin{equation*}
C \sum_{r=1}^{n} r^{1- 2 \alpha - 2 \epsilon } n^{2 \alpha + 2 \epsilon } Es (n)^{2} \le C n^{2} Es (n)^{2}.
\end{equation*}
On the other hand, by Proposition \ref{up-to-const lew}, we see that $Es (|z_{k} - z_{k-1} |)^{2} \le C Es (n)^{2}$ for $z_{k} \in B^{2}$. Therefore the sum for $z_{k} \in B^{2}$ is bounded above by $C n^{2} Es (n)^{2}$. So we see that $I_{3} \le C n^{2} Es (n)^{2}$. Similar arguments as above give that $I_{1}$ is bounded above by
\begin{equation}\label{semina}
I_{1} \le C Es \big( |z_{k-1} - z_{k-2} | \wedge d (z_{k-1} ) \big) n^{2} Es (n).
\end{equation}

We next consider $I_{5}$. To estimate it, we recall that $G (r)$ was defined as in Lemma \ref{leip}. Note that $I_{5}$ is bounded above by 
\begin{equation*}
\sum_{j = 0}^{\log_{2} n} \sum_{z_{k} \in G ( 2^{j} ) \setminus G (2^{j-1} ) } G (z_{k-1}, z_{k}, B (n) ) Es (d ( z_{k} ) )^{2},
\end{equation*}
which is, by Proposition \ref{up-to-const lew}, less than
\begin{equation*}
C \sum_{j = 0}^{\log_{2} n} Es (2^{j} )^{2} \sum_{z_{k} \in G ( 2^{j} ) } G (z_{k-1}, z_{k}, B (n) ).
\end{equation*}
Applying Lemma \ref{leip} to the sum for $z_{k} \in G ( 2^{j} )$ above, we see that $I_{5}$ is bounded above by 
\begin{equation*}
C \sum_{j = 0}^{\log_{2} n} 2^{2j} Es (2^{j} )^{2}.
\end{equation*}
Since $\alpha < 1$, we can take $\epsilon >0$ so that $2- 2\alpha -2 \epsilon > 0$. By applying Lemma \ref{mendo2} to this $\epsilon$, we see that the sum above is bounded above by 
\begin{equation*}
C n^{2 \alpha + 2 \epsilon} Es (n)^{2} \sum_{j = 0}^{\log_{2} n} (2^{j})^{2 - 2\alpha - 2 \epsilon} \le C n^{2} Es (n)^{2}.
\end{equation*}
Thus we get $I_{5} \le C n^{2} Es (n)^{2}$. Similarly, we have 
\begin{equation}\label{sorosoroowaru}
I_{2} \le C Es \big( |z_{k-1} - z_{k-2} | \wedge d (z_{k-1} ) \big) n^{2} Es (n).
\end{equation}

Consequently, we see that 
\begin{align*}
&\sum_{z_{1} \in B (n) } \sum_{z_{2} \in B (n) } \cdots \sum_{z_{k} \in B (n) } \prod_{i =1}^{k} G (z_{i-1}, z_{i}, B (n) ) Es ( d^{ \vec{z} }_{i} ) \\
&\le C n^{2} Es (n) \sum_{z_{1} \in B (n) } \sum_{z_{2} \in B (n) } \cdots \sum_{z_{k-1} \in B (n) } G_{k-2} G (z_{k-2}, z_{k-1}, B (n) ) \\
& \ \ \times \Big[ Es \big( |z_{k-1} - z_{k-2} | \wedge d (z_{k-1} ) \big) + Es (n) \Big] \\
&\le C n^{2} Es (n) \sum_{z_{1} \in B (n) } \sum_{z_{2} \in B (n) } \cdots \sum_{z_{k-1} \in B (n) } G_{k-2} G (z_{k-2}, z_{k-1}, B (n) ) \\
& \ \ \times  Es \big( |z_{k-1} - z_{k-2} | \wedge d (z_{k-1} ) \big).
\end{align*}
Iterating this $k-1$ times, we get \eqref{leipzig-1}.
\end{proof}

By Theorem \ref{leipzig}, in order to prove that $E (M_{n}^{k} )$ is bounded above by $C^{k} k ! \big( E (M_{n} ) \big)^{k}$, we need to show that $E (M_{n} ) $ is bounded below by $c n^{2} Es (n)$. This is proved in the next proposition.

\begin{prop}\label{zhan}
Let $d=3$. Then there exists $c > 0$ such that for all $n \ge 1$, 
\begin{equation}\label{momom}
E (M_{n} ) \ge c n^{2} Es (n).
\end{equation}
\end{prop}

\begin{proof}
Since the number of points in $B ( \frac{2 n}{3} ) \setminus B ( \frac{n}{3} )$ is comparable to $n^{3}$, it suffices to prove that for $x \in B ( \frac{2 n}{3} ) \setminus B ( \frac{n}{3} )$, 
\begin{equation}\label{zhan-1}
P \big( x \in LE ( S[0, \tau_{n} ] ) \big) \ge  \frac{c}{n } Es (n).
\end{equation}
So suppose that $x \in B ( \frac{2 n}{3} ) \setminus B ( \frac{n}{3} )$. We write $Z$ for a random walk starting from the origin conditioned that it hits $x$ before exiting from $B (n)$. We let $u$ be the last time that $Z$ visits to $x$ up to its first exit of $B (n)$. Suppose that $S^{1}$ is a simple random walk started at $x$ which is independent of $Z$. Then Proposition \ref{barlow1} gives that 
\begin{equation*}
P \big( x \in LE ( S[0, \tau_{n} ] ) \big) = G (0, x, B (n) ) P \big( LE ( Z [0, u] ) \cap S^{1} [1, \tau^{1}_{n} ] = \emptyset \big). 
\end{equation*}
Proposition 1.5.10 of \cite{Law b} gives that $G (0, x, B (n) )$ is comparable to $\frac{1}{n}$. Thus it suffices to show that 
\begin{equation}\label{ayashii}
P \big( LE ( Z [0, u] ) \cap S^{1} [1, \tau^{1}_{n} ] = \emptyset \big) \ge c Es (n).
\end{equation}
However, by the time reversibility of LERW (see Lemma 7.2.1 of \cite{Law b}), the probability in \eqref{ayashii} is equal to 
\begin{equation*}
P \big( LE ( Z [0, u]^{R} ) \cap S^{1} [1, \tau^{1}_{n} ] = \emptyset \big).
\end{equation*}
Suppose that $Y$ is a random walk starting from $x$ conditioned that it hits the origin before exiting from $B (n)$, which is independent of $S^{1}$. We write $\sigma^{Y}_{0}$ for the first time that $Y$ hits the origin. Since the law of $LE ( Z [0, u]^{R} )$ is same as the law of $LE ( Y [0, \sigma^{Y}_{0} ] )$, the non-intersecting probability above is equal to 
\begin{equation}\label{odoruodoru}
P \big( LE ( Y [0, \sigma^{Y}_{0} ] ) \cap S^{1} [1, \tau^{1}_{n} ] = \emptyset \big).
\end{equation}
So we need to show that the probability of \eqref{odoruodoru} is bounded below by $c Es (n)$. Note that by definition, that probability is equal to 
\begin{equation*}
\frac{P^{x, x} \Big( S^{1} [1, \tau^{1}_{n} ] \cap LE ( S^{2} [0, \sigma^{2}_{0} ] ) = \emptyset, \ \sigma^{2}_{0} < \tau^{2}_{n} \Big) }{ P^{x} ( \sigma^{2}_{0} < \tau^{2}_{n} ) },
\end{equation*}
which is, by Proposition 1.5.10 of \cite{Law b}, comparable to 
\begin{equation*}
n P^{x, x} \Big( S^{1} [1, \tau^{1}_{n} ] \cap LE ( S^{2} [0, \sigma^{2}_{0} ] ) = \emptyset, \ \sigma^{2}_{0} < \tau^{2}_{n} \Big).
\end{equation*}
To estimate the probability above, we set $t^{i} = \inf \{ j \ | \ S^{i} (j) \notin B (x, \frac{n}{4} ) \}$ for the first time that $S^{i}$ exits from $B (x, \frac{n}{4} ) $ for each $i = 1, 2$. Recall that the separation lemma (Theorem \ref{sep lem 3dim}) shows that conditioned that $S^{1} [1, t^{1} ]$ and $ LE ( S^{2} [0, t^{2} ] ) $ do not intersect, they are ``well-separated" with positive conditional probability. Namely if we let 
\begin{equation*}
{\cal D} = \text{dist} \big( S^{1} (t^{1}), LE ( S^{2} [0, t^{2} ] ) \big) \wedge \text{dist} \big( S^{2} (t^{2}), S^{1} [1, t^{1} ] \big),
\end{equation*}
then there exists $c > 0$ such that 
\begin{equation*}
P^{x, x} \Big( S^{1} [1, t^{1} ] \cap LE ( S^{2} [0, t^{2} ] ) = \emptyset, \ {\cal D} \ge c n \Big) \ge c Es (n),
\end{equation*}
where we also used Proposition \ref{up-to-const lew}. Conditioned that $S^{1} [1, t^{1} ]$ and $ LE ( S^{2} [0, t^{2} ] ) $ do not intersect and they are separated, we can attach $S^{1} [t^{1}, \tau^{1}_{n} ]$ and $S^{2} [t^{2}, \sigma^{2}_{0} ]$ so that $S^{1} [1, \tau^{1}_{n} ]$ and $LE ( S^{2} [0, \sigma^{2}_{0} ] )$ do not intersect and $\sigma^{2}_{0} < \tau^{2}_{n}$ with conditional probability at least $\frac{c}{n}$. Therefore, we see that 
\begin{equation*}
 P^{x, x} \Big( S^{1} [1, \tau^{1}_{n} ] \cap LE ( S^{2} [0, \sigma^{2}_{0} ] ) = \emptyset, \ \sigma^{2}_{0} < \tau^{2}_{n} \Big) \ge \frac{c Es (n)}{n},
\end{equation*}
which finishes the proof.
\end{proof}

Now we are ready to give upper exponential tail bounds on $M_{n}$ in the next theorem.

\begin{thm}\label{dounaruno}
Let $d=3$. There exist $0 < c, C < \infty$ such that for all $k \ge 1$, $n \ge 1$ and $\kappa > 0$ the following holds:
\begin{align}
&E ( M_{n}^{k} ) \le C^{k} k ! \big( E (M_{n} ) \big)^{k}, \label{kuraku} \\
& P \big( M_{n} \ge \kappa E (M_{n} ) \big) \le 2 e^{- c \kappa }. \label{nattekita}
\end{align} 
\end{thm}

\begin{proof}
Theorem \ref{leipzig} and Proposition \ref{zhan} immediately give \eqref{kuraku}.

We let $c_{1} = \frac{1}{2C}$ where $C$ is a constant as in \eqref{kuraku}. Then we see that
\begin{equation*}
E \Big( \exp \Big\{ \frac{ c_{1} M_{n} }{ E (M_{n} )} \Big\} \Big) = \sum_{k = 0}^{\infty} \frac{ (c_{1})^{k} E (M_{n}^{k} ) }{k !  \big( E (M_{n} ) \big)^{k} } \le \sum_{k = 0}^{\infty} 2^{-k} = 2.
\end{equation*}
Therefore, \eqref{nattekita} follows from Markov's inequality.
\end{proof}

\subsection{Lower tail estimates}
In this subsection, we will give a lower tail estimate on $M_{n}$ the length of LERW in three dimensions. Namely, the goal of this section is to show that for any $\epsilon > 0$ there exist $0 < c= c_{\epsilon}, C=C_{\epsilon} < \infty$ such that for all $\kappa > 0$ and $n \ge 1$
\begin{equation}\label{yareyare}
P \Big( M_{n} \le \frac{E ( M_{n} ) }{\kappa} \Big) \le C \exp \big\{ - c \kappa^{\frac{1}{ 2- \alpha } - \epsilon} \big\},
\end{equation}
where $\alpha$ is the exponent as in Corollary \ref{subad}. This is an analog of Theorem 1.2 of \cite{BM} in 3 dimensions. To prove \eqref{yareyare}, we will follow the same strategy as in the proof of Theorem 1.2 of \cite{BM}. We will recall the idea of proof here briefly before going to the proof.

Let $k \in \mathbb{N}$ and we write $\gamma = LE ( S [0, \tau_{k n} ])$ for the loop-erasure of a simple random walk up to the first exit of $B (k n)$. For each $i = 1, \cdots , k$, we set $u_{i}$ for the first time that $\gamma$ exits from $B ( i n)$. We will show that there exist universal constants $c > 0$ and $p \in (0, 1)$ such that conditioned on $\gamma [0, u_{i} ]$, the conditional probability that the length of $\gamma [u_{i}, u_{i+1} ]$ is bounded below by $c E (M_{n})$ is bigger than $p$ for each $i$, i.e., 
\begin{equation}\label{yaninattekuru}
P \Big( \text{len} \gamma [ u_{i}, u_{i+1} ] \le c E (M_{n})  \ \big| \ \gamma [0, u_{i} ] \Big) \le p.
\end{equation}
This gives that 
\begin{equation}\label{yaninattekuru-1}
P \Big( M_{k n} \le c E (M_{n})   \Big) \le p^{k}.
\end{equation}
Once we have proved \eqref{yaninattekuru-1}, by choosing suitable $k$ and relating $E (M_{k n} )$ to $E ( M_{n} )$, we get \eqref{yareyare}.

In order to \eqref{yaninattekuru}, by the domain Markov property (see Proposition \ref{DMP}), we need to estimate the length of the loop-erasure of a random walk conditioned that it does not hit $\gamma [0, u_{i} ]$ before exiting $B (k n)$. We will estimate the first and second moments of its length. Then by using the second moment method, we will prove \eqref{yaninattekuru}.

\medskip

We will start by introducing a random walk conditioned not to hit a given set as follows.

\begin{dfn}\label{conditioned-random-walk-da}
Suppose that $m, n , N$, $A_{m}$, $x$, and $A_{n} (x)$ are as in Definition \ref{set-near-cube}. We take a subset $K \subset A_{m}$. Let $X$ be a random walk conditioned that $X [1, \tau^{X}_{N} ] \cap K = \emptyset$, where $\tau^{X}_{N}$ is the first time that $X$ exits from $B (N)$. We set $\eta$ for $LE ( X [0, \tau^{X}_{N} ] )$ up to the first time that $LE ( X [0, \tau^{X}_{N} ] )$ exits from $B (x, n)$. Finally we let $J^{K}_{m, n, N, x} = \sharp \big( \eta \cap A_{n} (x) \big)$ be the number of points lying in both $\eta$ and $A_{n} (x)$.
\end{dfn}

We are interested in the first and second moments of $J^{K}_{m, n, N, x}$ defined as above. The next lemma gives a lower bound on the probability that $\eta$ hits a given point lying in $A_{n} (x)$.

\begin{lem}\label{hakubai}
Let $d=3$. Suppose that $m, n , N$, and $x$ are as in Definition \ref{set-near-cube} and that $K$, $X$, $\eta$, and $J^{K}_{m, n, N, x}$ are as in Definition \ref{conditioned-random-walk-da}. Then there exists $c > 0$ such that for all $z \in A_{n} (x)$,
\begin{equation}\label{hakubai-1}
P \big( z \in \eta \big) \ge \frac{c}{n} Es (n).
\end{equation}
In particular, we have
\begin{equation}\label{kitano}
E \Big( J^{K}_{m, n, N, x} \Big) \ge c n^{2} Es (n).
\end{equation}
\end{lem}

\begin{proof}
The second inequality immediately follows from the first inequality. We will show \eqref{hakubai-1}.

Suppose that $z \in A_{n} (x)$. We write $Y$ for a random walk starting from $x$ conditioned that it hits $z$ without hitting both $K$ and $\partial B (N)$ which is independent of $X$. Let $u$ be the last time that $Y$ hits $z$ before hitting $\partial B (N)$. Then Proposition \ref{barlow1} gives that 
\begin{equation}\label{hakubai-2}
P \big( z \in \eta \big) = G^{X} ( x, z, B (N) ) P_{X, Y}^{z, x} \Big( LE ( Y [0, u] ) \cap X [1, \tau^{X}_{N} ] = \emptyset, \ LE ( Y [0, u] ) \subset B (x, n ) \Big),
\end{equation}
where $P_{X, Y}^{z, x}$ stands for the probability of $X$ and $Y$ assuming that $X (0) = z$, $Y (0) = x$ and that $X$ and $Y$ are independent.

We recall that $\ell$ is the half infinite line defined as in Definition \ref{set-near-cube}. We first estimate Green's function in \eqref{hakubai-2}. Note that $G^{X} ( x, z, B (N) )$ is bounded below by the probability that $X$ hits $z$ before $\tau^{X}_{N}$. The definition of $X$ gives that it is bounded below by
\begin{equation}\label{hakubai-3}
\frac{P^{x} \big( \sigma_{z} < \sigma_{K} \wedge \tau_{N} \big) }{P^{x} \big(  \sigma_{K} < \tau_{ B (x, \frac{n}{8})} \big)},
\end{equation}
where $\tau_{ B (x, \frac{n}{8})}$ is the first exit of $B (x, \frac{n}{8})$. However, Proposition \ref{Masson} shows that conditioned on $\sigma_{K} < \tau_{ B (x, \frac{n}{8})}$, with positive conditional positive probability, the first exit point from $ B (x, \frac{n}{8})$ lies in $A = \{ w \in \partial B (x, \frac{n}{8}) \ | \ \text{dist} (w, \ell) \le \frac{n}{20} \}$. Then with probability at least $\frac{c}{n}$, it hits $z$ before hitting $K$ and $\partial B (N)$. Thus we have
\begin{equation*}
P^{x} \big( \sigma_{z} < \sigma_{K} \wedge \tau_{N} \big) \ge \frac{c}{n} P^{x} \big(  \sigma_{K} < \tau_{ B (x, \frac{n}{8})} \big),
\end{equation*}
which gives that $G^{X} ( x, z, B (N) ) \ge \frac{c}{n}$.

Next we deal with the probability in the right hand side of \eqref{hakubai-2}. To estimate the probability, we will use Theorem \ref{sep lem 3dim}. The definitions of $X, Y$ and the time reversibility of LERW (see Lemma 7.2.1 of \cite{Law b}) give that the probability in the right hand side of \eqref{hakubai-2} is equal to 
\begin{equation}\label{hakubai-5}
\frac{P^{z, z} \Big( S^{1} [1, \tau^{1}_{N} ] \cap LE ( S^{2} [ 0, \sigma^{2}_{x} ] ) = \emptyset, \ \tau^{1}_{N} < \sigma^{1}_{K}, \ \sigma^{2}_{x} < \sigma^{2}_{K} \wedge \tau^{2}_{N} \Big)}{P^{z, z} \Big( \tau^{1}_{N} < \sigma^{1}_{K}, \ \sigma^{2}_{x} < \sigma^{2}_{K} \wedge \tau^{2}_{N} \Big) }.
\end{equation}
We let $x^{i} = S^{i} (\tau^{i}_{B (z, \frac{n}{20} )} )$ be the first exit point from $B (z, \frac{n}{20} )$ for each $S^{i}$. We set 
\begin{equation*}
\tilde{d} = \text{dist} \big( x^{1}, LE ( S^{2}[0,  \tau^{2}_{B (z, \frac{n}{20} )} ] ) \big) \wedge \text{dist} \big( x^{2}, S^{1}[0,  \tau^{1}_{B (z, \frac{n}{20} )} ] \big). 
\end{equation*}
Then Theorem \ref{sep lem 3dim} shows that conditioned on $S^{1}[1,  \tau^{1}_{B (z, \frac{n}{20} )} ] \cap LE ( S^{2}[0,  \tau^{2}_{B (z, \frac{n}{20} )} ] ) = \emptyset$, with positive conditional probability, $\tilde{d} \ge c n$ for some $c > 0$. Once they are separated, we can find a path $\lambda$ such that $\lambda (0) = x^{2}$, $\lambda (\text{len} \lambda ) = x$, $\lambda \subset B (x, n)$ and that a $\frac{n}{100}$-neighborhood of $\lambda$ and $S^{1}[0,  \tau^{1}_{B (z, \frac{n}{20} )}]$ do not intersect. We write $F^{2}$ for the $\frac{n}{100}$-neighborhood of $\lambda$. 

Conditioned on $S^{1}[1,  \tau^{1}_{B (z, \frac{n}{20} )} ] \cap LE ( S^{2}[0,  \tau^{2}_{B (z, \frac{n}{20} )} ] ) = \emptyset$ and $\tilde{d} \ge c n$, we first want to compare the probability that $S^{1}$ started at $x^{1}$ does not hit $K \cup LE ( S^{2}[0,  \tau^{2}_{B (z, \frac{n}{20} )} ] ) \cup F^{2}$ until its first exit of $B (N)$ with the probability that $S^{1}$ started at $z$ does not hit $K$ until its first exit of $B (N)$. Since $\tilde{d} \ge c n$, we see that 
\begin{equation*}
P^{x^{1} } \Big( S^{1} [0, \tau^{1}_{B (x, L n ) } ] \cap \big( K \cup LE ( S^{2}[0,  \tau^{2}_{B (z, \frac{n}{20} )} ] ) \cup F^{2} \big) = \emptyset, \ \text{dist} \big( S^{1} \big( \tau^{1}_{B (x, L n ) }  \big), \ell \big) \le \frac{L n}{20} \Big) \ge c_{L},
\end{equation*}
where $L$ is a large constant which will be defined later and $\ell$ is an half infinite line as in Definition \ref{set-near-cube}. We set $G$ for the set of points in $\partial B (x, L n)$ lying in a $\frac{L n}{20}$-neighborhood of $\ell$. Then by Proposition 1.5.10 of \cite{Law b} and Lemma \ref{reflect2}, we see that for all $v \in G$,
\begin{equation*}
P^{v} \Big( \tau^{1}_{N} < \sigma^{1}_{K}, \ \sigma^{1}_{B (x, n)} < \tau^{1}_{N} \Big) \le \frac{C}{L} P^{v} \Big( \tau^{1}_{N} < \sigma^{1}_{K} \Big),
\end{equation*}
for some universal constant $C$. Choosing $L$ large so that $\frac{C}{L} < \frac{1}{2}$, we see that the probability 
\begin{equation*}
P^{x^{1} } \Big( S^{1} [0, \tau^{1}_{N} ] \cap \big( K \cup LE ( S^{2}[0,  \tau^{2}_{B (z, \frac{n}{20} )} ] ) \cup F^{2} \big) = \emptyset \Big) 
\end{equation*}
is bounded below by
\begin{equation*}
c \min_{v \in G} P^{v } \Big( S^{1} [0, \tau^{1}_{N} ] \cap K = \emptyset \Big).
\end{equation*}
Using Lemma \ref{reflect2} again, this is bounded below by 
\begin{equation*}
c  P^{z } \big(  \tau^{1}_{N} < \sigma^{1}_{K} \big).
\end{equation*}

Conditioned on $S^{1}[1,  \tau^{1}_{B (z, \frac{n}{20} )} ] \cap LE ( S^{2}[0,  \tau^{2}_{B (z, \frac{n}{20} )} ] ) = \emptyset$ and $\tilde{d} \ge c n$, we next estimate the probability that $S^{2}$ started at $x^{2}$ satisfies that $S^{2} [0, \sigma^{2}_{x} -1] \cap K = \emptyset$ and that $S^{2} [0, \sigma^{2}_{x} ]$ is contained in $F^{2}$. Namely, we want to show that the probability 
\begin{equation*}
P^{x^{2}} \Big(  S^{2} [0, \sigma^{2}_{x} ] \cap K = \emptyset, \ S^{2} [0, \sigma^{2}_{x} ] \subset F^{2} \Big)
\end{equation*}
is comparable to the probability $P^{ z} \big(  \sigma^{2}_{x} < \sigma^{2}_{K} \wedge \tau^{2}_{N} \big)$. Reversing a path and using Lemma 3.1 of \cite{Mas}, in order to prove that those two probabilities are comparable, it suffices to show that the probability
\begin{equation*}
P^{x} \big( \sigma^{2}_{x^{2}} < \sigma^{2}_{K} \wedge \sigma^{2}_{x} \wedge \tau^{2}_{F^{2}} \big)
\end{equation*}
is comparable to the probability $P^{x} \big( \sigma^{2}_{z} < \sigma^{2}_{K} \wedge \sigma^{2}_{x} \wedge \tau^{2}_{N} \big)$. But using Proposition \ref{Masson}, we see that those probabilities are comparable.

Consequently, we see that the ratio in \eqref{hakubai-5} is bounded below by $c Es (\frac{n}{20} )$. Proposition \ref{up-to-const lew} shows that $Es (n)$ is comparable to $Es (\frac{n}{20} )$, so we finish the proof.
\end{proof}

Now we give a second moment estimate of $J^{K}_{m, n, N, x}$ in the next lemma. The next lemma shows that the second moment of $J^{K}_{m, n, N, x}$ is comparable to the square of its first moment, which allows us to use the second moment method.

\begin{lem}\label{katsura}
Let $d=3$. Suppose that $m, n , N$, and $x$ are as in Definition \ref{set-near-cube} and that $K$ and $J^{K}_{m, n, N, x}$ are as in Definition \ref{conditioned-random-walk-da}. Then there exists an absolute constant $C < \infty$ such that  
\begin{equation}\label{katsura-1}
E \Big( \big( J^{K}_{m, n, N, x} \big)^{2} \Big) \le C E \Big(  J^{K}_{m, n, N, x}  \Big)^{2}.
\end{equation}
In particular, there exists $c > 0$ such that  
\begin{equation}\label{katsura-2}
P \Big(  J^{K}_{m, n, N, x}  \ge c n^{2} Es (n) \Big) \ge c.
\end{equation}
\end{lem}

\begin{proof}
The second inequality follows from Lemma \ref{hakubai}, \eqref{katsura-1} and the second moment method. We will show \eqref{katsura-1}.

For $z, w \in A_{n}(x)$, we write $d^{1}_{z, w} = \text{dist} (z, \partial B (N) ) \wedge |z-x| \wedge |z-w|$ and $d^{2}_{z, w} = \text{dist} (z, \partial B (N) )  \wedge |z-w|$. Then by Proposition \ref{pukosukegairu}, we see that the second moment of $J^{K}_{m, n, N, x}$ is bounded above by
\begin{equation*}
 \sum_{z, w \in A_{n}(x)} P \big( z, w \in \eta \big) \le C \sum_{z, w \in A_{n}(x)} G^{X } (x, z, B (N) ) G^{X } (z, w, B (N) ) Es (d^{1}_{z, w}) Es (d^{2}_{z, w} ).
\end{equation*}
The definition of $A_{n}(x)$ gives that both $d^{1}_{z, w}$ and $d^{2}_{z, w}$ are comparable to $|z-w|$. Therefore, by using Proposition \ref{up-to-const lew}, we see that 
\begin{equation*}
E \Big( \big( J^{K}_{m, n, N, x} \big)^{2} \Big) \le C \sum_{z, w \in A_{n}(x)} G^{X } (x, z, B (N) ) G^{X } (z, w, B (N) ) Es (|z-w|)^{2}.
\end{equation*}

We will first estimate $ G^{X } (x, z, B (N) )$. Lemma 2.1 of \cite{BM} gives that $ G^{X } (x, z, B (N) )$ is equal to 
\begin{equation*}
G \big( z, z, B (N) \setminus K \big) \frac{P^{x} \big( \sigma_{z} < \sigma_{K} \wedge \tau_{N} \big) P^{z} \big( \tau_{N} < \sigma_{K} \big) }{P^{x} \big( \tau_{N} < \sigma_{K} \big)}.
\end{equation*}
Note that $G \big( z, z, B (N) \setminus K \big) \le C$. We write $A = \{ w \in \partial B (x, \frac{n}{8}) \ | \ \text{dist} (w, \ell) \le \frac{n}{20} \}$ as in the proof of Lemma \ref{hakubai}. Then the probability $P^{x} \big( \tau_{N} < \sigma_{K} \big)$ is bounded below by 
\begin{equation*}
P^{x} \Big( \tau_{N} < \sigma_{K}, \ S \big( \tau_{B (x, \frac{n}{8})} \big) \in A \Big),
\end{equation*}
which is, by Proposition \ref{Masson} and the strong Markov property, bounded below by 
\begin{equation*}
c P^{x} \Big( \tau_{B (x, \frac{n}{8})} < \sigma_{K} \Big) \min_{v \in A} P^{v} \Big ( \tau_{N} < \sigma_{K} \Big).
\end{equation*}
However, the discrete Harnack principle (see Theorem 1.7.6 of \cite{Law b}) gives that $\min_{v \in A} P^{v} \Big ( \tau_{N} < \sigma_{K} \Big)$ is comparable to $P^{z} \big( \tau_{N} < \sigma_{K} \big)$. On the other hand, using Proposition 1.5.10 of \cite{Law b}, we see that $P^{x} \big( \sigma_{z} < \sigma_{K} \wedge \tau_{N} \big)$ is bounded above by 
\begin{equation*}
 \frac{C}{n} P^{x} \Big( \tau_{B (x, \frac{n}{8})} < \sigma_{K} \Big).
\end{equation*}
Thus it follows that $ G^{X } (x, z, B (N) ) \le \frac{C}{n}$. Similarly, we see that $G^{X } (z, w, B (N) ) \le \frac{C}{|z-w|}$. Consequently, we have
\begin{equation*}
E \Big( \big( J^{K}_{m, n, N, x} \big)^{2} \Big) \le C \sum_{z, w \in A_{n}(x)} \frac{C}{n |z-w|}  Es (|z-w|)^{2}.
\end{equation*}

Since $\alpha < 1$ (see \eqref{Goal1-1} for this), we can choose $\epsilon > 0$ such that $1- 2 \alpha - 2 \epsilon > -1$. Now we apply Lemma \ref{mendo2} to this $\epsilon$ to show that 
\begin{align*}
&\sum_{z, w \in A_{n}(x)} \frac{C}{n |z-w|}  Es (|z-w|)^{2} \le \frac{C}{n} \sum_{z \in A_{n} (x) } \sum_{r =1}^{n} r Es (r)^{2} \\
&\le C n^{2} \sum_{r =1}^{n} r \Big( n^{\alpha + \epsilon} Es (n) r^{-\alpha - \epsilon} \Big)^{2} \le  C n^{2} \sum_{r =1}^{n} n^{2\alpha + 2\epsilon} Es (n)^{2} r^{1-2\alpha - 2\epsilon} \le C n^{4} Es(n)^{2}.
\end{align*}
But by Lemma \ref{hakubai}, we have $E \Big( J^{K}_{m, n, N, x} \Big) \ge c n^{2} Es (n)$ and finish the proof.
\end{proof}

Recall that $\gamma^{\infty} = LE ( S[0, \infty ) )$ stands for the infinite LERW and that $\tau^{\infty}_{n}$ is the first time that $\gamma^{\infty}$ exits from $B (n)$. The next lemma relates $E (M_{k n} )$ and $E (\tau^{\infty}_{k n} )$ to $E (M_{n} )$. 

\begin{lem}\label{yamashina}
Let $d=3$. Then for all $\epsilon > 0$ there exists $c_{\epsilon}$ such that for all $k \ge 1$ and $n \ge 1$, we have
\begin{align}
&E (M_{k n} ) \le c_{\epsilon} k^{2 - \alpha + \epsilon} E (M_{n} ) \label{yama-1} \\
&E (\tau^{\infty}_{k n}  ) \le c_{\epsilon} k^{2 - \alpha + \epsilon} E (M_{n} ), \label{yama-2}
\end{align}
where $\alpha$ is the exponent as in Corollary \ref{subad}.
\end{lem}

\begin{proof}
Since the proof of \eqref{yama-2} is similar to the proof of \eqref{yama-1}, we will only prove \eqref{yama-1}. Using Theorem \ref{leipzig} and Proposition \ref{up to const indep1}, we see that $E (M_{k n} )$ is bounded above by $C k^{2} n^{2} Es (n) Es (n , kn)$. On the other hand, Lemma \ref{mendo1} shows that there exists $c_{\epsilon}$ such that for all $n \ge 1$, $Es (n , kn)$ is bounded above by $c_{\epsilon} k^{- \alpha + \epsilon}$. Since $E (M_{n} ) \ge c n^{2} Es (n)$ by Proposition \ref{zhan}, we finish the proof.
\end{proof}

Using the domain Markov property (see Proposition \ref{DMP}) along with \eqref{katsura-2}, we get the following proposition which gives lower tail bounds on $M_{k n}$ and $\tau^{\infty}_{k n}$.

\begin{prop}\label{atsuine}
Let $d=3$. There exist $0 < c_{1}, c_{2} < \infty$ such that for all $k \ge 2$ and $n \ge 1$, we have 
\begin{align}
&P \big( M_{k n} \le c_{1} E ( M_{n} ) \big) \le e^{- c_{2} k}, \label{atsui-1} \\
&P \big( \tau^{\infty}_{k n} \le c_{1} E ( M_{n} ) \big) \le e^{- c_{2} k}. \label{atsui-2}
\end{align}
\end{prop}

\begin{proof}
We will only prove \eqref{atsui-1}. The second inequality \eqref{atsui-2} can be shown similarly.

Let $\gamma = LE ( S[ 0, \tau_{k n} ] )$. We set $k' = \lfloor \frac{k}{\sqrt{3}} \rfloor$. Recall that $A_{m} = [-m, m]^{3}$ stands for a cube of length $2m$ centered at the origin. We consider $k'$ cubes $A_{j n}$ ($j= 1, 2, \cdots , k'$). For each $j$, we let $t_{j}$ be the first time that $\gamma$ exits from $A_{j n}$. We are interested in $t_{j +1} -t_{j}$ which 
is the length of $\gamma [t_{j}, t_{j+1} ]$. Suppose that $M_{k n} \le c_{1} E ( M_{n} )$ (we will define $c_{1}$ later). This implies for all $j=1, \cdots , k'$, $t_{j +1} -t_{j}$ is bounded above by $c_{1} E ( M_{n} )$. The domain Markov property (see Proposition \ref{DMP}) gives that conditioned on $\gamma [0, t_{j} ]$, the law of $\gamma$ after time $t_{j}$ is 
same as the law of the loop-erasure of a random walk $X$ starting from $x_{j} := \gamma (t_{j} )$ conditioned that $X [1, \tau^{X}_{k n} ] \cap \gamma [0, t_{j} ] = \emptyset$. Therefore, conditioned on $\gamma [0, t_{j} ]$, the law of $t_{j +1} -t_{j}$ is same as the law of the first time that $LE ( X [0, \tau^{X}_{k n} ] )$ exits from $A_{(j+1) n}$, which is bounded below by $J^{\gamma[0, t_{j}]}_{j n, n, k n, x_{j}}$ (see Definition \ref{conditioned-random-walk-da} for $J^{K}_{m, n, N, x}$). Thus we have
\begin{align*}
&P \big( M_{k n} \le c_{1} E ( M_{n} ) \big) \le P \Big( \bigcap_{j=1}^{k'} \big\{ t_{j +1} -t_{j} \le c_{1} E ( M_{n} ) \big\} \Big) \\
&\le E \Big[ \bigcap_{j=1}^{k'-1} \big\{ t_{j +1} -t_{j} \le c_{1} E ( M_{n} ) \big\} P \big( J^{\gamma[0, t_{k'}]}_{k' n, n, k n, x_{k'}} \le c_{1} E ( M_{n} ) \big) \Big]
\end{align*}
However, Lemma \ref{katsura} shows that there exists an absolute constant $c_{1} > 0$ such that 
\begin{equation*}
P \big( J^{\gamma[0, t_{j}]}_{j n, n, k n, x_{j}} \le c_{1} E ( M_{n} ) \big) \le 1- c_{1},
\end{equation*}
for all $j$. Thus, for this $c_{1}$, we have
\begin{equation}\label{atsuiyo-1}
P \big( M_{k n} \le c_{1} E ( M_{n} ) \big) \le (1- c_{1} ) P \Big( \bigcap_{j=1}^{k'-1} \big\{ t_{j +1} -t_{j} \le c_{1} E ( M_{n} ) \big\} \Big).
\end{equation}
Iterating this, we see that the left hand side of \eqref{atsuiyo-1} is bounded above by $(1- c_{1})^{k'}$, which finishes the proof.
\end{proof}

Now we are ready to establish exponential tail bounds on $M_{n}$ and $\tau^{\infty}_{n}$. The next theorem is an analog of Theorem 6.7 of \cite{BM} in 3 dimensions.

\begin{thm}\label{yattoda}
Let $d=3$. Recall that $\alpha$ is the exponent as in Corollary \ref{subad}. For any $\epsilon \in (0,1)$, there exist $c=c( \epsilon ) > 0$ and $C=C(\epsilon ) < \infty$ such that for all $\kappa  \ge 1$ and $n \ge 1$, 
\begin{align}
&P \Big( M_{n} \le \frac{ E(M_{n} ) }{\kappa} \Big) \le C \exp \big( - c \kappa^{\frac{1}{2- \alpha} - \epsilon } \big), \label{yattoda-1} \\
&P \Big( \tau^{\infty}_{n} \le \frac{ E(M_{n} ) }{\kappa} \Big) \le C \exp \big( - c \kappa^{\frac{1}{2- \alpha} - \epsilon } \big). \label{yattoda-2}
\end{align}
\end{thm}

\begin{proof}
We will only show \eqref{yattoda-1}. The second inequality can be proved similarly.

We let $k= \kappa^{ \frac{1}{2-\alpha} - \epsilon}$ so that 
\begin{equation*}
k^{2 - \alpha + \epsilon } = \kappa^{1 + \frac{\epsilon}{2-\alpha} - (2- \alpha + \epsilon ) \epsilon }.
\end{equation*}
On the other hand, by Lemma \ref{yamashina}, we see that there exists $c_{\epsilon}$ such that $E ( M_{k n} )$ is bounded above by 
\begin{equation}\label{yattone}
c_{\epsilon} k^{2 - \alpha + \epsilon } E( M_{n} ) = c_{\epsilon} \kappa^{1 + \frac{\epsilon}{2-\alpha} - (2- \alpha + \epsilon ) \epsilon } E (M_{n} ),
\end{equation}
for all $ n \ge 1$. Since $\alpha < 1$ (see \eqref{Goal1-1}), we have $1 + \frac{\epsilon}{2-\alpha} - (2- \alpha + \epsilon ) \epsilon < 1$. Therefore, we can find a large constant $\kappa_{\epsilon} < \infty$ depending on $\epsilon$ such that the right hand side of \eqref{yattone} is bounded above by $c_{1} \kappa E (M_{n})$ for all $\kappa \ge \kappa_{\epsilon}$, where $c_{1}$ is the constant as in Proposition \ref{atsuine}. Consequently, by Proposition \ref{atsuine}, if $\kappa \ge \kappa_{\epsilon}$ we have
\begin{align*}
P \Big( M_{n} \le \frac{ E(M_{n} ) }{\kappa} \Big) = P \Big( M_{k (\frac{n}{k})} \le \frac{ E(M_{k (\frac{n}{k})} ) }{\kappa} \Big) \le P \Big( M_{k (\frac{n}{k})} \le c_{1}  E(M_{ \frac{n}{k}} )  \Big) \le e^{- c_{2} k} = e^{-c_{2} \kappa^{ \frac{1}{2-\alpha} - \epsilon}}.
\end{align*}
So we get the inequality \eqref{yattoda-1} for $n \ge 1$ and $\kappa \ge \kappa_{\epsilon}$. For the case that $1 \le \kappa \le \kappa_{\epsilon}$, it is easy to check that the inequality \eqref{yattoda-1} holds if we choose $C_{\epsilon}$ with $C_{\epsilon} \ge e^{c_{2} \kappa_{\epsilon}^{ \frac{1}{2-\alpha} - \epsilon}}$. So we finish the proof.
\end{proof}

\begin{rem}
It is conjectured (\cite{Wil2}) that 
\begin{equation}\label{karakara}
P \Big( M_{n} < \frac{ E(M_{n} ) }{\kappa} \Big) = C \exp \big( - c \kappa^{\frac{1}{2- \alpha} + o(1) } \big).
\end{equation}
We have proved in Theorem \ref{yattoda} that the left hand side is bounded above by the right hand side in \eqref{karakara}, but the other direction of the inequality remains open.
\end{rem}

\section{Improvement of estimates on len$LE (\overline{S} [0, \overline{T}_{n} ] )$ for $d=3$}

Recall that Theorem \ref{critical exp} shows that $\text{len} \big( LE (\overline{S} [0, \overline{T}_{n} ] ) \big)$ divided by $n^{\alpha_{\ell}(3) + \epsilon}$ converges to zero almost surely for all $\epsilon > 0$, and $\text{len} \big( LE (\overline{S} [0, \overline{T}_{n} ] ) \big)$ divided by $n^{\alpha_{\ell}(3) - \epsilon}$ diverges in the sense that the $\limsup$ of the ratio goes to infinity. The goal of this section is to improve the lower estimates to show that for all $\epsilon > 0$, 
\begin{equation}\label{anpontan}
\lim_{n \to \infty} \frac{ \text{len} \big( LE (\overline{S} [0, \overline{T}_{n} ] ) \big)}{n^{\alpha_{\ell}(3) - \epsilon}}= \infty, \text{ a.s.}
\end{equation}
Combining this with Theorem \ref{critical exp}, we have as $n \to \infty$
\begin{equation}\label{anpontan-1}
  \text{len} \big( LE (\overline{S} [0, \overline{T}_{n} ] ) \big) \approx n^{\alpha_{\ell}(3)}, \text{ a.s.}
\end{equation}
We will also show that $\alpha_{\ell}(3) = \frac{2- \alpha}{2- \xi_{3}}$ (see Corollary \ref{subad} for the exponent $\alpha$ and \eqref{intersection-exp} for $\xi_{3}$). These results will be obtained by by establishing Proposition \ref{upper-3dim-loop} and Proposition \ref{lower-3dim-loop} in Section 9.1 and Section 9.2, respectively.

\subsection{Upper bound for $\alpha_{\ell}(3)$}
In this subsection, we will show that $\alpha_{\ell}(3) \le \frac{2- \alpha}{2- \xi_{3}}$ by proving that $\text{len} \big( LE (\overline{S} [0, \overline{T}_{n} ] )$ divided by $n^{\frac{2- \alpha}{2- \xi_{3}} + \epsilon }$ converges to zero for all $\epsilon > 0$ in Proposition \ref{upper-3dim-loop}. The proof is completely same as Proposition \ref{upper-2dim-loop}. As in the proof of Proposition \ref{upper-2dim-loop}, we will use upper tail bounds on $M_{n}$ derived in Theorem \ref{dounaruno}.


\begin{prop}\label{upper-3dim-loop}
Let $d=3$. For all $b > \frac{2- \alpha}{2- \xi_{3}}$,
\begin{equation}
\overline{P} \Big( \lim_{n \to \infty} \frac{ \text{len} \big( LE (\overline{S} [0, \overline{T}_{n} ] ) \big) } { n^{b} } = 0 \Big) =1.
\end{equation}
In particular, $\alpha_{\ell}(3) \le \frac{2- \alpha}{2- \xi_{3}}$.
\end{prop}

\begin{proof}
Fix $\epsilon > 0$. We write $\overline{K}^{+}_{n}$ for the number of global cut times of $\overline{S}$ in $[0, \overline{\tau}^{+}_{n} ]$. In the proof of Theorem 1.1 of \cite{S}, it was shown that 
\begin{equation*}
n^{2- \xi_{3} - \epsilon } \le \overline{K}^{+}_{n} \le n^{2- \xi_{3} + \epsilon } \text{ for large } n, \ \overline{P}\text{-a.s.}
\end{equation*}
This gives that 
\begin{equation}\label{jirousan-33}
\overline{\tau}^{+}_{n^{\frac{1}{2- \xi_{3}} - 2 \epsilon }} \le \overline{T}_{n} \le \overline{\tau}^{+}_{n^{\frac{1}{2- \xi_{3}} + 2 \epsilon }} \text{ for large } n, \ \overline{P}\text{-a.s.}
\end{equation}

By Theorem \ref{leipzig} and Proposition \ref{zhan}, we have $E (M_{n} ) \asymp n^{2} Es (n)$. Furthermore, Theorem \ref{main result-kaetta} shows that $Es (n) \approx n^{-\alpha}$. Therefore, we see that $E (M_{n} ) \approx n^{2 - \alpha}$. Combining this with Theorem \ref{dounaruno}, we have 
\begin{equation}\label{uruseina-33}
P \Big( M_{n^{\frac{1}{2- \xi_{3}} + 2 \epsilon }}  \ge n^{ \frac{2- \alpha}{2- \xi_{3}} + 6 \epsilon } \Big) \le c_{0} e^{- c_{1} n^{\frac{\epsilon}{8}}},
\end{equation}
for some $0 < c_{0}, c_{1} < \infty$. 

Recall that Corollary 4.6 of \cite{L} gives that for $N > n^{\frac{1}{2- \xi_{3}} + 2 \epsilon}$
\begin{equation*}
\max_{x, y \in B \big( n^{\frac{1}{2- \xi_{3}} + 2 \epsilon} \big) } P^{x, y} \big( S^{1} [0, \tau^{1}_{N}] \cap S^{2} [0, \tau^{2}_{N}] = \emptyset \big) \le c \Big( \frac{N}{n^{\frac{1}{2- \xi_{3}} + 2 \epsilon}} \Big)^{-\xi_{3}}.
\end{equation*}
Using this along with \eqref{uruseina-33} and the strong Markov property, we see that
\begin{equation}\label{kietane-33}
P \Big( S^{1} [0, \tau^{1}_{N}] \cap S^{2} [1, \tau^{2}_{N}] = \emptyset, \ M_{ n^{\frac{1}{2- \xi_{3}} + 2 \epsilon} }   \ge n^{ \frac{2- \alpha}{2- \xi_{3}} + 6 \epsilon } \Big)  \le  c N^{-\xi_{3}} e^{- \frac{c_{1}}{2} n^{\frac{\epsilon}{8}}}.
\end{equation}
Theorem 1.3 of \cite{L} gives that $P \Big( S^{1} [0, \tau^{1}_{N}] \cap S^{2} [1, \tau^{2}_{N}] = \emptyset \Big) \asymp N^{-\xi_{3}}$. By dividing both sides of \eqref{kietane-33} by $P \Big( S^{1} [0, \tau^{1}_{N}] \cap S^{2} [1, \tau^{2}_{N}] = \emptyset \Big)$ first and then by letting $N$ go to infinity, we have
\begin{equation*}
\overline{P} \Big( \text{len} \big( LE ( \overline{S} [0, \overline{\tau}^{+}_{ n^{\frac{1}{2- \xi_{3}} + 2 \epsilon} } ] )  \big) \ge n^{\frac{2- \alpha}{2- \xi_{3}} + 6 \epsilon } \Big) \le c e^{- \frac{c_{1}}{2} n^{\frac{\epsilon}{8}}}.
\end{equation*}
By the Borel-Cantelli lemma, we have
\begin{equation}\label{jiro-2-33}
\text{len} \big( LE ( \overline{S} [0, \overline{\tau}^{+}_{ n^{\frac{1}{2- \xi_{3}} + 2 \epsilon} } ] )  \big) \le n^{\frac{2- \alpha}{2- \xi_{3}} + 6 \epsilon } \text{ for large } n, \ \overline{P}\text{-a.s.}
\end{equation}
 Using this and \eqref{jirousan-33}, with probability one, $\overline{T}_{n} \le \overline{\tau}^{+}_{n^{\frac{1}{2- \xi_{3}} + 2 \epsilon }}$ and $\text{len} \big( LE ( \overline{S} [0, \overline{\tau}^{+}_{ n^{\frac{1}{2- \xi_{3}} + 2 \epsilon} } ] )  \big) \le n^{\frac{2- \alpha}{2- \xi_{3}} + 6 \epsilon }$ hold for large $n$. This implies that $\text{len} \big( LE (\overline{S} [0, \overline{T}_{n} ] ) \big)$ is bounded above by $n^{\frac{2- \alpha}{2- \xi_{3}} + 6 \epsilon }$. Since $\epsilon > 0$ is an arbitrary positive number, we finish the proof.
\end{proof}

\subsection{Lower bound for $\alpha_{\ell}(3)$}
In this subsection, we will show that $\alpha_{\ell}(3) \ge \frac{2- \alpha}{2- \xi_{3}}$ by proving that $\text{len} \big( LE (\overline{S} [0, \overline{T}_{n} ] )$ divided by $n^{\frac{2- \alpha}{2- \xi_{3}} - \epsilon }$ goes to infinity for all $\epsilon > 0$ in Proposition \ref{lower-3dim-loop}. The proof is completely same as Proposition \ref{lower-2dim-loop}.


\begin{prop}\label{lower-3dim-loop}
Let $d=2$. For all $b < \frac{2- \alpha}{2- \xi_{3}}$,
\begin{equation}\label{oshiriitai}
\overline{P} \Big( \lim_{n \to \infty} \frac{ \text{len} \big( LE (\overline{S} [0, \overline{T}_{n} ] ) \big) } { n^{b} } = \infty \Big) =1.
\end{equation}
In particular, we have with probability one,
\begin{equation}
\text{len} \big( LE (\overline{S} [0, \overline{T}_{n} ] ) \big) \approx n^\frac{2- \alpha}{2- \xi_{3}},
\end{equation}
as $n \to \infty$ and $\alpha_{\ell}(3) = \frac{2- \alpha}{2- \xi_{3}}$.
\end{prop}

 
\begin{proof}
Since we have proved Proposition \ref{upper-3dim-loop}, it suffices to show \eqref{oshiriitai}.

Fix $\epsilon > 0$. By \eqref{jirousan-33}, we see that $\overline{\tau}^{+}_{n^{\frac{1}{2- \xi_{3}} - 2 \epsilon }} \le \overline{T}_{n} \le \overline{\tau}^{+}_{n^{\frac{1}{2- \xi_{3}} + 2 \epsilon }}$ for large $n$ with probability one. Furthermore, as in \eqref{dareda}, we have
\begin{equation}\label{dareda-33}
\overline{P} \big( n^{\frac{1}{2- \xi_{3}} - 3 \epsilon } < | \overline{S} ( \overline{T}_{n} ) | < n^{\frac{1}{2- \xi_{3}} + 3 \epsilon } \big) \ge 1- c n^{-\frac{\epsilon}{2}}.
\end{equation}

We set 
\begin{equation}\label{kaetta-33}
\overline{t} := \inf \{ k \ | \ LE (\overline{S} [0, \overline{\tau}^{+}_{n^{2}}] ) (k) \in B (n^{\frac{1}{2- \xi_{3}} - 3 \epsilon })^{c} \},
\end{equation}
which is an analog of \eqref{kaetta}. Then same argument as in \eqref{kaerunone} gives that with probability at least $1- c n^{-\frac{\epsilon}{2}}$,
\begin{equation}\label{kaerunone-33}
\text{len} \big( LE (\overline{S} [0, \overline{T}_{n} ] ) \big)  > \overline{t}.
\end{equation}

We next estimate $\overline{t}$ by using Theorem \ref{yattoda}. Let 
\begin{equation}\label{kaetta-1-33}
t := \inf \{ k \ | \ LE ( S^{2} [0, \tau^{2}_{n^{2}}]) (k) \in B (n^{\frac{1}{2- \xi_{3}} - 3 \epsilon })^{c} \}.
\end{equation}
Note that Corollary 4.5 of \cite{Mas} shows that the distribution of $LE ( S^{2} [0, \tau^{2}_{n^{2}}])$ up to time $t$ is same as the distribution of the infinite LERW $\gamma^{\infty}$ up to the first time that $\gamma^{\infty}$ exits from $B (n^{\frac{1}{2- \xi_{3}} - 3 \epsilon })$ up to multiplicative constants. With this in mind, we apply Theorem \ref{yattoda} to show that
\begin{equation*}
P ( t < n^{\frac{2- \alpha}{2- \xi_{3}} - 10 \epsilon} ) \le C e^{-c n^{\epsilon}},
\end{equation*}
which gives that
\begin{equation*}
\overline{P} \big( \overline{t} < n^{\frac{2- \alpha}{2- \xi_{3}} - 10 \epsilon} \big) \le C e^{- \frac{c}{2} n^{\epsilon}}.
\end{equation*}

Thus we can conclude that with probability at least $1- c n^{-\frac{\epsilon}{2}}$, $\text{len} \big( LE (\overline{S} [0, \overline{T}_{n} ] ) \big)$ is bounded below by $n^{\frac{2- \alpha}{2- \xi_{3}} - 10 \epsilon}$. Applying the Borel-Cantelli lemma to the case that $n = 2^{k}$ first, and then by choosing $k$ with $2^{k} \le n < 2^{k+1}$ for a general index $n$, we get the claim.
\end{proof}

\begin{rem}
Similar arguments as in the proof of Proposition \ref{upper-3dim-loop} and Proposition \ref{lower-3dim-loop} show that with probability one,
\begin{equation}
\text{len} LE ( \overline{S} [0, \overline{\tau}^{+}_{n} ] ) \approx n^{2-\alpha} \text{ as } n \to \infty,
\end{equation}
where $\alpha$ is the exponent as in Corollary \ref{subad}.
\end{rem}

\section{Discussion}
In this final section, we will summarize our results and discuss further direction. What we have proved are
\begin{itemize}
\item The law of $\overline{S}$ is invariant under the shift $\overline{\theta}$ and $\overline{\theta}$ is mixing for $d=2, 3$ (Theorem \ref{stationary-o} and Theorem \ref{mixing}).

\item Using Aaronson's results in \cite{A}, several exponents are defined in Theorem \ref{critical exp}. These exponents describe asymptotic behaviors of the length of the loop-erasure, graph distance, and effective resistance of $\overline{S} [0, \overline{T}_{n} ]$.

\item There exists $\alpha \in [\frac{1}{3}, 1)$ such that $E (M_{n} ) \approx n^{2- \alpha}$ for $d=3$, where $M_{n}$ stands for the length of the loop-erasure of $S[0, \tau_{n} ]$ (Theorem \ref{main result-kaetta}, Theorem \ref{leipzig} and Proposition \ref{zhan}).

\item Exponential tail bounds on $M_{n}$ are established for $d=3$ (Theorem \ref{dounaruno} and Theorem \ref{yattoda}).

\item For $d=2, 3$, we have $\text{len} LE (\overline{S} [0, \overline{T}_{n} ] ) \approx n^{\alpha_{\ell} (d) }$ with probability one (Theorem \ref{2-dim} and Proposition \ref{lower-3dim-loop}).

\item Both $M_{n}$ and $\text{len} LE (\overline{S} [0, \overline{\tau}^{+}_{n} ] )$ are of order $n^{\frac{5}{4}}$ in 2 dimensions. Both of them are of order $n^{2-\alpha}$ in 3 dimensions. 
\end{itemize}

Theorem \ref{critical exp} is the only place where we used results of ergodic theory. In an early stage of this project, we tried to prove that three quantities as in Theorem \ref{critical exp}, the length of the loop-erasure, graph distance and effective resistance of $\overline{S} [0, \overline{T}_{n} ]$ are in fact logarithmically asymptotic to $n^{\alpha_{\ell} (d) }$, $n^{\alpha_{g} (d) }$ and $n^{\alpha_{r} (d) }$ respectively, just by using some general results of ergodic theory. However, since we could not find such a way to achieve it, we decided to deal with the length of its loop-erasure by establishing necessary results for the loop-erasure of usual simple random walks. It seems difficult to derive similar results (e.g. moments estimates, existence of the exponent for the first moment, $\cdots$) for the graph distance and effective resistance. But still we conjecture the following:

\begin{con}
For $d=2, 3$, with probability one, we have
\begin{align}
&d_{\overline{S} [0, \overline{T}_{n} ]} (0, \overline{S} ( \overline{T}_{n} ) ) \approx n^{\alpha_{g} (d)}, \\
&R_{\overline{S} [0, \overline{T}_{n} ]} (0, \overline{S} ( \overline{T}_{n} )  ) \approx n^{\alpha_{r} (d)},
\end{align}
as $n \to \infty$.
\end{con}

Since we have proved that $\text{len} LE (\overline{S} [0, \overline{T}_{n} ] ) \approx n^{\alpha_{\ell} (d) }$, it is natural to ask how the distribution of the ratio 
\begin{equation*}
\frac{\text{len} LE (\overline{S} [0, \overline{T}_{n} ] )}{n^{\alpha_{\ell} (d) }}
\end{equation*}
behaves as $n \to \infty$. Recently, some distributional limit theorems for a class of positive, stationary and mixing processes are established in \cite{AZ}. Unfortunately, it has not been proved or disproved that our quantity $\text{len} LE (\overline{S} [0, \overline{T}_{n} ] )$ belongs to the class considered in \cite{AZ}. However we believe that our results derived in the present article are useful to understand the behavior of the ratio and the structure of random walk paths.

\medskip

{\bf Acknowledgement} Enormous thanks go to Takashi Kumagai for constant encouragement and fruitful discussion. Many thanks go to Gady Kozma, Greg Lawler and Alain-Sol Sznitman for useful discussions and comments. 

I would like to thank the referees for giving constructive comments and help in improving the contents of the paper.

Finally, and most importantly, the author thanks Hidemi Aihara for all her understanding, patience and support.This article would not have existed without
her.


\begin{thebibliography}{10}
\bibitem{A} Aaronson, J. An ergodic theorem with large normalising constants. Israel J. Math. 38 (1981), no. 3, 182-188.
\bibitem{AZ} Aaronson, J. and  Zweim\"{u}ller, R. Limit theory for some positive stationary processes with infinite mean. Ann. Inst. H. Poincare Probab. Statist. Volume 50, Number 1 (2014), 256-284.
\bibitem{BM} Barlow, M. T. and Masson, R.  Exponential tail bounds for loop-erased random walk in two dimensions. Ann. Probab. 38 (2010), no. 6, 2379-2417.
\bibitem{BGL}  Benjamini, I., Gurel-Gurevich, O. and Lyons, R. Recurrence of random walk traces. Ann. Probab. 35 (2007). 732-738. MR2308594
\bibitem{BGS} Benjamini, I. Gurel-Gurevich, O. and Schramm, O. Cutpoints and resistance of random walk paths. Annals of Probability 2011, Vol. 39, No. 3, 1122-1136
\bibitem{BSZ}  Bolthausen, E. Sznitman, A-S. Zeitouni, O. Cut points and diffusive random walks in random environment. Annales de l'institut Henri Poincare (B) Probabilites et Statistiques (2003) Volume: 39, Issue: 3, page 527-555
\bibitem{Cro} Croydon, D. A. Random walk on the range of random walk. J. Stat. Phys. 136 (2009) 349-372.
\bibitem{DS} Doyle, G. P. and Snell, J. L. Random Walks and Electric Networks. Mathematical Association of America. (1984).
\bibitem{Dur} Durrett, R. Probability: Theory and Example. Cambridge University Press (2010).
\bibitem{GB} Guttmann, J. and Bursill, R. J. Critical exponents for the loop erased self-avoiding walk by Monte Carlo methods, Journal of Statistical Physics 59:1/2 (1990), 1-9.
\bibitem{JP} James, N.   and Peres, Y.  Cutpoints and exchangeable events for random walks.    Teor. Veroyatnost. i Primenen.  41  (1996),  no. 4, 854--868;  translation in 
Theory Probab. Appl.  41  (1996),  no. 4, 666-677 (1997)
\bibitem{Ken} Kenyon, R.  The asymptotic determinant of the discrete Laplacian. Acta Math. 185 239-286. (2000)
\bibitem{Koz} Kozma, G. The scaling limit of loop-erased random walk in three dimensions. Acta Math. 199: 1 (2007) 29-152. 
\bibitem{lew} Lawler, G. F. A self-avoiding random walk. Duke Math J. 47 655-693. (1980) MR587173
\bibitem{Law b} Lawler, G. F.  Intersections of random walks. Probability and its Applications. Birkhauser Boston, Inc., Boston, MA, (1991). (soft-cover version)
\bibitem{Law2} Lawler, G. F. Escape probabilities for slowly recurrent sets. Probab. Theory Related Fields 94, 91-117 (1992)
\bibitem{Law-3} Lawler, G. F. The logarithmic correction for loop-erased walk in four dimensions. In Proceedings of the Conference in Honor of Jean-Pierre Kahane (Orsay, 1993) J. Fourier Anal. Appl. (1995) 347-361. MR1364896
\bibitem{L} Lawler, G. F. Cut times for simple random walk. Electron. J. Probab. 1 (1996), no. 13, approx. 24 pp. (electronic).
\bibitem{Law3} Lawler, G. F. Hausdorff dimension of cut points for Brownian motion. Electronic Journal of Probability 1 (1996), paper no.2.
\bibitem{Law5} Lawler, G. F. Loop-erased random walk, in Perplexing problems in probability: Festschrift in honor of Harry Kesten (M. Bramson and R. T. Durrett, eds.), Progr. Probab., vol. 44, Birkhauser Boston, Boston, MA, (1999), pp. 197-217.
\bibitem{LSW} Lawler, G. F. Schramm, O.  Werner, W. Values of Brownian intersection exponents. II. Plane exponents. Acta Math. 187 (2001), no. 2, 275-308.
\bibitem{LSW2} Lawler, G. F., Schramm, O.  Werner, W. Conformal invariance of planar loop-erased random walks and uniform spanning trees. Ann. Probab. 32 939-995 (2004)
\bibitem{Law-conf} Lawler, G. F. Conformally Invariant Processes in the Plane. American Mathematical Society. (2008).
\bibitem{Law b2} Lawler, G. F. and Limic, V. Random walk: a modern introduction. Cambridge University Press (2010). 
\bibitem{Greg} Lawler, G. F. and Vermesi, B. Fast convergence to an invariant measure for non-intersecting 3-dimensional Brownian paths. (2010) preprint, available at http://arxiv.org/abs/1008.4830
\bibitem{Lawler} Lawler, G. F. The probability that planar loop-erased random walk uses a given edge. arXiv:1301.5331
\bibitem{Mas}  Masson, R. The growth exponent for planar loop-erased random walk. Electron. J. Probab. 14 (2009), no. 36, 1012-1073.
\bibitem{Pem} Pemantle, R.  Choosing a spanning tree for the integer lattice uniformly. Ann. Probab. 19 1559-1574. (1991). MR1127715
\bibitem{PP}  Moerters, P. Peres, Y. Brownian Motion. Cambridge Series in Statistical and Probabilistic Mathematics. Cambridge University Press. (2010). 
\bibitem{Sch} Schramm, O. Scaling limits of loop-erased random walks and uniform spanning trees. Israel J. Math. 118  221-288. (2000)
\bibitem{S poi} Shiraishi, D. Heat kernel for random walk trace on $\mathbb{Z}^{3}$ and $\mathbb{Z}^{4}$. Annales de l'institut Henri Poincare (B) Probabilites et Statistiques 46(2010):1001-1024
\bibitem{S pt} Shiraishi, D. Exact value of the resistance exponent for four dimensional random walk trace. Probab. Theory Related Fields. Volume 153, Issue 1-2 , pp 191-232
\bibitem{S ejp}  Shiraishi, D. Two-sided random walks conditioned to have no intersections. (2012) Electron. J. Probab. 17, no. 18, 24 pp.
\bibitem{S} Shiraishi, D. Random walk on non-intersecting two-sided random walk trace is subdiffusive in low dimensions. Trans. AMS, to appear.
\bibitem{SS} Sapozhnikov, A. Shiraishi, D. On Brownian motion, simple paths, and loops. available at http://arxiv.org/abs/1512.04864
\bibitem{Wil} Wilson, D. B. Generating random spanning trees more quickly than the cover time. In Proceedings of the Twenty-Eighth Annual ACM Symposium on the Theory of Computing (Philadelphia, PA, 1996) 296-303. ACM, New York. MR1427525
\bibitem{Wil2} Wilson, D. B. The dimension of loop-erased random walk in 3D.  Physical Review E 82(6):062102, (2010). 

\end{thebibliography}
\end{document}